\documentclass[11pt]{amsart} 
\usepackage{graphicx, epsfig}
\usepackage{bbding} 
\usepackage{amssymb,amsfonts,amsmath, amscd, mathrsfs,amsthm,wasysym,pifont,stmaryrd,inputenc,ifthen}
\usepackage{epstopdf,yfonts, setspace, mdwlist, calc, float, comment}
\usepackage{hyperref,wrapfig,color,microtype, enumitem}

\usepackage{mathrsfs}
\usepackage{tikz,tikz-cd}

\newcounter{scomments}

\newcounter{tcomments}

\makeatletter
\let\@wraptoccontribs\wraptoccontribs
\makeatother

\newcounter{cl}
\newcounter{clno}
\newtheorem{theorem}{Theorem}[section]
\newtheorem{question}{Question}

\newtheorem{prop}[theorem]{Proposition}
\newtheorem{example}{Example}

\newtheorem*{conj}{Conjecture}
\newtheorem{cor}[theorem]{Corollary}
\newtheorem{remark}[theorem]{Remark}
\newtheorem{lemma}[theorem]{Lemma}
\newtheorem{defn}[theorem]{Definition}

\newtheorem{thmA}{Theorem}

\newtheorem*{thmscr}{Partial Resolution to Sela's Conjecture}

\newcommand{\Z}{\mathbb Z}
\newcommand{\consttwo}{E}
\newcommand{\op}{\operatorname}

\renewcommand{\~}[1]{\overline{#1}}
\renewcommand{\^}[1]{\hat{#1}}
\renewcommand{\geq}{\geqslant}
\renewcommand{\leq}{\leqslant}
\newcommand{\<}{\left\langle}
\renewcommand{\>}{\right\rangle}
\newcommand{\8}{\infty}

\renewcommand{\:}{\colon}
\renewcommand{\a}{\alpha}
\newcommand{\A}{\mathcal{A}}
\newcommand{\Aut}{\text{Aut}\,}
\renewcommand{\b}{\beta}

\newcommand{\bs}{\backslash}

\newcommand{\cs}[1]{\mathrm{set}_{#1}}
\renewcommand{\Cap}[2]{\underset{#1}{\overset{#2}{\bigcap} }}
\newcommand{\CProd}[2]{\underset{#1}{\overset{#2}{\prod} }}
\renewcommand{\Cup}[2]{\underset{#1}{\overset{#2}{\bigcup} }}

\newcommand{\e}{\epsilon}

\newcommand{\f}{\varphi}
\newcommand{\Fl}{\mathbb{F}}
\newcommand{\fix}{\mathrm{Fix}}
\newcommand{\F}{\mathbb {F}}
\newcommand{\g}{\gamma}
\newcommand{\G}{\Gamma}

\renewcommand{\H}{\mathcal{H}}
\newcommand{\Homeo}{\textnormal{Homeo}}
\newcommand{\hood}{\mathcal{N}}

\newcommand{\Inn}{\text{Inn}\,}
\renewcommand{\int}{\varint}
\newcommand{\Isom}{\mathrm{Isom}\,}
\newcommand{\J}{\mathbf{J}}
\newcommand{\K}{\mathcal{K}}

\renewcommand{\L}{\Lambda}

\newcommand{\Lim}[1]{\underset{#1}{\lim}}

\renewcommand{\max}[1]{\underset{#1}{\mathrm{max}}}

\newcommand{\N}{\mathbb{N}}
\newcommand{\norm}{\trianglelefteqslant}
\renewcommand{\O}{\mathcal{O}}

\newcommand{\Oplus}[2]{\underset{#1}{\overset{#2}{\oplus} }}

\newcommand{\Out}{\text{Out}}
\renewcommand{\P}{\mathbb{P}}

\newcommand{\Prod}[2]{\underset{#1}{\overset{#2}{\prod} }}

\newcommand{\PSL}{\textnormal{PSL}}
\newcommand{\PGL}{\textnormal{PGL}}
\newcommand{\Q}{\mathbb{Q}}

\newcommand{\R}{\mathbb{R}}

\newcommand{\set}{\mathrm{set}}
\newcommand{\SL}{\textnormal{SL}}
\newcommand{\Sqcup}[1]{\underset{#1}{\sqcup}}
\newcommand{\stab}{\mathrm{Stab}}
\newcommand{\Sum}[2]{\underset{#1}{\overset{#2}{\sum} }}
\newcommand{\Sup}[1]{\underset{#1}{\sup}\,}
\newcommand{\Sym}{\mathrm{Sym}}

\newcommand{\T}{\mathcal {T}}

\newcommand{\X}{\mathbb{X}}
\newcommand{\mfN}{N}
\newcommand{\nl}{\mathrm{(NL)}}

\title{The SemiSimple Theory of Higher Rank Acylindricity}
\author{Sahana Balasubramanya, Talia Fern\'os }

\address{Department of Mathematical Sciences, Lafayette College, Pardee Hall, USA}
\email{\url{hassanba@lafayette.edu}}

\address{Department of Mathematics, Vanderbilt University, 1326 Stevenson Center, Nashville, TN, USA}
\email{\url{Talia.Fernos@Vanderbilt.edu}}

\begin{document}

\maketitle 
\begin{center}
\emph{In honor of Bhama Srinivasan}
\end{center}

\begin{abstract}
We present a new notion of non-positively curved groups: the collection of discrete countable groups acting (AU-)acylindrically on finite products of $\delta$-hyperbolic spaces with general type factors. Inspired by the classical theory of ($S$-arithmetic) lattices and the flourishing theory of acylindrically hyperbolic groups, we show that, up to virtual isomorphism, finitely generated groups in this class enjoy a strongly canonical product decomposition. This semi-simple decomposition also descends to the outer-automorphism group, allowing us to give a partial resolution to a recent conjecture of Sela. We also develop various structure results including a free vs abelian Tits Alternative, and connections to lattice envelopes. Along the way we give  representation-theoretic proofs of various results about acylindricity -- some methods are new even in the rank-1 setting.

 The vastness of this class of groups is exhibited by recognizing that it contains, for example, $S$-arithmetic lattices with rank-1 factors, acylindrically hyperbolic groups, HHGs, groups with property (QT), and is closed under direct products,  passing to (totally general type) subgroups, and finite index over-groups. 
\end{abstract}

%\vskip .75in
\begin{figure}
    \centering
    \includegraphics[width=.69\linewidth]{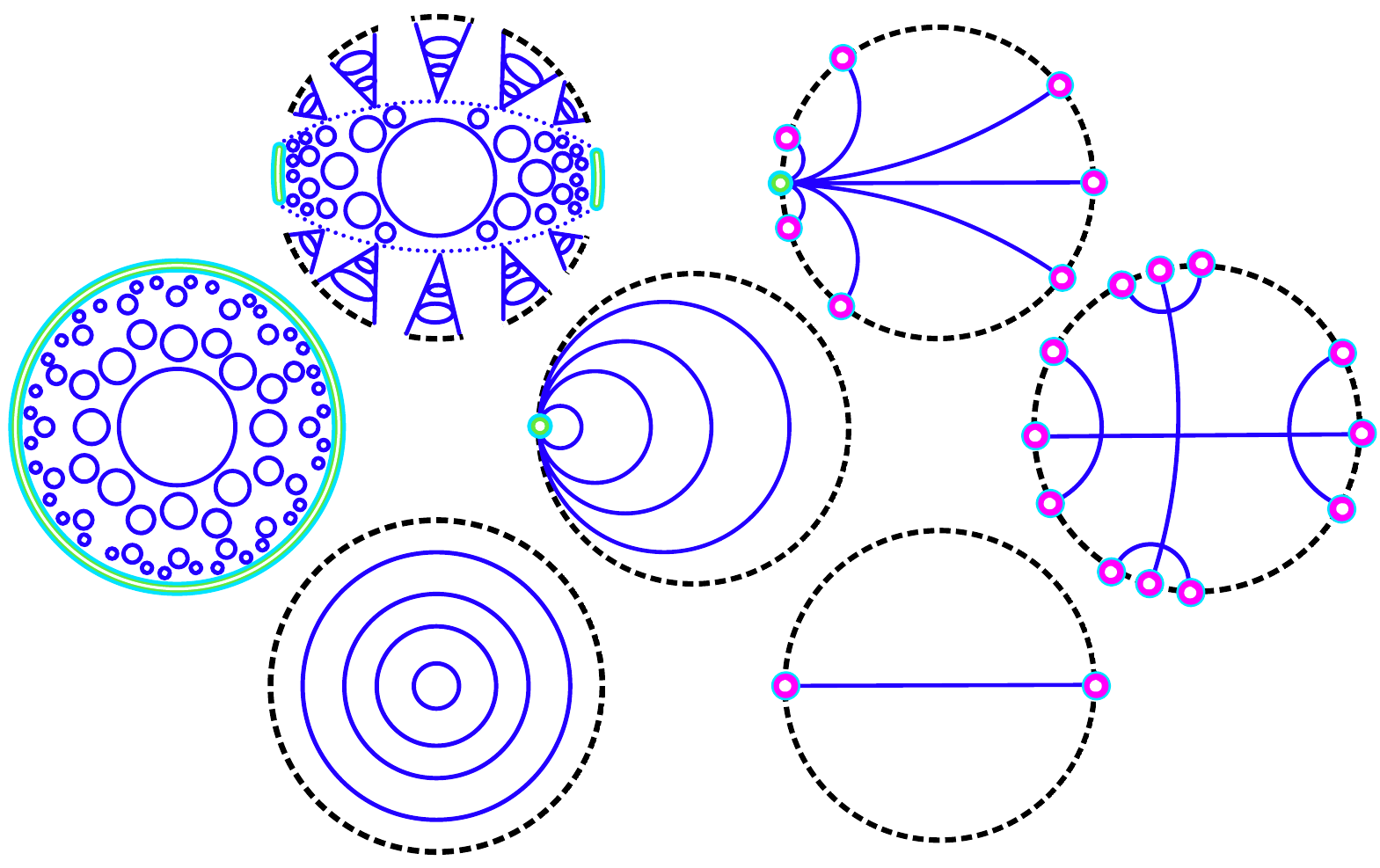}
    
    \emph{Figure showing the classification of actions on a $\delta$-hyperbolic space}
       
    Counter-clockwise from 3 o'clock: general type, quasiparabolic, rift, tremble, rotation, lineal

    Center: parabolic 
\end{figure}

\thispagestyle{empty}

\break
\noindent \textbf{Standing notation: } As this work is quite notation heavy, as a courtesy to the reader, we provide the following table. 

\begin{center}
\begin{table}[H]
\begin{tabular}{|c|c|}
 \hline 
 $\G$ & discrete countable group 
 \\ \hline 
 $K$ & normal subgroup 
 \\
  \hline 
 $H$ & arbitrary (commensurated) subgroup 
 \\
 \hline
   $X$ & complete separable $\delta$-hyperbolic   \\
   or  $X_i$ & geodesic metric space \\
 \hline
  $Y, Z$ & arbitrary metric space 
  \\
\hline
$\hood_\e(\cdot)$, $\~\hood_\e(\cdot)$  & open, closed $\e$-neighborhood   \\
\hline
$\set_\e(\cdot)$ & $\e$-coarse stabilizer
\\ 
 \hline  $R=R(\e), N=N(\e)$ & acylindricity constants, depend on $\e$
 \\\hline
 $\X= \CProd{i=1}{D}X_i$ & product  of $\delta$-hyperbolic spaces 
 \\
 \hline
 D & number of factors in product
 \\\hline 
 $\~X=\partial X\cup X$ & the Gromov bordification
 \\\hline $\~\X = \CProd{i=1}{D} \~X_i$ & bordification of $\X$\\
 \hline 
 $\partial \X = \~\X \setminus \X$ & boundary of $\X$
 \\\hline $\partial_{reg}\X = \CProd{i=1}{D} \partial X_i$ & regular points/boundary\\
 \hline 
 $\Sym_\X(D)$ & permutation group of isometric factors\\
 \hline 
 $\Aut \X =  \CProd{i=1}{D} \Isom X_i \rtimes\Sym_\X(D)$ & automorphism group
 \\
 \hline
  $\Aut_{\!0} \X = \CProd{i=1}{D} \Isom X_i$ & factor preserving automorphism group 
  \\
  \hline $\G_0\leq \G$ & the subgroup that maps to $\Aut_{\!0}\X$ \\
  \hline

\end{tabular}
\caption{Table of standing notation}
\label{table:notation}
\end{table}
\end{center}

\tableofcontents

\subsection*{A Semi-Simple Dictionary:}\label{Sec: SS Dict}
We shall establish the following dictionary between standard linear algebraic group terms and our world of AU-acylindricity in higher rank. See Section \ref{Sect:Theory is SS} for more details. 

\small
\begin{center}
\begin{table}[H]
\begin{tabular}{|c|c|}
 \hline 
  lattice & AU-acylindrical action 
 \\ \hline 
   cocompact lattice & acylindrical action 
 \\ \hline 
   nonuniform lattice & nonuniform acylindrical action 
 \\ \hline 
 semi-simple & AU-acylindrical with general type factors 
 \\ \hline
   parabolic subgroup &  $\Cap{i\in I}{}\stab_\G(\xi_i)$, for some $I\subset \{1, \dots, D\}$ \\
   & and $\xi_i\in \partial X_i$ for $i\in I$ 
   \\ 
\hline 
     Radical & Amenable radical
 \\ \hline 
   Cartan subgroup & $E_\G(\g^\pm)$, where $\g\in \G$ is regular 
   \\\hline 
   Center & Trembling radical (need not be abelian) 
   \\ \hline 
   
\end{tabular}
\caption{Semi-simple Dictionary}
\label{SS dictionary}
\end{table}
\end{center}

\normalsize
\section{Introduction}

The notion of semi-simplicity appears in a variety of contexts within group theory, and more broadly within mathematics. Conceptually, a semi-simple object is one that can be canonically decomposed into an abelian sum of objects within the same class, and where this abelian sum is well behaved with respect to projections and inclusions. Moreover, a simple object is one that cannot be non-trivially further decomposed.  In this work we present a semi-simple theory which  unites within a common framework, a plethora of classes of discrete countable (sometimes finitely generated) groups. 

In 1979 Harvey asked whether mapping class groups were arithmetic lattices \cite{Harvey}. This was answered in the negative by Ivanov \cite{Ivanov1} (see also \cite{Ivanov2}).
Kaimanovich and Masur then showed that higher rank irreducible lattices could not be realized as subgroups of mapping class groups \cite{KaimanovichMasur}. Finally, Farb and Masur showed that in fact irreducible lattices in higher rank semi-simple Lie groups admit only virtually trivial homomorphisms into mapping class groups \cite{FarbMasur}\footnote{The mapping class group of the genus-1 surface is  $\PSL_2\Z$ and of course linear. For genus 2 it is also known to be fact linear \cite{Bigelow}}. Nevertheless the philosophical parallels between mapping class groups and linear groups have been fruitful. In this work, we shall develop a   rigorous common framework to examine both $S$-arithmetic  lattices with rank-1 factors and exotic groups such as the mapping class groups.

Finitely generated mapping class groups are also the prime example of acylindrically hyperbolic groups. Taking inspiration from the work of Sela for groups acting on trees \cite{SelaAcylAcces}, Bowditch introduced the study of acylindrical actions  on general $\delta$-hyperbolic spaces with an eye towards the action of mapping class groups on their $\delta$-hyperbolic curve complexes \cite{MasurMinsky, Bowditch2008}. 

Acylindrical actions subsume (into a non-locally compact universe)  uniformly proper actions, which in turn subsume the actions of cocompact lattices on their ambient spaces (see Lemma ~\ref{Lem:acyl+loc comp implies unif proper}).   Since Osin's  introduction of the class of acylindrically hyperbolic groups (i.e. those that admit a general type acylindrical action on a $\delta$-hyperbolic space) \cite{Acylhyp}, their study and the study of acylindrical cobounded actions has seen an explosion of activity.  The class of acylindrically hyperbolic groups  is vast and includes many groups classically studied in geometric group theory; in addition to mapping class groups, one also has (irreducible) right angled Artin groups,  $\Out(F_n)$, and many Hierarchically hyperbolic groups.  Despite the largeness of the class, acylindricity has proven to be a strong enough condition to produce highly non-trivial results (see \cite{surveyosin} for a survey).

Our work here is the consequence of extending the philosophy that relates acylindricity to cocompact lattices to the higher-rank setting, namely finite products of general-type $\delta$-hyperbolic spaces. We wish to study all lattices, uniform or not and to this end, we allow for nonuniform acylindricity as well. To unite the two we introduce \emph{acylindricity of ambiguous uniformity} (which we refer to as AU-acylindricity, see Definition ~\ref{defn:typesofactions}).  We believe this philosophy has many applications, as we shall see from the work contained in this paper and forthcoming work by the authors \cite{BFApSM, BFCCC}. 

It is worth noting that, while $S$-arithmetic lattices with higher rank factors are known to admit only elliptic or parabolic actions on $\delta$-hyperbolic spaces \cite{Haettel,BaderCapraceFurmanSisto}, they do act AU-acylindrically on the corresponding product of symmetric spaces and buildings, which are non-positively curved, indeed CAT(0). The trivial action on a point is always acylindrical, and every countable group admits a proper (hence AU-acylindrical) parabolic action. For example the construction from Groves--Manning \cite{GrovesManning} can be used immediately to construct such an action for finitely generated groups. For a general countable group we may produce a locally finite graph on which the group acts either by embedding it in a 2-generator group \cite{HigmanNeumannNeumann} or use the proper metric from \cite{DranishnikovSmith} along with the Groves--Manning construction to produce said graph. We then apply the Groves--Manning construction to this locally finite graph to produce the desired hyperbolic space. Therefore, it seems unclear which groups do not admit interesting AU-acylindrical actions on some non-positively curved space. Exotic groups such as finitely generated infinite torsion groups seem to be good candidates, but even then, some are known to act properly on (infinite dimensional) CAT(0) cube complexes. However, coming back to our framework, many of these groups indeed do not act AU-acylindrically on a finite product of $\delta$-hyperbolic spaces as they have property hereditary (NL). It is also for these reasons that we consider the situation where the factor actions are all general type. 

Returning now to the mapping class groups as motivation, we see that the theory of AU-acylindricity in products of $\delta$-hyperbolic spaces should really be seen from a representation-theoretic point of view, meaning that the collection of \emph{all} such actions are interesting (potentially excluding the action on a point, according to taste: see for example \cite[Section 7]{ABO}). There is often a fixation with the notion of a ``best" action.  However, mapping class groups are both acylindrically hyperbolic and have Property (QT) i.e. act by isometries on a finite product of quasi-trees so that the orbit maps are quasi-isometric embeddings. Therefore they act AU-acylindrically on a single hyperbolic space, as well as on such associated products of quasi-trees.  We assert that both of these actions are important and bring to light different and interesting properties. Even the action on a bounded space, though inaccessible from the point of this theory, can be seen to capture various aspects of the group such as the left-regular or quasi-left-regular representations, and are hence also important.

We use the existing results of acylindrically hyperbolic groups, i.e. acylindricity in rank-1, together with the above philosophy to guide our development of the theory of (AU-)acylindricity in higher rank. Consequently, this paper also serves as a starting point for systematically studying (AU-)acylindrical actions on finite products of $\delta$-hyperbolic spaces, a study that necessitated revisiting proofs in rank-1, and a natural trifurcation of elliptic actions. Understanding the nature of the elliptic actions allows us to ``tame" the actions. This turns out to be a key step in establishing semi-simplicity in our higher rank setting -- namely establishing the existence of a strongly canonical product decomposition. In this vein, we also define and establish properties of the \emph{essential core} of a general type action, whose boundary is the limit set of the original action. 

Our semi-simplicity results in turn allows us to partially resolve Sela's conjecture (see Conjecture \ref{Conj: Sela}) concerning ``nice" HHGs. We show that, up to virtual isomorphism, the conjecture holds if and only if it holds for the irreducible factors in the canonical product decomposition. Furthermore, such HHGs are subdirect products of acylindrically hyperbolic groups, and Sela's strong acylindricty condition implies that a group with that property is a  product of hyperbolic groups up to virtual isomorphism. We also establish results pertaining to the theory of lattice envelopes: we show that the groups in our framework always possess Property (CAF) and explore which of these groups possess Property (NbC). Moreover, our groups satisfy a version of the Tits alternative. The proof of this requires an examination of stablizers of \emph{regular} points in the boundary. We also produce some obstructions to acylindriclity in higher rank which mirror known results from the rank-1 case. For instance, we show that not all factor actions can be quasi-parabolic. 

We note that we are not the first to consider variations on acylindricity. For example, there are those introduced by Hamenst\"adt \cite{Hamenstaedt}, Genevois \cite{Genevois}, Delzant \cite{Delzant} and Sela \cite{SelaHR1, SelaHR1Pub, SelaHR2}.  One can verify that our nonuniform acylindricity implies the version studied by Hamenst\"adt. Similarly, Genevois' nonuniform acylindricity is our AU-acylindricity and weak acylindricity is the standard acylindricity condition with $\e=0$.  Further, the weak acylindricity considered by Delzant is a uniform version of the WWPD property introduced by Bestvina-Bromberg-Fujiwara \cite{BestvinaBrombergFujiwara}; this notion of a weakly acylindrical action has also been considered in work by Wan-Yang \cite{WanYang} which studies proper actions on products of hyperbolic spaces with factor actions of this type. The type of acylindricities considered by Sela in his recent works will be explored more deeply in Section ~\ref{sec:ripsandsela}, particularly in the context of his recent conjecture (see Conjecture ~\ref{Conj: Sela}).

\section{Main Theorems}

Our standing notation will be to use $\X$ to denote a finite product of complete, separable, $\delta$-hyperbolic geodesic spaces. Our target group for representations is  $\Aut \X$, which consists of  permutations of isometric factors along with the product of isometries of the factors, which will often be non-elementary -- a list of standing notations used in this paper is provide in Table ~\ref{table:notation} for the convenience of the reader. One may view $\X$ as a non-elementary analogue of a vector space. The following result shows that within this context, the theory is indeed semi-simple. We note that the result is false in the elementary setting. For example, $\Z \times F_2$ does not admit a canonical (internal) product decomposition, though it does have unique factors.  This example shows the need to consider non-elementary factors. For the formal definition of AU-acylindricty, please see Definition \ref{defn:typesofactions}. Informally, AU-acylindricty generalizes the action of a lattice on its ambient space.

\begin{thmA}[Canonical product decomposition]\label{intro:proddecomp}
  Let $\G\to \Aut \X$ be a finitely generated group acting (AU-)acylindrically and with general type factors, and $\A\norm \G$ be the amenable radical (i.e. the maximal normal amenable subgroup). Then $\A$ is finite and there are  finitely many characteristic subgroups of finite index  $\G'\norm \G/\A$ such that:

  \begin{enumerate}
      \item $\G'=\G_1\cdots \G_F$ is a canonical product decomposition into strongly irreducible factors. Each $\G_i$ also admits an (AU-)acylindrical action on a product of factors associated to $\X$. 
      \item   The following are finite index inclusions $$ \Prod{i=1}{F}\Aut(\G_i)\norm \Aut(\G') \;\;\; \text{ and } \;\;\;  \Prod{i=1}{F}\Out(\G_i) \norm \Out(\G').$$
      \item $F$ is unique and $1\leq F\leq D$.
     \item The index $[\G':\G/\A]$ is minimal among all subgroups of finite index satisfying the above properties. 
  \end{enumerate}
\end{thmA}

We note that the results of Theorem \ref{intro:proddecomp} are more subtle than may seem at first glance. Indeed, a canonical product decomposition in this context is in fact strongly canonical meaning in particular that it is the finest product decomposition within its virtual isomorphism class, see Definition \ref{Defn: StrongCanonProductDecomp}. 
Further, the range of $F$ in Theorem ~\ref{intro:proddecomp} is effective, meaning that $F=1$ is achieved when the group is strongly irreducible, and $F=D$ would be the case of a product of $D$-many acylindrically hyperbolic groups.

The term \emph{strongly irreducible} here refers to the group being irreducible, even up to virtual isomorphism (i.e. finite index or taking a finite kernel-quotient). Our notion of strong irreducibility coincides with what Bader, Furman, and Sauer call Property (Irr). However, they also have a different notion of \emph{strong irreducibility}, under which they prove a similar result concerning automorphism groups. By comparison, the latter is a strictly weaker notion. The Bader--Furman--Sauer strong irreducibility for $\G$ means: If $\f: K_1\times K_2 \to \G$ is a homomorphism with finite-index image\footnote{We note that they use the term \emph{cokernel} for the image of the homomorphism, which has a different meaning coming from category theory.} then either $K_1$ or $K_2$ is in the kernel of $\f$. To substantiate that this is a strictly weaker notion, consider $(\SL_2\Z\times \SL_2\Z)/\Delta_2$, where $\Delta_2$ is the diagonal image of $\Z/2$ in the center, which is irreducible,  but has a $\Z/2$-kernel quotient to the reducible $\PSL_2\Z\times \PSL_2\Z$. Their work also proves finiteness results for the outer-automorphism group under reasonable conditions \cite[Corollary 3.25 and Section 3 respectively]{BaderFurmanSauer}. 

The hypothesis that the factor actions are of general type is of course necessary as is witnessed by the example of $\Z^2$, where $\Out(\Z)$ is finite but $\Out(\Z^2)$ is virtually $\SL_2\Z$. (See also \cite{Baumslag}.)

A point of interest to the geometric group theorist may be the fact that $\X$, under mild hypothesis on the factors, has a highly structured quasi-isometry group (e.g. the analogue to our $\Aut \X$ to the QI-setting) \cite{EskinFarb,KleinerLeeb,KapovichKleinerLeeb,BowditchProducts}. While quasi-isometry is a leading type of equivalence within the field, particularly in the context of hyperbolicity, we remark that there are many interesting and important notions that are not invariant under quasi-isometry, for example the rigid Property (T) (see \cite[Theorem 3.6.5]{BdlHV}). Relevant to our work is the fact that cocompact lattices in the same product of ambient groups are necessarily quasi-isometric, independent of whether they are reducible or irreducible. Therefore the finitely presented simple lattices in products of trees   \cite{BurgerMozes, Wise} are quasi-isometric to reducible co-compact lattices in the same, which are virtually products of free groups. 

Sela has recently worked on extending the theory of Makanin-Razborov diagrams and JSJ decompositions to the higher rank setting, which naturally includes examining acylindricity in higher-rank. The work originates in the rank-1 case, which led to some very strong results, including a description of $\Out(\G)$, in particular when $\G$ is a one-ended hyperbolic group (see Theorem ~\ref{Thm:OutHyp1End} for details) \cite{Sela1997, Levitt2005}. This work led to a conjecture about outer automorphism groups of certain colorable HHGs (see Conjecture ~\ref{Conj: Sela}). 

HHGs--which stands for \emph{Hierarchically hyperbolic groups}--are a class of groups first defined by Behrstock--Hagen--Sisto. These groups admit a Cayley graph with an associated axiomatic structure of hyperbolic spaces arranged in a hierarchy (see \cite{HHG17}). Within this hierarchy, there exist maximal domains called \emph{eyries}. We are able to give a partial resolution to Sela's conjecture as follows. (See also the discussion after Remark ~\ref{Rem: Colourable HHG}.)

\begin{thmscr}
    An HHG none of whose eyries are a quasi-line, up to virtual isomorphism satisfies Sela's conjecture if and only if the irreducible factors in its canonical product decomposition satisfy Sela's conjecture. 
\end{thmscr}

We are also able to provide additional structure to groups that satisfy Sela's acylindricity conditions. 
 
\begin{thmA}[Sela's weak acylindricity]\label{intro:selaconj}
Let $\G$ be an HHG none of whose maximal unbounded domains (a.k.a. eyries) are quasi-lines. Up to virtual isomorphism, $\G$ admits a canonical product decomposition. 

 Furthermore, if the action of $\G$ on each eyrie stabilizer is weakly acylindrical then $\G$ is virtually isomorphic to a subdirect product in $D$-many acylindrically hyperbolic groups, where $D$ is the number of eyries, and acts coboundedly on said product. 
\end{thmA} 

A subgroup $\G\leq \Prod{i=1}{D}\G_i$ is a \emph{subdirect product} if it is surjective onto each factor, and \emph{full} if it has non-trivial intersection with each direct factor inside $\Prod{i=1}{D}\G_i$. The hypothesis of Sela's Conjecture of weak-acylindricity applied to the eyries, naturally puts the group in question, up to virtual isomorphism, in the setting of full subdirect products. The reader may wonder to what extent the unused hypotheses of Sela's conjecture (mainly colorability) can be used to deduce that the subdirect product is in fact of finite index so that the canonical product decomposition for $\G$, as is guaranteed by Theorem ~\ref{intro:proddecomp} is the same as those appearing in the subdirect product framework as in Theorem ~\ref{intro:selaconj}. 

This question has been answered for certain types of limit groups and in those cases being of finite index is equivalent to or inferred by the finite dimensionality of $H_n(\G_0, \Q)$ for $n=1, \dots, D$ and for every finite index subgroup $\G_0\leq \G$, in particular if $\G$ is of type $FP_D(\Q)$\cite{BridsonHowieMillerShort, KochloukovaLopezdeGamizZearra}. Limit groups are relatively hyperbolic with respect to their maximal abelian subgroups \cite{Alibegovic, Dahmani}. 

This brings us to back to our guiding philosophy that acylindricity generalizes the notion of a cocompact lattice. While this analogy has been quite useful, these homological considerations create a contrast between cocompact lattices in nonpositive curvature (which have finite dimensional homology groups and in particular are of type $\mathrm{FP}_\8$ \cite[Proposition II.5.13]{BridsonHaefliger}) and their subgroups (that generally do not enjoy such finiteness conditions). Therefore, perhaps insisting that a group be of type $\mathrm{FP}_D(\Q)$ is a reasonable addition to the existence of an acylindrical action of our group on a product of $D$-many $\delta$-hyperbolic spaces.

It turns out that under stronger acylindricity conditions, the structure of $\G$ can be determined even more precisely.

\begin{thmA}[Sela's strong acylindricity]\label{intro:strongacyl} Let $\G\to \Aut \X$ be strongly acylindrical with general type factors. Up to  virtual isomorphism, $\G=\CProd{i=1}{D}{}\G_i$ is the direct product of non-elementary hyperbolic groups, where $D$ is the number of factors of $\X$.  
\end{thmA}

Aligned with the philosophy of viewing acylindricity as a generalisation of co-compact lattices, we examine the lattice envelope classification from the work of Bader--Furman--Sauer in \cite{BaderFurmanSauer}. They  establish a short list of properties so that when a group $\G$ satisfies these conditions, it restricts the class of lattice envelopes $\G$ may admit. We prove that these properties are satisfied by the groups that admit higher rank AU-acylindrical actions, under an additional technical condition we call \emph{(wNbC)-essential-rift-free} factors (see Definition ~\ref{Def: (wNbC)-essential-rift-free} and the discussion within Section ~\ref{subsec:nbc}). As discussed below, a \emph{rift} is a type of elliptic action (see also Theorem ~\ref{thm:actionsclassrk1}).  

In the statement of the following theorem, (CAF) stands for ``Commensurated and Amenable implies Finite" and (NbC) stands for ``Normal by Commensurated". We refer the reader to Definition \ref{defn:BFSprops} for details.

\begin{thmA}[Lattice envelopes]\label{intro:lattice}
    Let $\G\to \Aut \X$ be AU-acylindrical with general type factors. Then $\G$ has Property (CAF), and in particular has finite amenable radical. If in addition, the factors of $\X$ are (wNbC)-essential-rift-free, then $\G$ has Property (NbC). 
    
    Therefore if $\G$ is strongly irreducible and satisfies the above hypotheses, then any non-discrete lattice envelope of $\G$ is as in Theorem ~\ref{BFS}.
\end{thmA}

It is worth noting that linear groups satisfy property (NbC) and those with finite amenable radical also satisfy property (CAF). Linear groups were also the setting for the celebrated Tits alternative from 1972 \cite{Tits}. Since then, there have been many others, including but certainly not limited to results in \cite{BestvinaFeighnHandel,SageevWise, CapraceSageev, FernosFPCCC, 
 Fioravanti2018, OsajdaPrzytycki, GuptaJankiewiczNg, ArzhantsevZaidenberg, LeConte}. It can be thought of as an amenable versus non-amenable dichotomy, and was Tits' response to von Neumann's (false in general) conjecture that non-amenability is equivalent to containing a free group on 2 generators.

Groups that admit a non-elliptic acylindrical action in rank-1 also satisfy a Tits Alternative: they are either virtually isomorphic to  $\Z$ or contain non-elementary free subgroups \cite[Theorem 1.1]{Acylhyp}. We contribute to the list with the following:

\begin{thmA}[Tits alternative]\label{intro:titsalt}
    Let $\G\to \Aut\X$ be acylindrical and not elliptic. Either $\G$ contains a free group of rank 2 or $\G$ is virtually isomorphic to $\Z^k$, where $1\leq k\leq D$.
\end{thmA}

Closely linked to the proof of the above theorem is the structure of point stabilizers on the regular boundary of $\X$, i.e. minimal parabolics in the terms of the semi-simple dictionary. In this vein, we prove the following results. Recall that $\partial_{reg} \X = \Prod{i=1}{D} \partial X_i$ is the set of regular points on the boundary and that $\Aut_{\!0}\X$ is the factor preserving automorphism group of $\X$.

\begin{thmA}\label{intro:stabs}
Suppose that $\xi \in \partial_{reg}  \X$, where $\G \to \Aut_{\!0}\X$ is acylindrical. If $\stab (\xi)$ is infinite then there exists an $\xi' \in \partial_{reg} \X$ such that $\stab_\G(\xi) = \fix_\G\{\xi, \xi'\}$. Moreover, if $D'= |\{i: \xi_i\neq \xi_i'\}|$ then the stabilizers are virtually $\Z^k$ for some $1 \leq k \leq D'$.
\end{thmA}

Another well known result in rank-1 is that pertaining to the classification of isometries and actions (see Section ~\ref{sec:isomrank1}). In this paper, we extend the classification of isometries and actions in rank-1 to consider the different possibilities for elliptic elements and actions. It turns out that these have their own internal trichotomy, manifesting themselves as \emph{trembles, rotations or rifts} (see Definition ~\ref{defn:elementsclassextn}, and the left-most three in the figure on the title page). Of note, the trembles shall in a sense take on the role of the center in a semi-simple lattice. This analogy is further extended to a  semi-simple dictionary in Section ~\ref{Sec: SS Dict}. This approach will be helpful in dealing with results and proofs in the section on lattice envelopes.

 Osin showed that acylindrical actions on hyperbolic spaces are even more restricted -- they are either elliptic, lineal or general type. i.e. horocyclic and quasiparabolic acylindrical actions do not exist (see \cite[Theorem 1]{Acylhyp}). In particular, parabolic isometries are not contained in acylindrical actions, furthering the analogy with cocompact lattices. We extend this result to the higher rank setting by proving the following. For the classification of hyperbolic actions, we refer the reader to Theorem \ref{thm:actionsclassrk1}.

\begin{thmA}[Obstructions to acylindricity in higher rank]\label{intro:acylprob}Let $\G \to \Aut_{\!0}\X$  such that the projections $\G\to \Isom X_i$ are all either elliptic, parabolic or quasiparabolic (with at least one factor being parabolic or quasiparabolic). Then $\G \to \Aut \X$ is not acylindrical.

In particular, if $\G \to \Aut_{\!0}(\X)$ is acylindrical, then every element of $\G$ is either an elliptic isometry or contains a loxodromic factor.
\end{thmA}

As a consequence of this obstruction, and the superrigidity results from \cite{BaderCapraceFurmanSisto}, we are able to deduce that $\PSL_2\Z[\frac{1}{p}]$ (and other nonuniform $S$-arithmetic lattices with rank-1 factors) acts non-uniformly acylindrically on the product of the hyperbolic plane with the $(p+1)$-regular tree, but does not act acylindrically on any $\X$ (see Section ~\ref{Sect: S-arith}).

We also consider the higher rank analogue of loxodromic elements, which are elements that act as loxodromics in each factor; we call these \emph{regular} elements. Maher and Tiozzo used random walks to show that loxodromic elements in a general type action can be encountered asymptotically with high probability  \cite{MaherTiozzo}. We use their result to easily conclude the following. 

\begin{thmA}[Regular elements]\label{intro:regelts} Let $\G\to \Aut \X$ be an action with general type factors. Then there exists $\g\in \G$ that preserves factors and acts as a loxodromic in each factor, i.e. regular elements exist. 
\end{thmA}

It is often desirable to gain control of the space on which a group is acting by removing what can seem like extraneous or possibly uncontrollable pieces of the space, while retaining nice-properties of the action. In our context, what is possibly extraneous and uncontrollable are the parts of the $\delta$-hyperbolic space that are not traversed by loxodromic elements. 
As a first step, we show that $\mathcal{L}_\rho(X)$, a well chosen neighborhood of the union of quasigeodesic axes of the loxodromic elements, is an \emph{essential core} whose boundary is the limit set. Essential means that every $\G$ invariant, quasiconvex subset is coarsely dense. 

\begin{thmA}[The essential core]\label{intro:essentialcore}
Let $\rho:\G\to \Isom X$ be of general type. Then $\mathcal{L}_\rho(X)$ is a quasiconvex, $\G$-invariant subspace, on which  the $\G$-action is essential such that   $$\partial \mathcal{L}_\rho(X) = \L(\rho(\G)).$$

Moreover, setting $\T:=\ker(\G\to \Homeo(\L(\rho(\G))))$ there is a $\delta'$-hyperbolic geodesic space $X'$ equipped with an injective homomorphism $\rho':\G/\T\hookrightarrow \Isom X'$ whose associated action is essential and general type, and a quasi-isometry $\f: \mathcal{L}_\rho(X) \to X'$ that is $\G$-equivariant with respect to the restriction $\rho(\G)|_{\mathcal{L}_\rho(X)}$ and the image $\rho'(\G/\T)$.
\end{thmA}

The notation $\T$ used above is to indicate the term ``trembling radical", which after passing to the essential core is the largest normal elliptic subgroup. It turns out, in the acylindrical higher rank setting, that this essential core retains all the necessary information about the original action, and can be made \emph{tremble-free} as well by utilizing $X'$ as above. This furthers the semi-simple analogy coming from the world of linear groups. The possibility for the theory to be semi-simple stems from the following theorem, in particular the structure of the kernel.

\begin{thmA}[The tremble-free essential core]\label{intro:tremblefree}
 Let $\rho:\G \to \Aut \X$ be (AU-)acylindrical  with general type factors. Then, there exists $\X':=\CProd{i=1}{D}X_i'$ a product of $D$-many  $\delta'$-hyperbolic spaces such that:
    \begin{enumerate}
        \item there is a homomorphism $\rho':\G\to \Aut \X'$ that factors through $\rho$  such that the action is (AU-)acylindrical;
        \item the action on each factor is essential, of general type, and tremble-free;
        \item if $\L_i$  denotes the limit set in the $i^{\text{th}}$ coordinate then the following kernels are finite. $$\ker(\rho') = \ker(\G \to \Homeo(\CProd{i=1}{D}\L_i));$$ 
        \item if a factor of $\X$ is locally compact, then so is the associated factor of $\X'$.
    \end{enumerate}
\end{thmA}

We end this introduction by posing the following questions:
\begin{question}
    Given the vastness of the examples in Section ~\ref{sec:examples}, which groups act in an interesting way on a non-positively curved space and yet do not admit an AU-acylindrical action on any non-positively curved space? 
\end{question}

\begin{question}
    Are there examples of groups acting (AU-)acylindrically on $\X$ which satisfy (wNbC) but do not have (wNbC)-essential-rift-free factors?
\end{question}

\begin{question}\label{Q: subdirect prod}
    Under what conditions is an HHG, as in Sela's Conjecture ~\ref{Conj: Sela}, of finite index as a subdirect product as in Theorem ~\ref{intro:selaconj}? Is having finite dimensional homology $H_n(\G', \Q)$ for each $n=1, \dots, D$, and each $\G'\leq \G$ of finite index sufficient as it is in some cases \cite[Theorem C]{BridsonHowieMillerShort}\cite[Theorem 1.1]{KochloukovaLopezdeGamizZearra}?
\end{question}

\medskip 
\noindent
\textbf{Acknowledgements: } The authors would like to thank Uri Bader, Jennifer Beck, Yves Benoist, Corey Bregman, Montserrat Casals-Ruiz, Indira Chatterji, R\'emi Coulon, Tullia Dymarz, David Fisher, Alex Furman, Daniel Groves, Thomas Haettel, Jingyin Huang, Michael Hull, Kasia Jankiewicz, Fra\c cois Labourie, Marco Linton, Robbie Lyman, Jason Manning, Mahan Mj, Shahar Mozes, Jacob Russell, Michah Sageev, Zlil Sela, Alessandro Sisto, Davide Spriano,  Daniel Woodhouse, Dan Margalit and Abdul Zalloum for useful conversations. We especially thank Joseph Maher for the idea behind Lemma ~\ref{Lem: Convergence}, Denis Osin for providing us with Example ~\ref{Ex:Denis}, Harry Petyt for discussions concerning the content of Remarks ~\ref{Rem: Colourable HHG}   and  ~\ref{Rem:Eyries}, Henry Wilton for asking us a question that led to the work contained in Section ~\ref{sec:ripsandsela}. The second named author would also like to thank  the American Institute of Mathematics, the Fields Institute, the Insitut Henri Poincar\'e, and the NSF for the support through NSF grant DMS–2005640.

\medskip 
\noindent \textbf{Structure of the paper: } Section ~\ref{sec:basics} deals with some basic notation, definitions and results that will be used in this paper. These are divided into 4 subsections, on products, hyperbolicity, the bordification of a hyperbolic space and the Morse Lemma. Sections ~\ref{sec:isomrank1} and ~\ref{sec:isomhigherrank} address the classification of actions and isometries in rank-1 and higher rank respectively; Theorem ~\ref{intro:regelts} is Proposition ~\ref{Prop:Regular Exist}.  Examples and non-examples are discussed in detail in Section ~\ref{sec:examples}. Sections ~\ref{sec:elemsubgrp} establishes the structure of elementary subgroups and the Tits Alternative Theorem ~\ref{intro:titsalt}. Theorem ~\ref{intro:stabs} is also proved therein as Propositions ~\ref{prop:elemsubvirtab} and ~\ref{prop: reg fixed implies reg pair}. Section ~\ref{sec:latticeenvs} discusses connections to lattice envelopes and contains the proof of Theorem ~\ref{intro:lattice}; see Theorem ~\ref{Thm: Acyl -> CAF} and Corollary ~\ref{cor:nbc}. In Section ~\ref{Sect:Theory is SS}, Theorem ~\ref{intro:proddecomp} is Theorem ~\ref{Thm: CanProdDecomp} and we establish a semi-simple dictionary. Lastly, Section ~\ref{sec:ripsandsela} studies recent types of acylindricities introduced by Sela; Theorems ~\ref{intro:selaconj} and ~\ref{intro:strongacyl} are proved as Theorems ~\ref{Thm: Weak Acyl is Product} and ~\ref{Thm: Strong acyl HHG product} respectively. The partial resolution to Sela's conjecture is Corollary ~\ref{Cor: Partial Resol Sela Conj}.

\section{Metric Basics}\label{sec:basics}

We begin with notions that do not require any special properties about the metric space, before progressing to complete separable $\delta$-hyperbolic geodesic spaces and their finite products.

\subsection{General Metric Spaces}
We collect some facts that we will need in the sequel. Since there are no geometric requirements on these results we state them for  general metric spaces.

\subsubsection{The Polish Topology on $\Isom Y$}\label{Sect: top on isom group} Recall that a group is said to be \emph{Polish} if it is separable and admits a  complete metric with respect to which the group operations are continuous maps. Polish groups fit in a hierarchy as in Figure ~\ref{fig:discrete to polish}.

\begin{figure}
    \centering
    \includegraphics[width=0.55\linewidth]{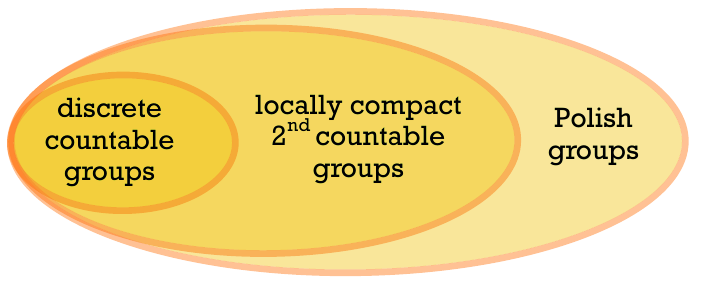}
    \caption{From discrete to Polish}
    \label{fig:discrete to polish}
\end{figure}

If $Y$ is a complete separable metric space then  $\Isom Y$ is a Polish group with respect to the topology of point-wise convergence.

More specifically, if $\{x_n\}\subset Y$ is a countable dense set and $g,h\in \Isom Y$ then the following metric is a complete  metric on $\Isom Y$ that generates the topology of point-wise convergence \cite[Example 9.B(9)]{Kechris}:

\begin{equation}\label{Polish Metric}
    d(g,h) = \Sum{k=0}{\8} \frac{1}{2^{n+1}}\left(\frac{d(g(x_k), h(x_k))}{1+d(g(x_k), h(x_k))} + \frac{d(g^{-1}(x_k), h^{-1}(x_k))}{1+d(g^{-1}(x_k), h^{-1}(x_k))}\right).
\end{equation}

Suppose now that $Y$ is locally compact, separable, connected (which is the case when $Y$ is geodesic), and complete. By \cite{ManoussosStrantzalos} we have that $\Isom Y$ is locally compact with respect to the above metric and acts properly on $Y$. Moreover, if in addition $Y$ is geodesic then every closed metric ball is compact by the Hopf-Rinow Theorem \cite[Proposition I.3.7]{BridsonHaefliger}. It is for these reasons that we will have the standing assumption that our metric spaces are separable and complete.

\subsubsection{Actions and AU-acylindricity}

Let $S$ be a set on which the group $\G$ acts, $A\subset S$. We shall denote by   $\stab_\G(A)$ the set of $\g\in \G$ such that  $\g A=A$, and $\fix(A) = \Cap{a\in A}{}\stab(a)$.

\begin{defn}\label{def: coarse stab}
    For an action $\G\to \Isom Y$, $x\in Y$, and $\e>0$, the $\e$\emph{-coarse stabilizer} of $x$ is $\cs{\e}(x):=\{g\in \G: d(gx,x)\leq \e\}$. If $S\subset Y$ then we denote by $\cs{\e}(S)= \Cap{x\in S}{} \cs{\e}(x)$.
\end{defn}

\begin{defn}\label{defn:typesofactions} An action  $\G\to \Isom Y$ is said to be 
\begin{itemize}
\item \emph{cobounded} if for all $x\in Y$ there is an $R>0$ such that $\G\cdot\hood_R(x)=Y$;
\item \emph{cocompact} if there exists $C\subset Y$ compact such that $\G\cdot C= Y$;
 \item \emph{proper} if for every $\e>0$ and every $x\in Y$ we have $|\cs{\e}(x)|<\8$;
 \item \emph{uniformly proper} if for every $\e>0$ there is an $N>0$ so that for every $x\in Y$ we have $|\cs{\e}(x)|\leq N$;
    \item \emph{acylindrical}, if for every $\e>0$, there exist nonnegative constants $R=R(\e) $ and $ N=N(\e)$ such that for any points $x,y \in X$ with $d(x,y) \geq R$, we have $$| \cs{\e}(x,y) |\leq N;$$ 
    \item \emph{AU-acylindrical}, or \emph{acylindrical of ambiguous uniformity} if for every $\e>0$, there exists  $R\geq 0$ such that for any points $x,y \in Y$ with $d(x,y) \geq R$, we have $| \cs{\e}(x,y)| <\8;$
    \item \emph{nonuniformly acylindrical}, if it is AU-acylindrical but not acylindrical. 
\end{itemize}

\end{defn}

    The reader may think of AU-acylindricity as a  weakening of a proper action on $Y \times Y$ minus a ``thick diagonal". AU-acylindricity may also be seen as a generalization of acting properly or uniformly properly, which may in turn be thought of as generalizing the type of action a lattice enjoys on its ambient space. In this case, acylindricity is the generalization for those that are cocompact, and tautologically, nonuniform acylindricity for those that are not uniform. 

\begin{lemma}\label{Lem:prop cobound is acyl}
    If an action  $\G\to \Isom Y$ is  proper and cobounded then it is acylindrical. 
\end{lemma}

Since elements that map to the identity must fix everything, we immediately have the following: 

\begin{lemma}\label{Lem:AU-acyl has finite kernel}
    If $\theta:\G \to \Isom Y$ is AU-acylindrical and $Y$ is unbounded then the kernel $\ker(\theta)= \{g\in \G: \theta(g) = \mathrm{id}\}$ is finite.
\end{lemma}

We note that if a metric space has cocompact isometry group then it is uniformly locally compact. 

\begin{defn}
    $Y$ is said to be uniformly locally compact if for every $\rho, \tau>0$ there is a $C(\rho, \tau)> 0$ such that for every $x\in Y$ and every open cover of the closed ball $\~\hood_\rho(x)$ by open $\tau$ balls, there is a subcover of cardinality at most $C(\rho, \tau)$.
\end{defn}

We now prove that the notion of AU-acylindricity coincides with the notion of properness for actions on locally compact spaces (though the two notions differ in general for non-locally compact spaces). Figure ~\ref{fig: LCSC to acyl}  summarizes some of these relationships.

\begin{lemma}\label{Lem:acyl+loc comp implies unif proper} Let $Y$ be an unbounded locally compact metric space.  The following are true:
\begin{enumerate}
    \item If $\G\to \Isom Y$ is AU-acylindrical then it is proper. 
    \item If $\G\to \Isom (Y)$ is uniformly proper then it is acylindrical. 
    \item Assume $Y$ is uniformly locally compact and that there is some $T>0$ so that for every  $x\in Y$ and $n\in \N$  there is an $y\in Y$ with $d(x,y)\in[nT, (n+1)T]$. If $\G\to \Isom Y$  is acylindrical then it is uniformly proper.
\end{enumerate}
 \end{lemma}

\begin{proof}  The proof of the second statement follows directly from the definitions.The proof of first statement is similar and simpler to that of the third, which we now prove. Let $\e>0$ and let $N(2\e), R(2\e)$ be the associated acylindricity constants for $2\e$. Fix $x\in Y$. Let $n_{R(2\e)} = \min\{n\in \N: nT\geq R(2\e)\}$. Then, there is a $y\in Y$ with $(n_{R(2\e)}+1)T\geq d(x,y)\geq n_{R(2\e)} T\geq R$. 

Set $\rho=\e +(n_{R(2\e)}+1)T$. Since $Y$ is uniformly locally compact, there exist  $C(\rho, \e)\in \N$ and $z_1, \dots, z_m \in \~\hood_\rho(x)$ with  $m\leq C(\rho,\e )$,  forming a finite open cover $\Cup{i = 1}{m}\hood_\e(z_i)\supset \~\hood_\rho(x)$. 

 Let $g, h\in  \cs{\e}(x)$ we note that $g\in h\cdot\cs{2\e}(x)$ i.e. $d(h^{-1}gx,x) \leq 2\e$.
Furthermore, $gy\in \~\hood_{\rho}(x)$ since

$$d(x, gy) \leq d(x, gx) + d(gx, gy) \leq \e + d(x,y) \leq  \rho.$$
This means that $gy\in \hood_\e(z_i)$ for some $i\in \{1, \dots, m\}$. If also  $hy\in \hood_\e(z_i)$ then $d(h^{-1}gy,y) \leq 2\e$ or equivalently $g\in h\cdot\cs{2\e}(y)$. 

Let $I\subset \{1, \dots, m\}$ be the set of indices such that $j\in I$ if and only if $\hood_\e(z_j) \cap \left(\cs{\e}(x)\cdot y\right)\neq \varnothing$. We have established that $I\neq \varnothing$. For each $j\in I$ fix $g_j$ with $g_jy \in \hood_\e(z_j)$. Then by the above, we have 
 $\cs\e (x)\subset \Cup{j\in I}{}\, g_j\cdot \cs{2\e}(x,y)$ and therefore has cardinality at most $C(\rho, \e)\cdot N(2\e)$.
\end{proof}

\begin{figure}
    \centering
    \includegraphics[width=0.55\linewidth]{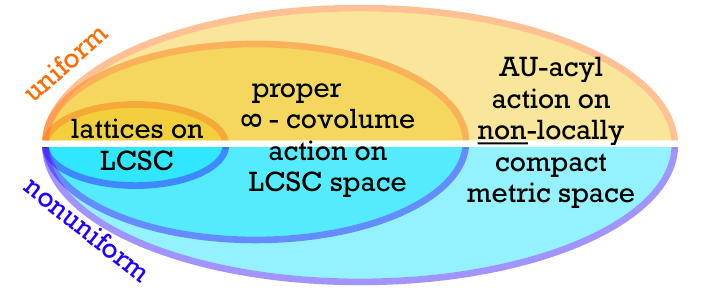}
    \caption{From lattices to AU-acylindrical actions}
    \label{fig: LCSC to acyl}
\end{figure}

\begin{defn}
    Let $C\geq 1$ and $\lambda\geq 0$. A map $\f: Y\to Z$ between metric spaces is a $(C, \lambda)$ \emph{quasi-isometric embedding} if it is coarsely bi-Lipschitz i.e. for all $y,y' \in Y$ we have
    $$\frac{1}{C}d(y,y')-\lambda\leq d(\f(y),\f(y'))\leq Cd(y,y') + \lambda.$$
    If in addition, $\f$ is \emph{coarsely surjective}, i.e. for every $z\in Z$ there is a $x\in Y$ such that $d(\f(x),z)\leq \lambda$ we  say it is a quasi-isometry.

    Moreover, if $Y= \Z$ then the image of $\f$ is called a \emph{quasigeodesic}.
\end{defn}

\subsection{Hyperbolic Geodesic Metric Spaces}

The Gromov product of $x,y\in Y$ with respect to the base-point $o\in Y$ is defined as $$\<x,y\>_o=\frac{1}{2}(d(x,o)+d(o,y)-d(x,y))\geq 0.$$ It is a measure of how far the triple is from achieving the triangle equality.

Assume $Y$ is geodesic and consider $a,b,c\in Y$ with choice of geodesics, as in Figure ~\ref{Fig: Triangle}. The points $s$ and $t$ are marked on the geodesic $[c,a]$ so that $d(c,b) = d(c,t)$ and $d(a,b)= d(a,s)$. The point $m$ is the midpoint between $s$ and $t$. The points $m'$ and $m''$ are then the corresponding points on the geodesics $[a,b]$ and $[b,c]$ respectively so that 

\begin{eqnarray*}
d(a,m)=d(a,m')\\
d(b,m')= d(b,m'')\\
d(c,m'')= d(c,m)
\end{eqnarray*}

This shows that $d(s,m) = d(m,t) = \<c,a\>_b$ and with the previous observations, we see that
$$d(b,m') = d(b,m'') = \<c,a\>_b.$$

Combining these, we obtain the comparison tripod on the right in Figure ~\ref{Fig: Triangle}, which is uniquely determined by  $d(A,M) = d(a,m)$, $d(B,M) = d(b,m')$ and $d(C,M) = d(c,m)$. We note that this yields a natural projection

$$\pi: [a,b]\cup[b,c]\cup[c,a]\to[A,B]\cup[B,C]\cup[C,A]$$

\begin{center}

\end{center}

\begin{figure}
 \includegraphics[width=3.2in]{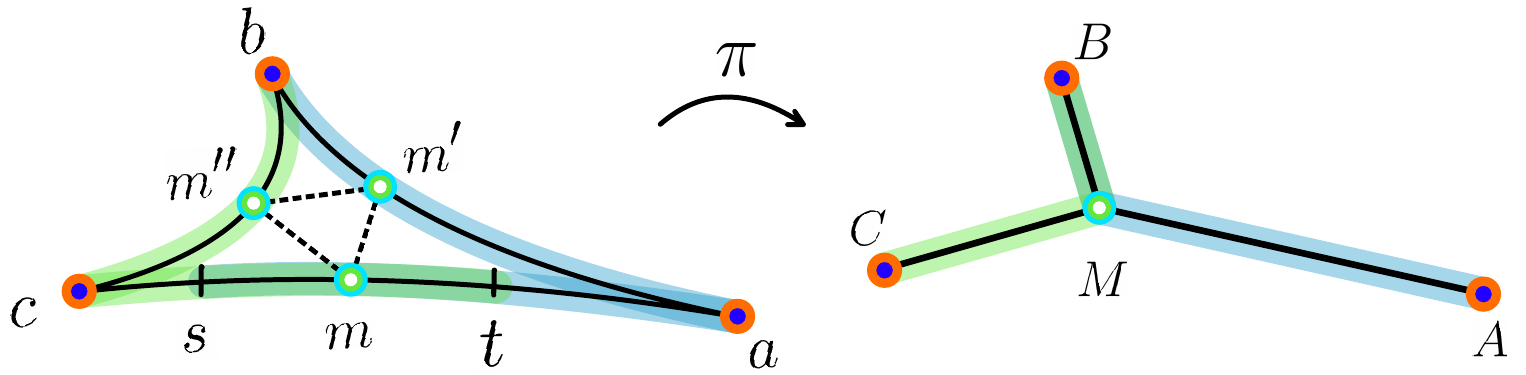}
 \caption{Gromov Product with Comparison Tripod for general triangle}
\label{Fig: Triangle}
\end{figure}

\begin{defn}
A geodesic metric space $X$ is said to be \emph{$\delta$-hyperbolic} if for every $a,b,c\in X$ with comparison tripod as above, the diameter of point-preimages under $\pi$ is uniformly bounded by $\delta$. 
\end{defn}

\begin{remark}\label{Rem:center of triangle}
For a geodesic triangle as in Figure 1, we shall refer to its \emph{center} as the triple $\{m, m', m''\}= \pi^{-1}\{M\}$. We note that if $X$ is $\delta$-hyperbolic then the center of any geodesic triangle has diameter bounded by $\delta$. There are other (equivalent) formulations of $\delta$-hyperbolicity -- such as the so-called ```slim-triangles" condition -- with additional details can be found in \cite[Chapter III.H]{BridsonHaefliger}.
\end{remark}

 \begin{cor}\label{Cor: Gromov Prod vs projection}
     Let $X$ be a $\delta$-hyperbolic geodesic metric space. Fix $a,b,c\in X$ and $[a,c]$ a geodesic between $a$ and $c$. If $p\in [a,c]$ is a point closest to $b$ then $\<a,c\>_b\leq d(p,b) \leq \<a,c\>_b + \delta$. 
 \end{cor}

The following is immediate from the definitions, see Figure ~\ref{Fig: parallelogram}. 

 \begin{lemma}\label{Lem: thin quad} 
 Let $X$ be a $\delta$-hyperbolic geodesic space and $x,y,x',y'\in X$ with $d(x,y) = d(x',y')$. Suppose also that $d(x,x'),d(y,y')\leq L$. If $q,q'$ are geodesics connecting $x$ to $y$ and $x'$ to $y'$ respectively then for all $t\in [0,d(x,y)]$ we have that $d(q(t),q'(t))\leq \max{}\{2\delta, 3L\}$. Furthermore, if $t\in[L, d(x,y)-L]$ then $d(q(t), q'(t))\leq 2\delta$.
\end{lemma}

\begin{figure}
 \includegraphics[width=2.2in]{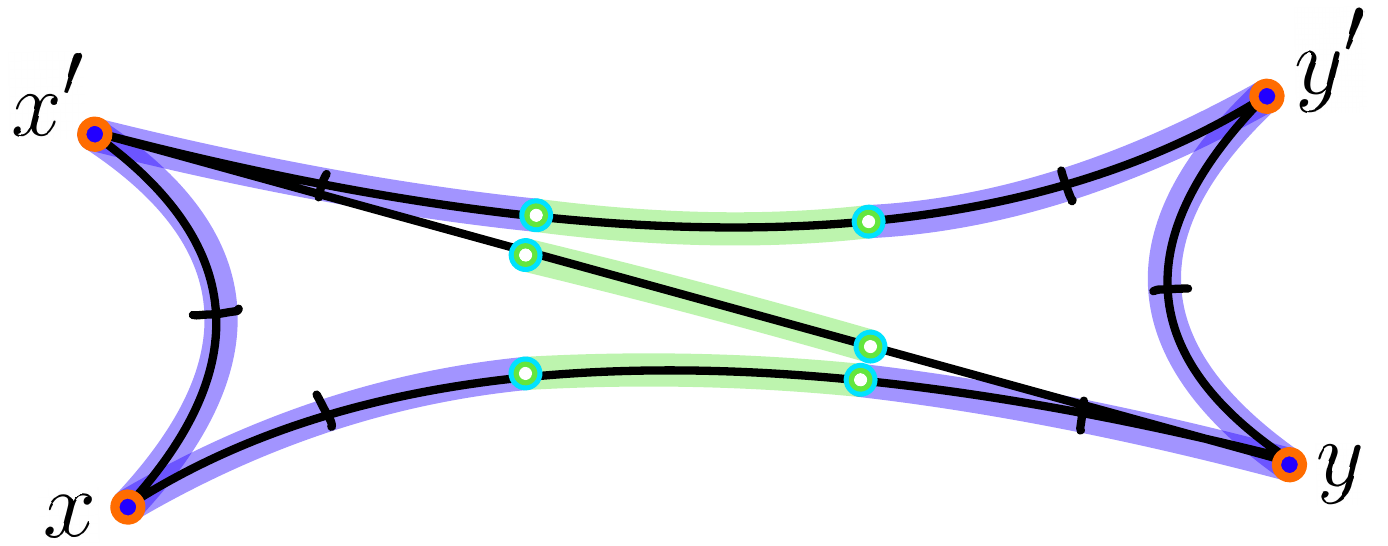}
 \caption{A thin parallelogram as in Lemma ~\ref{Lem: thin quad}}
\label{Fig: parallelogram}
\end{figure}

\subsubsection{Gromov Bordification}

Let $X$ be a $\delta$-hyperbolic geodesic metric space. A sequence $x(n)\in X$, $n\in \N$ is said to be a Gromov sequence if $\Lim{n,m\to\8}\<x(n),x(m)\>_o =\8$.  Recall that if $a(n,m)\in \R$ is a double indexed sequence then $\Lim{n,m\to\8}\, a(n,m) = \8$ if for every $L>0$ there is an $N>0$ so that if $n,m\geq N$ then $a(n,m)\geq L$. Two Gromov sequences $x(n), y(n), n\in \N$ are said to be equivalent if $\Lim{n,m\to\8}\<x(n),y(m)\>_o =\8$. We note that these conditions are independent of the choice of $o\in X$.

The Gromov boundary as a set is the collection of equivalence classes of Gromov sequences converging to infinity and is denoted $\partial X$. If $\xi, \eta\in \partial X$ then we may extend the Gromov product by taking the following infimum ranging over equivalence class representatives:

$$\<\xi,\eta\>_o= \inf \left(\Lim{n,m\to \8}\<x(n),y(m)\>_o\right)$$

\begin{remark}\label{Rem:BenakliKapovich}
By Remark 2.16 of \cite{BenakliKapovich} or \cite[Proposition 5.2]{BonkSchramm}, given $x\in X$, and $\xi, \eta \in \partial X$ there exists a $(1, 10\delta)$ quasigeodesic from $x$ to $\xi$, and a $(1, 20\delta)$ quasigeodesic between $\xi$ and $\eta$. 
\end{remark} 

We topologize $\partial X$ by saying that two classes are ``nearby" if they have ``large" Gromov product, or equivalently the corresponding quasigeodesic rays  fellow travel for a ``long'' time (see Corollary ~\ref{Cor: Boundary close means fellow travel}).

We shall denote by $\~ X = X\cup\partial X$ the Gromov bordification of $X$ and  note that it may not be compact, and $\partial X$ not closed. (It is for this reason that we use the term bordification.) 
We note that isometries of $X$ act by homeomorphisms of $\partial X$ \cite[Theorem 5.3]{Vaisala} (see also \cite{Hamann}).

\subsubsection{The Morse Property}

An important property of hyperbolic metric spaces is the Morse property for stability of quasigeodesics. 

\begin{prop}[Finite Morse Property]\cite[Theorem III.H.1.7]{BridsonHaefliger}\label{lem:qg triangles slim hyp} Let $X$ be a $\delta$-hyperbolic space, and $\lambda \geq 1, C\geq 0$ be two constants. Then there is a constant $M = M(\delta, \lambda, C)$ such that if $q$ is a $(\lambda, C)$-quasigeodesic in $X$ and $[x,y]$ is a geodesic segment between the end points of $q$, then the Hausdorff distance between $q$ and $[x,y]$ is at most $M$.
\end{prop}

The following may be found in \cite[Theorem 6.32]{Vaisala}.

\begin{theorem}[Infinite Morse Lemma]\label{Slim Q-Bigons infinite}
Let $X$ be $\delta$-hyperbolic, and $\lambda \geq 0$. There exists $M= M(\delta,\lambda)$ so that if   $c, q$ are either two $(1, \lambda)$ quasigeodesic rays  starting from the same point, and converging to the same point in $\partial X$ or are two bi-infinite $(1, \lambda)$ quasigeodesics with the same endpoints on $\partial X$, then $d_{Haus}(im(c), im(q))\leq M$. 
\end{theorem}

A consequence of the Morse Lemma is that the ``slim-triangles" formulation of hyperbolicity may be extended to quasigeodesic triangles as well. 

\begin{lemma}\cite[Corollary III.H.1.8]{BridsonHaefliger} A geodesic metric space $X$ is $\delta$-hyperbolic if and only if for every $\lambda \geq 1, C\geq 0$, there is a constant $M = M(\delta, \lambda, C)$ such that every quasigeodesic triangle in $X$ is $M$-slim. i.e. if $[x,y] \cup [y,z] \cup [z,x]$ is a triangle with geodesic sides in $X$, then for any point $p \in [x,y]$, $d(p, [y,z] \cup [z,x]) \leq M$. 
\end{lemma}

\begin{lemma}\label{lemma: fellow travel q-rays}
  Let $c, q: \N\to X$ be $(1,\lambda)$ quasigeodesic rays with $c(0)=q(0)$ and $c(n), q(n) \to \xi \in \partial X$. If $M=M(\lambda)$ is the associated Morse constant then $d(c(n),q(n))\leq 2M + 3\lambda$ for all $n\in \N$. 
\end{lemma}

\begin{proof}
 Fix $n\in \N$. By Theorem ~\ref{Slim Q-Bigons infinite} there is an $m_n\in \N$ so that $d(c(n),q(m_n))\leq M$. Since $c$ and $q$ are quasigeodesics with the same base point, using the reverse triangle inequality we have that
 $$|n-m_n| -2\lambda \leq |d(c(n),c(0)) -d(q(m_n),q(0))| \leq d(c(n),q(m_n))\leq M.$$
 Therefore
\begin{eqnarray*}
d(c(n),q(n))&\leq& d(c(n),q(m_n))+d(q(m_n),q(n))\\
&\leq& M + |m_n-n| +\lambda \\
&\leq& 2M + 3\lambda.
\end{eqnarray*}
\end{proof}

\begin{cor}\label{cor: (1,mu)-fellow travel}
 Let $c, q: \Z\to X$ be $(1,\lambda)$ quasigeodesics  converging to the same points in $\partial X$.  Then there exists an $M'>0$ such that for every $n\in \Z$ we have that $$d(c(n), q(n))\leq 2M' + 3\lambda + 3d(c(0),q(0)).$$
\end{cor}

\begin{proof}
Let $c'(0) = q(0)$ and $c'(n) = c(n)$ if $n>0$. Then $c'$, and $q$ are $(1, \lambda')$ quasigeodesic rays for $\lambda' = \lambda  + d(c(0),q(0))$. Let $M'$ be the Morse constant associated to $\lambda'$. Therefore, by Lemma ~\ref{lemma: fellow travel q-rays} we have that for all $n\in \N$ 
$$d(c'(n), q(n))\leq 2M' + 3\lambda' \leq 
2M' + 3\lambda + 3d(c(0),q(0)).$$
The same argument holds for $n<0$ and therefore for all $n \in \Z$. \end{proof}

\begin{remark}
 We note that Lemma ~\ref{lemma: fellow travel q-rays} is not true for $(C,\lambda)$ quasigeodesics if $C \neq 1$. Indeed, let $c(n) = n\in \R$ and $q(n) = C n$. Then $d(c(n),q(n)) = (C-1) n$ which is unbounded if $\lambda \neq 1$.
\end{remark}
 
 Observe that the proofs of Lemma ~\ref{lemma: fellow travel q-rays} and Corollary ~\ref{cor: (1,mu)-fellow travel} also hold for $(1, \lambda)$ quasigeodesic segments by applying the finite version of the Morse Lemma \cite{BridsonHaefliger}. Furthermore, if we know the quasigeodesics in question are at bounded Hausdorff distance, then hyperbolicity is unnecessary and we deduce the following which holds for arbitrary metric spaces. Recall that we use $Y$ to denote a general metric space.

\begin{cor}\label{cor: Finite Morse}
 Let $c,q: [0,T]\to Y$ be $(1, \lambda)$ quasigeodesics with $c(0) = q(0)$ at Hausdorff distance bounded by $M>0$. Then for all $t\in [0,T]$ we have $d(c(t),q(t))\leq 2M + 3\lambda$. 
\end{cor}

Lastly, we state the following result, which is a direct consequence of Lemma ~\ref{lem:qg triangles slim hyp}  and \cite[Lemma 1.2.3]{BuyaloSchroeder} and may be thought of as a version of the classical result of Gromov that $n$-many points in a hyperbolic space may be $(1, \lambda(n))$-uniformly approximated by a simplicial tree (see also \cite{Kerr}).

\begin{lemma}\label{Lem: N-gons quasitrees}
    Consider a $(1,20\delta)$ quasigeodesic infinite $n$-gon between the points $x_1, \dots, x_n\in \~X$ i.e. $q^i$ are $(1, 20\delta)$ quasigeodesics connecting $x_i$ to $x_{i+1}$ for $i\in \Z/n$. Let $M(1, 20\delta)$ be the associated Morse constant. Then $q^i$ is in the $2M(1, 20\delta)+ \delta\log_2(n)$ neighborhood of $\Cup{j\neq i}{}q^i$. In particular, such infinite $n$-gons are quasitrees.
\end{lemma}

\begin{remark}\label{Rem: QI to QT and Gromov Product}
    The reader shall see that throughout, the use of $(1, \lambda)$ quasi-isometries are  very useful. One key reason for this is that if $Y\to Y'$ is a $(1, \lambda)$ quasi-isometry, denoted by $x\mapsto x'$ then $|\langle x,y \rangle_o- \langle x',y' \rangle_{o'}|\leq 2\lambda$. In particular, together with Lemma ~\ref{Lem: N-gons quasitrees} this will allow us to work with Gromov products of a finite collection of points on a tree. 
\end{remark}

\subsection{Products}\label{Subsec:Products}
Finite products of $\delta$-hyperbolic spaces are of course the star of this show. We shall use consistently the notation that $\X=\CProd{i=1}{D} X_i$, where $X_i$ are geodesic $\delta$-hyperbolic spaces. Without loss of generality, we assume these are also separable so that their isometry groups will be Polish, as in Section ~\ref{Sect: top on isom group}. For a point $x\in \X$ we shall denote the $i$th coordinate as $x_i$. 

Denote by $\mathrm{Sym}_\X(D)$ the collection of possible permutations of indices according to whether the spaces are isometric and fix such an isomorphism. This allows us to consider the corresponding automorphism group $\Aut\X :=\left( \CProd{i = 1}{D} \Isom X_i \right) \rtimes \mathrm{Sym}_\X(D)$ and we will denote by $\Aut_{\!0}\X= \CProd{i = 1}{D} \Isom X_i$, which is the maximal subgroup which preserves the factors. Furthermore, for $i \in \{1, \dots, D\}$ the $i$th coordinate of an element $g\in \Aut_{\!0}\X$ shall be denoted by $g_i$ i.e. this is the image under the projection of $\Aut_{\! 0}\X$ to $\Isom X_i$. 

There are of course different metrics that one could consider on $\X$, the class of $\ell^p$-metrics being an important family of equivalent metrics. The availability of the $\ell^2$-metric allows us to place $\X$ in the category of non-positively curved spaces. However, we  will most often use the $\ell^1$ or $\ell^\infty$ metrics, and while we distinguish these in the following definition, we shall use $d$ for all metrics in the sequel.

\begin{defn}[The $\ell^p$-product metric on $\X$]
 If  $x,y \in \X$ and $p\in [1,\8)$ then $d_p(x,y):=\sqrt[^p]{\Sum{i=1}{D}[d_i(x_i,y_i)]^{^p}}$ and if $p=\8$ then $d_\8(x,y):=\max{i=1,\dots, D}d_i(x_i,y_i).$
\end{defn}

We note that if the factors are geodesic then $\X$ is also geodesic with respect to the $\ell^p$-metric (e.g. consider the product of geodesics in each factor and choose an $\ell^p$-geodesic in the resulting $D$-dimensional interval). Furthermore, $\X$ has finite affine rank. By \cite[Corollary 1.3]{FoertschLytchak} if each $X_i$ is unbounded and not a quasi-line, then the isometry group of $\X$ with respect to the $\ell^1$ metric is $\Aut\X$. 

Finally, using the Roller compactification and boundary of a CAT(0) cube complex as inspiration, we shall consider $\~{\X}:= \CProd{i=1}{D} \~X_i$ and $\partial \X := \~{\X}\setminus \X$. We shall simply refer to $\partial \X$ as the \emph{boundary} of $\X$. Similarly, we shall define the \emph{regular points} of the boundary as the subset $\partial_{reg}\X = \CProd{i=1}{D} \partial X_i$. (See also \cite{KarSageev,FernosFPCCC}.)

\section{Isometric Actions in Rank-1}\label{sec:isomrank1}

In order to classify the isometries in higher rank, we first discuss isometries and actions in rank-1.

\subsection{Classification results.}

We recall some standard terminology used in relation to actions on hyperbolic spaces. However, we will also introduce some new terminology along the way. Recall that given a hyperbolic space $X$, we denote by $\partial X$ its Gromov boundary. 

\begin{defn}\label{defn:elementsclassextn}
 Let $X$ be a $\delta$-hyperbolic space, and $\g\in \Isom X$. For $L\geq 0$, let $\O^L(\g)= \{x\in X: d(\g^n x, x) \leq L, n\in \Z\}$. We shall say that $\g$ is 

 \begin{enumerate}
 \item \emph{elliptic} if all orbits are bounded;
\begin{enumerate}
    \item a \emph{tremble} if it is elliptic and there is an $L$ such that the set $\O^L(\g)=X$
  \item a \emph{rotation} if it is elliptic and $\O^L(\g)$ is bounded for every $L>0$; 
   \item a \emph{rift} if it is elliptic and neither a tremble nor a rotation. 
  \end{enumerate}

  \item \emph{parabolic} if $\langle \g \rangle$ has a unique fixed point in $\partial X$ and has unbounded orbits.

  \item \emph{loxodromic} if $\g$ has exactly two fixed points on $\partial X$ and $\{\g^n x: n\in \Z\}$ is quasi-isometric to $\Z$ for some (equivalently every) $x \in X$. 
 \end{enumerate}
\end{defn}

\begin{example} To distinguish between the types of elliptic elements, consider the standard embedding of the 4-valent tree $T$ in $\R^2$. A rotation by $\pi/2$ would be a rotation on the tree,  and a flip along a horizontal axis would be a rift. On the other hand, trees without leaves have no non-trivial trembles. 

To create a non-trivial tremble, we may take the product of $T$ with a bounded graph. Then any nontrivial isometry of the bounded graph will induce a non-trivial tremble of the $\delta$-hyperbolic space $T\times C$.
\end{example}

 Given an action $\G \to \Isom X$,  we denote by $\Lambda (\G)$ the set of limit points of $\G$ on $\partial X$. That is, $$\Lambda (\G)=\partial X\cap \overline{\G .x},$$ where $\overline{\G .x}$ denotes the closure of a $\G$-orbit in $X\cup \partial X$ for some  choice of basepoint $x\in X$; although, the limit set is independent of this choice of basepoint. The action of $\G$ on $X$ naturally extends to a continuous action of $\G$ on $\partial X$.

The following theorem summarizes the standard classification of groups acting on hyperbolic spaces due to Gromov \cite[Section 8.2]{Gro}, and Hamann \cite{Hamann}. We also add some new terminology to take into account the different types of elliptic isometries. Similar to above for an action $\G\to \Isom X$ and $L>0$ we set $$\O^L(\G)=\{x\in X: d(gx,x)\leq L\text{ for all } g\in \G\}.$$

The reader may have noted the figure on the title page, which has a representation for each of the types of actions in the following theorem. This schematic should be read as one would read figures in the disk model of the hyperbolic plane. We hope that, beyond its aesthetic appeal, it is also instructive \cite{CoulonDorfsman-HopkinsHarrissSkrodzki}.

\begin{theorem}\label{Thm:Class Actions Rank1}
    
\label{thm:actionsclassrk1} Let $X$ be a hyperbolic space.  For an action $\G\to \Isom X$, we have the following:
 \begin{enumerate}
 \item $|\L(\G) |= 0$ so the action is \emph{elliptic}, i.e. $X=\Cup{L\geq 0}{}\O^L(\G)$. If in addition
 \begin{itemize}
     \item there exists $L>0$ such that $\O^L(\G)=X$ then we say the action is a \emph{tremble};
     \item for every $L>0$ we have  that $\O^L(\G)$ is bounded then we say the action is a \emph{rotation};
     \item for all $L$ sufficiently large we have that $\O^L(\G)\neq X$ then we say the action is a \emph{rift}.

 \end{itemize}

  \item $|\L(\G)| = 1$ so the action is \emph{parabolic}. In this case,  $\G$ has a unique fixed point in $\partial X$, unbounded orbits, and   no loxodromic elements;
   \item $|\L(\G)| = 2$ so the action is \emph{lineal}; If further, $\G$ fixes each limit point, the action is called \emph{oriented lineal}.
  
  \item $|\L(\G)| = \8$  and so the action is one of the following:
  \begin{enumerate}
    \item \emph{quasiparabolic} if in addition $\G$ fixes a unique point in $\partial X$;
      \item \emph{general type} if there exists two elements in $\G$ acting as loxodromics with disjoint fixed point sets in $\partial X$. 
  \end{enumerate}
 \end{enumerate}
\end{theorem}

\subsection{Busemann quasimorphism}

A function $b\colon \G \to \mathbb R$ is a \emph{quasimorphism} if there exists a constant $C\geq 0$   such that $$|b(gh)-b(g)-b(h)|\le C$$ for all $g,h\in \G$. We say that $b$ has \emph{defect at most $C$}. If, in addition, the restriction of $b$ to every cyclic subgroup of $G$ is a homomorphism, $b$ is called a homogeneous quasimorphism. Every quasimorphism $b$ gives rise to a homogeneous quasimorphism $\beta$ defined  as follows for $g\in \G$:
$$
\beta(g)=\lim_{n\to \infty} \frac{b(g{^{n}})}n.
$$

We note that since $b(g^n)$ is uniformly close to a subadditive sequence this limit exists by Fekete's Lemma. 
The function $\beta$ is called the \emph{homogenization of $b$.} It is straightforward to check that
$$
 |\beta(g) -b(g)|\le C
$$
for all $g\in \G$ where $C$ is the constant from above.

Given any action of a group on a hyperbolic space fixing a point $\xi$ on the boundary, one can associate the \emph{Busemann quasimorphism}. We briefly recall the construction and necessary properties here, and refer the reader to \cite[Sec. 7.5.D]{Gro} and \cite[Sec. 4.1]{Man} for further details.

\begin{defn}\label{Bpc}
Let $\G\to \Isom X$ be an action that fixes a point $\xi\in \partial X$.
Let $s=\{x(n)\}_{n\in \N}$ be any sequence of points of $X$ converging to $\xi$. Then  the function $b_{s}\colon \G\to \mathbb R$ defined by
$$
b_{s}(g)=\limsup\limits_{n\to \infty}\left(d (gx(0), x(n))-d(x(0), x(n)\right)
$$
is a quasimorphism. Its homogenization $\beta_{s}$ is called the \emph{Busemann quasimorphism}. It is known that this definition is independent of the choice of $s$ (see \cite[Lemma 4.6]{Man}), and thus we can drop the subscript and write $\beta$.\end{defn}

Roughly speaking, the Busemann quasimorphism measures the translation of group elements towards or away from $\xi$. Note that, in general, $\b$ need not be a homomorphism, but it is when the group is amenable or if $X$ is proper (see \cite[Corollary 3.9]{Amen}). Furthermore it is straightforward to verify that $\b(g) \neq 0$ if and only if $g$ is loxodromic. In particular, so $\b \neq 0$ if the action is oriented lineal or quasiparabolic.

We will use the following results concerning the Busemann quasimorphism, which appear in \cite[Section 3]{ABR} in a slightly different setting. We include the setup and statements here -- these, notation included, will be reused throughout, particularly in Section ~\ref{sec:elemsubgrp}. We omit some of the proofs as they follow directly from the proofs in \cite{ABR}.

 Fix $\xi \in \partial X$ and $g\in\fix(\xi)$. Let $M$ be the Morse constant associated to $(1,10\delta)$ quasigeodesic rays in $X$ with one common endpoint in $\partial X$ or $(1,20\delta)$ quasigeodesics with both shared endpoints in $\partial X$.  Let $q: \Z\to X$ be a bi-infinite quasigeodesic converging to $\xi$ in the positive direction. Fix $x = q(0)$.

The ray $q|_{[0,\infty)}$ and its translate $gq|_{[0,\infty)}$ are both $(1,20\delta)$ quasigeodesic rays that share the endpoint $\xi$ and thus  are eventually $M$-Hausdorff close (and synchronously forever after as in Lemma ~\ref{lemma: Unif Bound}). Specifically, there are numbers $t_0=t_0(g)$ and $s_0=s_0(g)\geq 0$ depending on $g$ and $x=q(0)$ so that $q|_{[t_0,\infty)}$ and $gq|_{[s_0,\infty)}$ are $M$-Hausdorff close and $d(q(t_0),gq(s_0))\leq M$.   In other words $s_0$ is a bound for how long we must wait for the ray $gq|_{[0,\infty)}$ to become close to the ray $q|_{[0,\infty)}$. This depends only on $d(x,gx)$, and $s_0(g)$ may be chosen smaller than a function of $d(x,gx)$.  We consider the difference $l=t_0-s_0$ as the amount that $g$ ``shifts'' the quasigeodesic $q$, which may be positive or negative. In the case when the space $X$ is a quasi-line, we may take $s_0 = 0$ since $q, gq$ share both end points on $\partial X$ and are thus always Hausdorff-close. 

\begin{lemma}\cite[Corollary 3.15]{ABR}
\label{Lem:shiftfn}
Under the setup described above, there is a constant $\consttwo>0$ so that for any $g\in \G$, if $s\geq s_0(g)$ then $$d\big(q(s-\b(g)),gq(s)\big)\leq \consttwo.$$ 

In particular, when $X$ is a quasi-line, then $d\big(q(-\b(g)), gx \big)\leq \consttwo$ and it follows that $d(x, gx) \leq |\b(g)| + \consttwo + 20 \delta$. 
\end{lemma}

\begin{prop}\cite[Proposition 3.12]{ABR}\label{prop:smalltranslation}
Let $X$ be a quasi-line, $\xi\in \partial X$ and $\G\to \fix(\xi)$ be an action so that the associated Busemann quasimorphism $\b \colon \G\to \R$ is a  homomorphism. Define the function $A\colon [0,\8)\to \R$ by $$A(r) = \op{sup}\{ |\b(g)| \: g\in \G \text{ s.t. } d(x ,gx) \leq r\}.$$ There exists a constant $B>0$ such that the following holds:  For any $g, h \in \G$ with $d(x,gx)\leq r$ and $\b(h)\leq A(r)$, we have $d(ghx,hx)\leq |\b(g)|+B$.
\end{prop}

We will also need the following lemma pertaining to the behavior of quasigeodesics in parabolic actions. If $\G$ is infinite then  we may therefore fix  $\{1\} \subsetneq S_1 \subsetneq S_2 \subsetneq \dots$ a strictly increasing sequence of finite symmetric subsets of $\G$.

\begin{lemma}\label{lemmaelimpara}  Suppose $\G\to \fix (\xi)$ is parabolic and  $q: \N \to X$ is a $(1, \lambda)$ quasigeodesic ray  converging to $\xi$. Let $E$ be the constant from Lemma ~\ref{Lem:shiftfn}. Define $$M(t) = \op{max} \{ m \mid d( gq(t) , q(t) ) \leq E \hspace{5pt} \forall g \in S_m \}.$$ Then for every $n \in \mathbb{N}$, there is a $t_n$ such that $M(t) \geq n$ for all $t \geq t_n$. In particular, $\displaystyle \lim_{t \rightarrow \infty} M(t) = \infty$.
\end{lemma}

\begin{proof} Since the action is parabolic, there are no loxodromic elements and hence the Busemann quasimorphism $\b \equiv 0$ on $\G$ (and hence $\b$ is a homomorphism).

Since the action is by isometries fixing $\xi$, if $q:\N\to X$ is a $(1, \lambda)$ quasigeodesic converging to $\xi$, then so is $gq$ for every $g \in \G$. Fix $n\in \N$ a let $g_1, g_2\in S_n$. Then  there exists a $t_0 = t_0(g_1, g_2)$ such that $d(g_1q(t), g_2q(t)) \leq E$ for all $t \geq t_0$; this follows from Lemma ~\ref{Lem:shiftfn}. Set $t_n = \op{max} \{t_0(g_1, g_2)\mid g_1 , g_2 \in S_n\}$. It follows that if $t \geq t_n$ then $S_n$ satisfies the definition of $M(t)$ and therefore $M(t) \geq  n $. 
\end{proof}

\subsection{Taming the Space via the Action}\label{Subsec:Taming Space via Action}

In this section we record some results about ``taming"  the spaces on which a group acts. These can be categorized into results about elliptic and loxodromic taming.

\subsubsection{Elliptic Tamings}

The main goal of this subsection is to prove a key result that will be used in Section ~\ref{sec:ripsandsela}. Namely, that the quotient of a complete locally compact geodesic $\delta$-hyperbolic space by a closed subgroup of uniformly bounded isometries (i.e. a closed tremble group) is again a  complete locally compact geodesic $\delta$-hyperbolic space (see Proposition ~\ref{Prop: loc compact tremb quot}).  We start by proving some necessary lemmas.

\begin{defn} Given a metric space $Y$, a subset $Z$ is \emph{quasiconvex} if there is a constant $\sigma \geq 0$ such that any geodesic in $Y$ with end points in $Z$ is contained in the $\sigma$-neighborhood of $Z$. 
\end{defn}

\begin{lemma}\label{Lem: elliptic orbits are quasiconvex}
    Let $\G\to \Isom X$ be elliptic. Let $L \geq 2\delta$. The set $\O^L(\G)$ is quasiconvex. 
\end{lemma}

\begin{proof}
     Let $g\in \G$, and $x,y\in \O^L(\G)$ with $d(x,y)\geq 2L$. Fix  a geodesic $\a: [0,d(x,y)]\to X$ from $x$ to $y$. By Lemma ~\ref{Lem: thin quad} we have that for every $t\in [L,d(x,y)-L]$, $d(g\a(t),\a(t)) \leq 2\delta$ and hence $\a(t)\in \O^L(\G)$. The rest of the points on the geodesic are of course within $L$ of $x$ or $y$, and hence the set is quasiconvex. 
\end{proof}

We now prove the following main result about elliptic tamings. We shall use $\<\<\T\>\>$ to denote the normalizer of $\T$ in $\Isom X$

\begin{prop}\label{Prop: loc compact tremb quot} Let $X$ be a  locally compact, complete, hyperbolic geodesic space and  $\T \leq \Isom(X)$ a compact subgroup  with respect to the topology of uniform convergence on compact sets acting as a tremble. Then $\T \setminus X$ -- the quotient of the space $X$ by orbits of $\T$ -- is a locally compact, complete, $\delta$-hyperbolic space. Moreover,  the natural map $X\to X'$ is a continuous quasi-isometry that is $\<\<\T\>\>$-equivariant.
\end{prop}

\begin{proof} We must  endow $X': = \T \setminus X$ with a metric so that it is locally compact, complete, geodesic, and $\delta$-hyperbolic. We do so in a series of claims. By $x'$ we denote the image of a point $x \in X$, i.e. $x'=\T x$. Since $\T$ acts as a tremble, let $B>0$ be a constant such that the orbits of $\T$ have diameter at most $B$.

\noindent
\textbf{Claim \arabic{clno}: $X'$ is a metric space with respect to $$d'(x',y') = \inf\{ d(gx, hy) \mid g,h \in \T . \}.$$}\addtocounter{clno}{1}

We note that since $\T$ is compact and $X$ is complete, that the infimum is in fact a minimum that is achieved. Obviously $d'$ takes on non-negative values, is well-defined and symmetric. Further $ d'( x',y') \leq d(x,y)$. It is also easy to verify that $x'=y'$ if and only if $d'(x', y') =0$. Therefore, the claim will follow if we prove the triangle inequality. To this end, let $x',y',z' \in X'$. Up to replacing $x,y, z$ with different $\T$-orbit representatives, since the minimum is achieved we may assume $d'(x',z')= d(x,z)$, $ d'(x',y') = d(x,y)$ and $d(gy,z) = d'(y',z')$, for some $g \in \T$.  

Now $d'(x',y') + d'(y',z') = d(x,y) + d(gy, z) = d(x,y) + d(y, g^{-1}z)$. So by the triangle inequality and the fact that $d'(x',z')$ is an infimum we have
$$d'(x',z')\leq d(x, g^{-1}z) \leq d(x,y) + d(y, g^{-1}z) = d'(x',y') + d'(y',z').$$ %%%%%%%%%%%%%%%%%%%%%%%%%%%%%%%%%%%%%%%%%
\noindent
\textbf{Claim \arabic{clno}: $(X',d')$ is Hausdorff and locally compact}.\addtocounter{clno}{1} 
Note that as $\T$ is compact, the equivalence classes $\T x$ are compact and hence closed, making $X'$ Hausdorff. Moreover,  the natural map $(X,d)\to (X',d')$ is clearly onto and continuous: the inverse image of $\hood_\e(x')$ is $\Cup{g\in \T}{}\,\hood_\e(gx)$, which is open.  Since $X$ is complete, locally compact by the Hopf-Rinow Theorem \cite[Proposition I.3.7]{BridsonHaefliger}, closed bounded subsets are compact. In particular, $\~\hood_{B+\e}(x)$ is compact and contains $\Cup{g\in \T}{}\,\hood_\e(gx)$. Therefore, the image of $\~\hood_{B+\e}(x)$ in $X'$ is a compact set containing $\hood_\e(x')$, and so $X'$ is locally compact.

%%%%%%%%%%%%%%%%%%%%%%%%%%%%%%%%%%%%%%%%%%
\noindent
\textbf{Claim \arabic{clno}: $(X', d')$ is complete.}\addtocounter{clno}{1}
Let $x'_n\in X'$ be a $d'$-Cauchy sequence. For each $n\in \N$ choose a lift $x_n\in x'_n$. Since $\{x'_n\}$ is Cauchy, it is bounded, so there is an $R>0$ such that $ d'(x'_n, x'_1)\leq R$ for all $n$. Therefore, $d(x_1, x_m) \leq R+2B$. Since the closed ball in $X$ of radius $R+2B$ is compact, there is a convergent subsequence, i.e. there exists $n_k$ and $y\in X$ such that $x_{n_k}\to y$. Since $d'(y',x'_{n_k})\leq d(y,x_{n_k})$ we deduce that $x'_{n_k} \to y'$. Since $x'_n$ is Cauchy, we deduce that $x'_n \to y'$, which proves the claim.

%%%%%%%%%%%%%%%%%%%%%%%%%%%%%%%%%%%%%%%%%%
\noindent
\textbf{Claim \arabic{clno}: $(X', \widehat d)$ is a length space.}\addtocounter{clno}{1} For $x',y'\in X'$ let $\widehat d(x',y')$ be the infimum of the lengths of rectifiable paths connecting $x'$ and $ y'$.  By \cite[Proposition I.3.2]{BridsonHaefliger} this makes $(X', \widehat d)$ into  a length space.

%%%%%%%%%%%%%%%%%%%%%%%%%%%%%%%%%%%%%%%%%%
\noindent
\textbf{Claim \arabic{clno}: The metrics $\widehat{d}, d'$ are equivalent. Consequently, $X'$ is complete and locally compact with respect to the metric $\widehat{d}.$} \addtocounter{clno}{1} Let $\e>0$ be given. We use the symbol $\approx_{\,\e}$ to denote that two quantities have an absolute difference of at most $\e$. i.e. the two numbers are within $\e$ 
of each other. 

By \cite[Proposition I.3.2]{BridsonHaefliger} $\widehat{d}\leq d'$ and so for every $r>0$, and $x'\in X'$ we have that $\hood_r^{d'}(x')\subseteq \hood_r^{\widehat{d}}(x')$. We now show that $\hood_r^{\widehat{d}}(x')\subseteq \hood_{r+2\e+2B}^{d'}(x')$.

By the definition of $\widehat{d}$ there exists a rectifiable path $\eta$ between $x',y'$ such that $\widehat{d}(x',y') \leq \ell(\eta) \leq \widehat{d}(x',y') + \e/2$. As $\eta$ is rectifiable, there is an $N >0$ and $x'=x'_0, \dots, x'_{N+1}=y'\in \eta$ such that $\displaystyle \left| \sum_{k =0} ^N d'(x'_k, x'_{k+1}) - \ell(\eta) \right| \leq \e/2$. Combining these inequalities we deduce that
 $$\widehat{d}(x',y')\approx_{\,\e}\sum_{k =0} ^N d'(x'_k, x'_{k+1}).$$
 
For each $k\in \{0, \dots N\}$, we may choose $y_k \in x'_k, z_{k+1} \in x'_{k+1}$ such that $$d(y_k, z_{k+1}) =d'(x'_k, x'_{k+1}).$$

Note that $y_k, z_{k+1} \in X$ and since $X$ is a geodesic metric space, we fix a geodesic $[y_k, z_{k+1}]$. Utilizing these, define a piece-wise geodesic path $\zeta$ in $X$ between $y_0$ and $z_{N+1}$ whose length is the same approximation of $\ell(\eta)$ as above. Take $y_0$ and $z_1$ and $\zeta_0 = [y_0, z_1]$ in $X$. As $y_1$ and $z_1$ are in the same orbit, there exists a $g_1 \in \T$ such that $g_1y_1 = z_1$. Take the translate $\zeta_1 = g_1[y_1, z_2]$ of the geodesic $[y_1, z_2]$ and consider the concatenation $\zeta_0 \cup \zeta_1$. Again, as $y_2, z_2$ are in the same orbit, there exists a $g_2 \in \T$ such that $g_2y_2 = g_1z_2$. Take the translate $\zeta_2 = g_2[y_2, z_3]$ and consider the concatenation $\zeta_0 \cup \zeta_1 \cup \zeta_2$. Continuing this process we obtain a path $\zeta = \zeta_0 \cup \zeta_1 \cup \cdots \zeta_{N}$ between $y_0 \in x'$ and some $z_{N+1} \in y'$.

Observe that by construction, the length of the path $\zeta$ satisfies $$\widehat{d}(x',y')\approx_{\,\e}\ell(\zeta)\approx_{\,\e}d(y_0, z_{N+1}).$$

Choose $x \in x'$ and $y \in y'$ so that $d'(x',y') = d(x,y)$. Then since $y_0 \in x'$, it follows that $d(y_0, x) \leq B$. Similarly, $d(y, z_{N+1}) \leq B$. It follows then that $$\widehat{d}(x',y')\approx_{\,\e}\ell(\zeta)\approx_{\,\e}d(y_0, z_{N+1}) \approx_{\,2B} d(x,y) =  d'(x',y').$$

%%%%%%%%%%%%%%%%%%%%%%%%%%%%%%%%%%%%%%%%%%
\noindent
\textbf{Claim \arabic{clno}: $X'$ is a geodesic space.}\addtocounter{clno}{1} Again by the Hopf-Rinow Theorem, as $X'$ is complete, locally compact and a length space, $X'$ is proper and a geodesic metric space with respect to $\widehat{d}$ (and consequently $d'$ also). 

%%%%%%%%%%%%%%%%%%%%%%%%%%%%%%%%%%%%%%%%%%
\noindent
\textbf{Claim \arabic{clno}: $X'$ is $(1,B)$ quasi-isometric to $X$ and hence hyperbolic.} \addtocounter{clno}{1} We will prove the claim by showing that the natural projection $X\to X'$ is a quasi-isometry. To this end, observe that since $\T$ is compact, $\T$ is bounded in $\Isom X$ with respect to the metric from Section ~\ref{Sect: top on isom group}, and hence has uniformly bounded orbits of diameter at most $B$. As above the map $X\to X'$ is a contraction, i.e. $d'(x',y')\leq d(x,y)$ for all $x,y\in X$.

Furthermore, if $x,y\in X$, then choose $a\in x'=\T x, b\in y'=\T y$ such that $d(a,b)= d'(x',y')$. Then $d(x,y)\leq 2B + d(a,b)= 2B +d'(x',y').$ 

Consequently, the projection $X\to X'$, $x\mapsto x'$ as above is clearly a continuous quasi-isometry. Moreover, the normalizer $\<\<\T\>\>$ preserves the $\T$ orbits and we deduce that the above maps are $\<\<\T\>\>$-equivariant. This completes the proof of the proposition. \end{proof}

\begin{defn} An action of a group $\G$ on a hyperbolic space $X$ is called \emph{essential} if every $\G$-invariant quasiconvex subset is coarsely dense.
\end{defn}

The notion of essentiality is very close to that of being coarsely minimal in the sense of \cite[Definition 3.4]{BaderCapraceFurmanSisto}. The main difference is that coarse minimiality restricts itself to non-parabolic actions on unbounded hyperbolic spaces, whereas the trivial action on a point is essential. 

\begin{lemma}\cite[Lemma 4.10]{PropNL}\label{Lem: Kill Normal Ell}
    Suppose that $K\norm \G$ and $\G\to \Isom X$ is such that $K$ is elliptic. Then there is a $\G$-invariant quasiconvex $Y\subset X$, a hyperbolic graph $X'$, a map $\G/K \to \Isom X'$, and a quasi-isometry $\f: Y\to X'$ that is $\G$-equivariant. 
\end{lemma}

We now prove a result that will be crucially used in several places in this paper, before proceeding to the statement and proofs of two corollaries on taming actions - these represent the same result for the locally compact (complete), and non-locally compact spaces.

\begin{defn}
    A subgroup $H \leq \G$ is said to be \emph{commensurated} in $\G$ (or simply commensurated if the ambient group is clear) if for every $g\in\G$ we have that $H\cap gHg^{-1}$ is of finite index in both $H$ and $gHg^{-1}$. 
\end{defn} 

\begin{lemma}\label{Lem: normal in total gen type is tremb/gen type}
 Let $\G\to \Isom X$ be essential and of general type. If  $H \leq \G$ is commensurated then the action of $H$  is either elliptic or of general type. Furthermore, if $H$ is normal in $\G$ and acts elliptically, then the action of $H$ is a tremble. 
\end{lemma}

\begin{proof} We will employ the classification of isometric actions from Theorem ~\ref{thm:actionsclassrk1}. Assume there exists an $h\in H$ acting as a loxodromic. Since $\G$ is of general type, there is a loxodromic $g\in \G$ whose fixed points are disjoint from those of $h$. In particular, there is a power $k \geq0$ such that $g^khg^{-k}$ is independent from $h$. Note that $g^khg^{-k}$ is a loxodromic as it is a conjugate of $h$. Since $H$ is commensurated, $H\cap g^kHg^{-k}$ is finite index in $H$ and $g^kHg^{-k}$ and therefore, $h$ and $g^khg^{-k}$ have powers that are in $H \cap g^kHg^{-k}$ and thus in $H$, and so, $H$ contains independent loxodromics. 

Suppose now that $H$ does not contain any loxodromics, then the action of $H$ is either parabolic or elliptic.

Suppose $H$ fixes a point $\xi\in\partial X$. Since $\G$ contains independent loxodromics, it contains a loxodromic $g\in \G$ whose end points are distinct from $\xi$. Then the finite index subgroup $H \cap gHg^{-1}$ fixes both $\xi$ and $g\xi$ which are distinct. So the action of $H$ is not parabolic and hence must be elliptic. 

Lastly, assume $H$ is normal and acting elliptically. Then,  there is an $L>0$ so that $\O^L(H):= \{x\in X: d(hx,x)\leq L \text{ for all } h\in H \}\neq \varnothing$. By Lemma ~\ref{Lem: elliptic orbits are quasiconvex} this set is quasiconvex. Since $H$ is normal, $\O^L(H)$ is $\G$-invariant. By Lemma ~\ref{Lem: irred core is essential invariant}, the action of $\G$ is essential.  Thus we have that $\O^L(H)$ is coarsely dense in $X$ and hence $H$ is a tremble. 
\end{proof}

\begin{cor}\label{Cor:LC Tremble Tame}
  Assume that $X$ is complete and locally compact and that the action of $\Isom X$ is essential and of general type. Let $\T(X)$ denote the kernel of the homomorphism $\Isom X\to \Homeo (\partial X)$. Then $\T(X)$ is elliptic and compact and the quotient $X' = \T(X)\setminus X$ is a complete, locally compact $\delta'$-hyperbolic geodesic space. Furthermore, the natural map $X\to X'$ is a continuous  $\Isom X$-equivariant subgroup-action-type preserving quasi-isometry and yields the following commutative diagram. 

\begin{center}
\begin{tikzcd}
  \Isom X \arrow[r] \arrow[dr]
    & (\Isom X)/\T(X) \arrow[d]\\
& \Isom X'= \Isom(\T(X)\bs X) \end{tikzcd}
    
\end{center}
\end{cor}

\begin{proof} We begin by observing that  $\T(X)$ can not contain a loxodromic element as these act non-trivially on the boundary and so it is a tremble by Lemma ~\ref{Lem: normal in total gen type is tremb/gen type}  (which justifies its notation). This means that $\T(X)$ is bounded in $\Isom X$ with respect to the metric from Section ~\ref{Sect: top on isom group}. Moreover, $\T(X)$ is also closed:  Indeed, since $\Isom X$ is locally compact if $g_n\in \T(X)$ is a sequence converging to $g\in \Isom X$ then as the $g_n$  have uniformly bounded orbits, with bound uniform over $n$, it follows that $g$ has uniformly bounded orbits as well, i.e. $g\in \T(X)$. Since $X$ is locally compact and complete, then $\T(X)$ is compact by the Hopf-Rinow Theorem. Therefore, $X'$ is a complete, locally compact $\delta'$-hyperbolic geodesic space where the natural map $X\to X'$ is a continuous  $\Isom X$-equivariant quasi-isometry by Proposition ~\ref{Prop: loc compact tremb quot}.

To prove that we have the commutative diagram, consider the natural map as above $X\to X'$. Let $g\in \Isom X$, $x',y'\in X'$ then $gx':= g.\T(X) x= \T(X)g.x$ by normality. Therefore, the action of $\Isom X$ on $X'$ by permutations factors through $(\Isom X)/ \T(X)$. Moreover, the action is in fact by isometries: Let $g\in \Isom X$, $x',y'\in X'$ and fix $x,y\in X$ so that $d(x,y) = d'(x',y')$. Then
$d'(x',y') = d(x,y) = d(gx,gy)\geq d'(gx',gy')$.

To show the other inequality, let $a\in g\T(X)x$, $b\in g\T(X)y$ such that $d'(gx',gy')=d(a,b)$. Then $a=gtx$ and $b=gsy$ for some $t,s\in \T(X)$ and so $d(a,b) = d(gtx, gsy)= d(tx,sy)\geq d'(x',y')$ by definition of $d'$ as the minimum over orbit representatives. 

To observe that the map $\Isom X\to \Isom X'$ preserves action types, recall that by Claim 7 of Proposition ~\ref{Prop: loc compact tremb quot} the spaces $X$ and $X'$ are $\Isom X$-equivariantly quasi-isometric. Therefore, the type of individual isometries is preserved as is the action type of subgroups.
\end{proof}

\begin{cor}\label{Cor:Non-LC tremble tame}
  Assume that the action of $\Isom X$ is essential and of general type. Let $\T(X)$ denote the kernel of the homomorphism $\Isom X\to \Homeo (\partial X)$. Then there is a $\delta'$-hyperbolic graph $X'$, a homomorphism  $\Isom X\to \Isom X'$ and a quasi-isometry $X\to X'$ that is $\Isom X$-equivariant such that the action type of subgroups of $\Isom X$ is preserved and $\T(X)$ acts trivially on $X'$ and $\partial X'$.
\end{cor}

\begin{proof}
    As in the proof of Corollary ~\ref{Cor:LC Tremble Tame} since the action is essential and of general type, $\T(X)$ has bounded orbits. Therefore, as in Lemma ~\ref{Lem: Kill Normal Ell}, there is a hyperbolic graph $X'$ on which $\Isom X$ acts with the action factoring through $\Isom X/\T(X)$, and an $\Isom X$-equivariant quasi-isometry to $X\to X'$. Therefore there is an equivariant homeomorphism $\partial X\to \partial X'$.
   The fact that action types of subgroups are preserved follows as in the previous corollary.
\end{proof}

\subsubsection{The Essential Core}

Loxodromics play a staring role in the theory of actions on hyperbolic spaces. Indeed, it is through them that we are able to see many important properties of the group. We shall use them to find an unbounded ``essential core" for actions with loxodromic elements. Similar concepts have been introduced before. For example  there is a notion of bushiness  in \cite{Bowditch1998, Manning2006}, coarse minimality in\cite{BaderCapraceFurmanSisto}, and essentiality in \cite{PropNL, CapraceSageev}.  

\begin{defn}\cite{Hamann}
   Given an action $\G\to \Isom X$, the \emph{loxodromic limit set}, denoted $\H(\G)$, is defined to be the points in $\partial X$ that are fixed by some loxodromic in $\G$. 
\end{defn}

\begin{lemma}\label{Lem: Contracting QGeod}
        Suppose  $\O\subset X$ is quasiconvex, and $q: \Z\to X$ a $(1, 20\delta)$ quasigeodesic. Let $\pi_q$ denote the projection where $\pi_q(x) = \{n \in \Z : d(q(n),x) \leq  d(q(m),x), m\in \Z\}$. If the set of elements in $\pi_q(\O)$ is not bounded above and not bounded below then there exists $M'=M'(\lambda, \delta)$ such that $q(\Z)\subset \hood_{M'}(\O)$. 
\end{lemma}

\begin{proof}
It is sufficient to prove the result for a $(1, 20\delta)$ quasiray $q:\N\to X$. Let $q(n)\to\xi \in \partial X$. Without loss of generality, we may assume that $\O$ is path connected, and the metric on $\O$ is the infimum of lengths of paths in $\O$, making it a $\delta'$-hyperbolic space in its own right; see Remark ~\ref{rem:takingnbd}. 

By assumption,  there exists a sequence $r_k\to \8$, and $x_k\in \O$ such that $r_k \in \pi_q(x_k)$. We claim that  $\{x_k\}$ is a sequence converging to $\xi$ as well. If this claim is proven, then we have that $\xi \in \partial \O$. Thus, there is a $(1, 20\delta')$ quasigeodesic ray $q'$ in $\O$ such that $q'$ converges to $\xi$ as well. As $\O$ is quasiconvex, it is easy to check that $\O$ is quasi-isometrically embedded (see for example, \cite[Lemma 3.4]{DGO}) and thus $q'$ is a quasigeodesic in $X$ as well. As a consequence of the Morse Property Theorem ~\ref{Slim Q-Bigons infinite}, we get that $q, q'$ are within bounded Hausdorff distance of each other. Thus there exists an $M'$ such that $q(\N)$ is within the $M'$-neighborhood of $\O$ in $X$. 

To prove the claim,  fix $n,m$ such that $r_n \leq r_m$. Consider the quasigeodesic 4-gon given by $$q[r_1, +\infty] \cup [\xi, x_m] \cup [x_m, x_n] \cup [x_n, q(r_1)].$$ 

 Observe that $q(r_n), q(r_m) \in q[r_1, +\infty]$. This 4-gon can be approximated by a $(1, \lambda(4))$ quasi-tree, see Lemma ~\ref{Lem: N-gons quasitrees} and Remark ~\ref{Rem: QI to QT and Gromov Product}. %Since $q[r_1, +\infty]$ is always a part of the $4-$gon, and $r_n, r_m$ are given by closest point projections, we can use the same quasi-isometry independent of $n,m$. 
 We will use primes to denote the the images under the quasi-isometry to the tree. Consequently, we get that 
\begin{eqnarray*}
    \langle x_n, x_m \rangle_{q(r_1)} &\approx_{\,2\lambda(4)}& \langle x'_n, x'_m \rangle_{q(r_1)'} = d(q(r_1)', [q(r_n)', x'_n]) \\ & \qquad =&  d(q(r_1)', q(r_n)') \approx_{\,2\lambda(4) + 20\delta}  |r_n-r_1|.
\end{eqnarray*} 
 
 Now, allowing $n,m$ to vary, these expressions tend to $+\infty$ as $n,m \to +\infty$. Thus $\{x_k\}$ is a Gromov sequence. 
Further we have that $| \langle q(r_n), x_n \rangle_{q(r_1)} - \langle q(r_n)', x'_n \rangle_{q(r_1)'}| \leq 2\lambda(4)$. Also we have that $\langle q(r_n)', x'_n \rangle_{q(r_1)'} = d(q(r_1)', [q(r_n)', x'_n])$. By the above calculation, it follows that $\{x_k\}$ is equivalent to $\{q(r_k)\}$ and hence converges to $\xi$ as well. 
\end{proof}

\begin{lemma}\label{Lem:Geods in Tree}
    Let $T$ be a tree, and $\xi(1), \dots, \xi(4)\in \partial X$ such that  $[\xi(1), \xi(2)], [\xi(3), \xi(4)]$ are bi-infinite geodesics. Suppose moreover that $x\in [\xi(1), \xi(2)], y\in [\xi(3), \xi(4)]$. Up to a permutation of the boundary points, we have that $[\xi(1),x]\cup[x,y]\cup [y, \xi(3)]$ is a bi-infinite geodesic. 
\end{lemma}
\begin{proof}
Under our hypothesis, the cardinality $|\{\xi(1),\dots, \xi(4)\}|=2, 3, 4$. Therefore, the non-trivial cases are demonstrated in Figure ~\ref{fig:geods in tree}, where the jagged line segments could possibly have length 0 and the result follows. 
\begin{figure}
        \centering
        \includegraphics[width=3in]{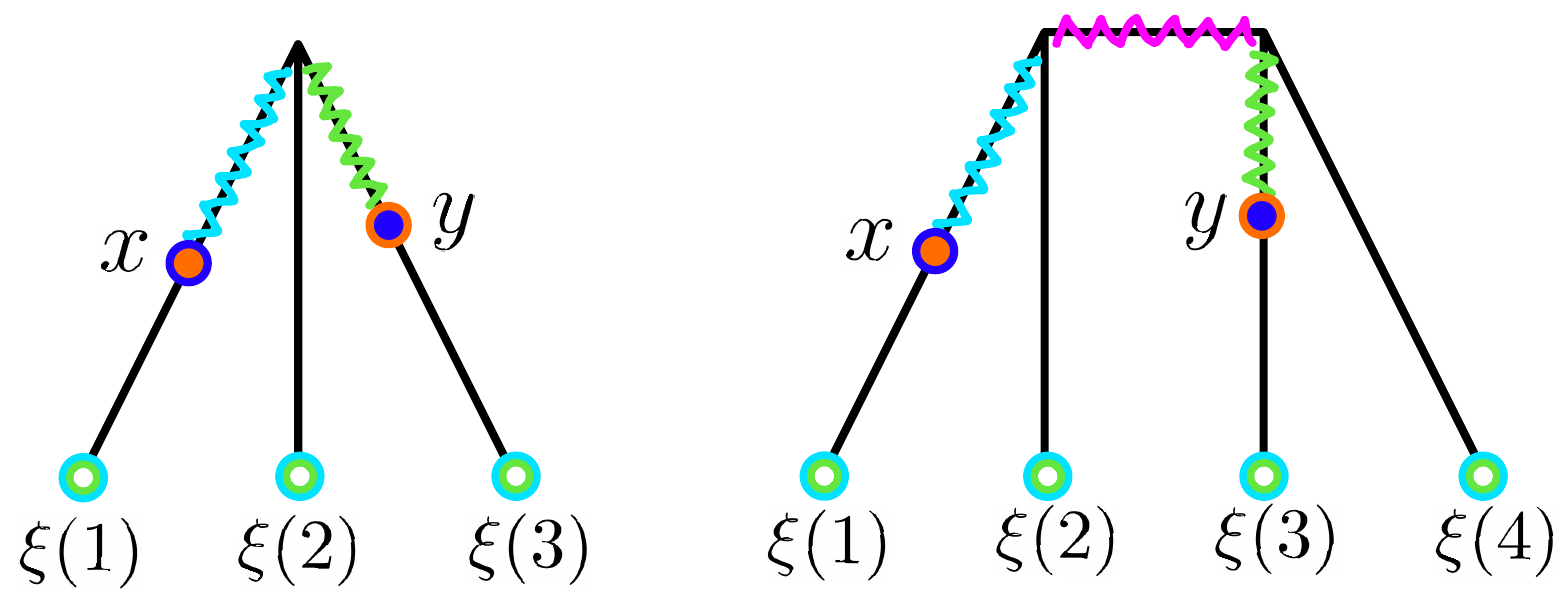}
        \caption{See Proof of Lemma ~\ref{Lem:Geods in Tree}}
        \label{fig:geods in tree}
    \end{figure}
\end{proof}
The following states that if $(1, \lambda)$ quasigeodesics have end-points that are ``close", then they must synchronously fellow travel for a long time. 

\begin{cor}\label{Cor: Boundary close means fellow travel}
    For all $\lambda\geq 0$, $R>0$ there exists $M''= M''(\lambda, R)$ such that if $c,q: \Z\to X$ are  $(1,\lambda)$ quasigeodesic between  $\xi^\pm\in \partial X$ and $\eta^\pm\in \partial X$  respectively and $\<\eta^-, \xi^-\>_{c(0)},\<\eta^+, \xi^+\>_{c(0)}\geq R$ then there is a $n_0\in \Z$ such that for every $n\in [-R,R]$ we have that such that $d(q(n_0+n),c(n))\leq M''$.
\end{cor}

\begin{proof}
We begin by completing to a $(1,20\delta)$ $n$-gon, where $n$ is the cardinality of $\{\xi^-,\xi^+,\eta^-,\eta^+\}$ which is 2, 3, or 4. By Lemma ~\ref{Lem: N-gons quasitrees}, this is $(1, \lambda(4))$ quasi-isometric to a tree, as in Lemma ~\ref{Lem:Geods in Tree}. It is straightforward to see that the claim holds in a tree and since a $(1, \lambda(4))$ quasi-isometry preserves the Gromov product up to an additive error of at most $2\lambda(4)$. 
\end{proof}

\begin{lemma}\label{Lem: Convergence}
    Fix $o\in X$ and $\xi\in \partial X$. If $c_n: \Z\to X$ is a sequence of $(1, \lambda)$ quasigeodesics such that neither sequence $c_n(-\8)$, nor $c_n(+\8)$ converges to $\xi$, then for every $R>0$ there is an $N$ so that if $n>N$ we have for all $m \in \Z$
    $$\<c_n(m),\xi\>_o\leq R.$$
\end{lemma}

\begin{proof}
    Let $R>0$. Since the sequences of boundary points $c_n(\pm\8)$ do not converge to $\xi$   there is an $N$ so that if $n>N$ then $\<c_n(\pm\8),\xi\>_o\leq R$. Let $q:\N\to X$ be a $(1, 10\delta)$ quasigeodesic between $o$ and $\xi$. Then, for each $n>N$ we may choose $(1, 20\delta)$-quasi geodesics so that $q(\N)$ and $c_n(\Z)$ are opposite sides of a 4-gon which by Lemma ~\ref{Lem: N-gons quasitrees} is a quasitree whose consants depend only on $\lambda$ and $\delta$. 
    
    Fix $n>N$. There is a $(1, \lambda')$ quasi-isometry mapping this 4-gon to a quasi-tree. Coarsely, we are then in one of the cases in Figure ~\ref{fig:Convergence}, where the jagged lines may be of length 0, and the result follows by Remark ~\ref{Rem: QI to QT and Gromov Product}.

     \begin{figure}
        \centering
        \includegraphics[width=0.6\linewidth]{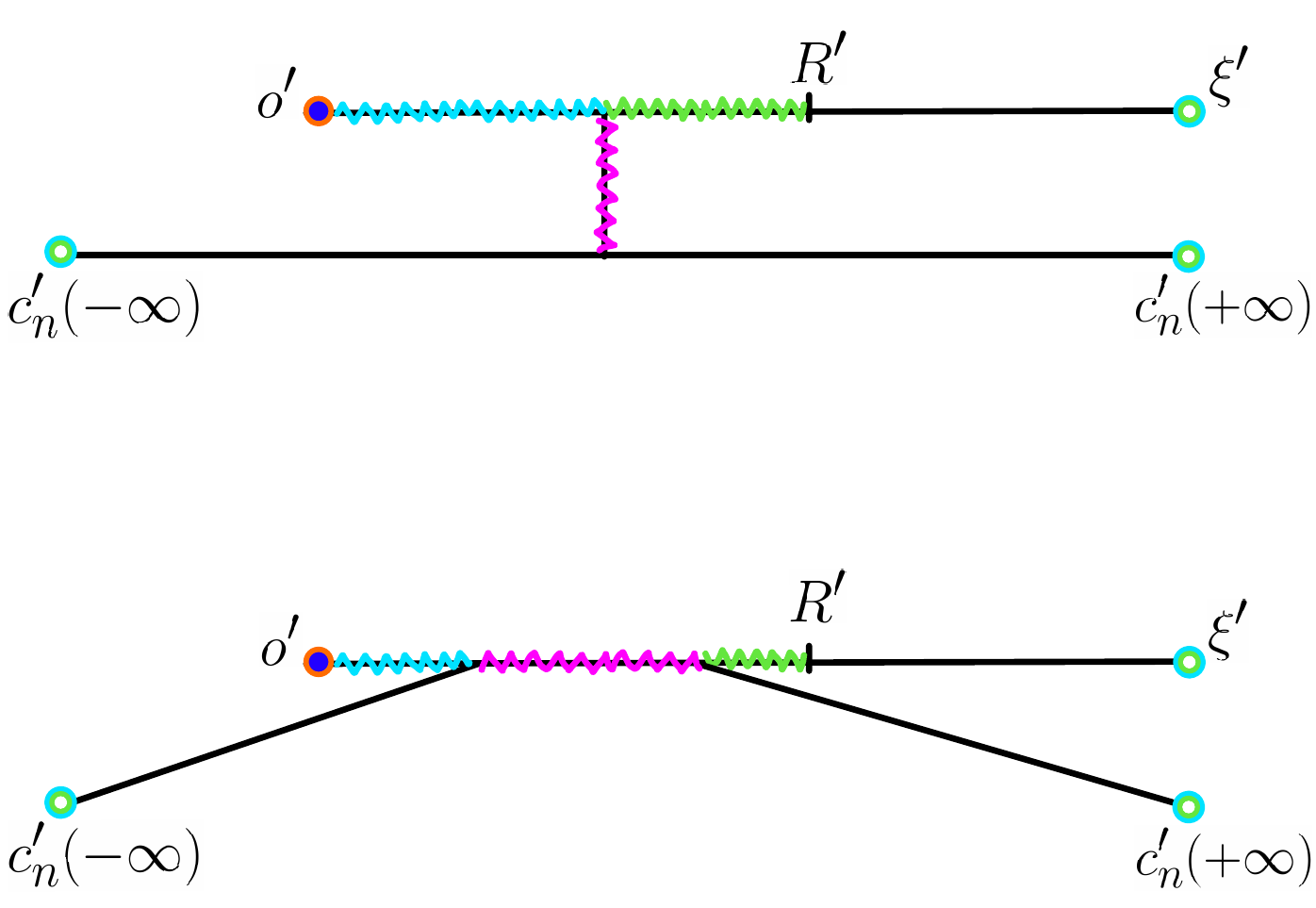}
        \caption{ Top: $m(o', \xi', c_n'(-\8))=m(o', \xi', c_n'(+\8))$. Bottom: otherwise. From  proof of Lemma ~\ref{Lem: Convergence}.}
        \label{fig:Convergence}
    \end{figure}
\end{proof}

The following combines \cite[Theorem 2.9]{Hamann} and Corollary ~\ref{Cor: Boundary close means fellow travel}. Hamman's theorem states that the loxodromic limit set is bilaterally dense in the limit set. We turn this into a statement about synchronous fellow traveling for the quasigeodesics. 

\begin{theorem}\label{thm: lox lim set is dense}
An action $\G\to \Isom X$ is not parabolic if and only if for every $R>0$ and $(1,20\delta)$ quasigeodesic $c:\Z\to X$ with end points in $\L(\G)$ there exists an element acting as a loxodromic $\g\in \G$ and $q: \Z\to X$  a $(1, 20\delta)$ quasi-axis for $\g$ such that $\<\g^{-\8}, c(-\8)\>_{c(0)}, \<\g^\8, c(\8)\>_{c(0)} \geq R$  and $d(c(n),q(n))\leq M'$ for all $n\in [-R,R]$. \end{theorem}

We now establish a few properties about quasiconvex sets that will be used to define the essential core. The first of which is immediate from the fact that the limit set is independent of choice of base-point and noting that elliptic actions have empty limit set.

\begin{lemma}\label{Lem: Limit set minimal}
    For an action $\G\to \Isom X$ if $Y \subset X$ is $\G$-invariant and quasiconvex then $\partial Y\supset \L(\G)$.
\end{lemma}

\begin{remark}\label{rem:takingnbd}
    We note that if $S\subset X$ is a $\sigma$-quasiconvex subset then the neighborhood $\~\hood_\sigma(S)$ is path connected. Indeed, consider two points $x,y \in \~\hood_\sigma(S)$. By definition, there exist points $x', y' \in S$ such that $d(x,x') \leq \sigma$ and $d(y,y') \leq \sigma.$ Consider the geodesics $[x,x'], [y',y]$ and $[x',y']$. First, for any point $t \in [x,x']$, we have that $d(x,t) \leq d(x,x') \leq \sigma$. Thus $t \in \~\hood_\sigma(S)$ and  $[x, x'] \subset \~\hood_\sigma(S)$. By an analogous argument, we get that $[y',y] \subset \~\hood_\sigma(S)$. Lastly, the geodesic $[x',y'] \subset \~\hood_\sigma(S)$ by quasiconvexity of $S$. It follows that the concatenated path $[x,x'] \cup [x',y'] \cup [y', y]$ is in $\~\hood_\sigma(S)$ and so the neighborhood $\~\hood_\sigma(S)$ is path connected. Note also that the concatenated path $[x,x'] \cup [x',y'] \cup [y'y]$ is a $(1, 4\sigma)$ quasigeodesic and therefore $\~\hood_\sigma(S)$ is also $M$-quasiconvex, where $M = M(1, 4\sigma)$ is the associated Morse constant. 

        The constant $r(\rho)$ in the  definition of the essential core and the following lemma is chosen so as to ensure we are in the context of a path-connected quasiconvex set which is therefore hyperbolic, as we shall see in the proof of Lemma ~\ref{Lem: irred core is essential invariant}.
\end{remark}

\begin{defn}[The Essential Core]
    Let  $\rho:\G\to \Isom X$ be of general type. Define the \emph{essential core} of the action to be $\mathcal{L}_\rho(X)=\~{\hood}_{r(\rho)}\left(\Cup{\g\in \mathrm{Lox}(\rho(\G))}{}\mathbb{A}(\g)\right)$, where $\mathbb{A}(\g)$ is the union of $(1, 20\delta)$ quasigeodesics between the end points of $\g$ in $\partial X$, and $\mathrm{Lox}(\rho(\G))$ is the set of elements of $\g\in\G$ so that $\rho(\g)$ is loxodromic. 
\end{defn}

The following result is a combination of \cite[Lemma 2.21]{PropNL} and Theorem ~\ref{thm: lox lim set is dense}.

\begin{lemma}\label{Lem: irred core is essential invariant}Let $\G\to \Isom X$ be such that $|\H(\G) |\geq 2$. Then the essential core $\mathcal{L}_\rho(X)$ is non-empty, quasiconvex, $\G$-invariant, on which  the $\G$-action is essential. 
\end{lemma}

\begin{proof}
Let $Y= \Cup{\g\in \mathrm{Lox}(\G)}{}\mathbb{A}(\g)$.
    By the classification of actions Theorem ~\ref{thm:actionsclassrk1} and  Remark ~\ref{Rem:BenakliKapovich} we have that  $Y$ is nonempty and clearly $\G$-invariant. Let us prove that it is quasiconvex.

    Suppose that $x,y \in Y$, with $c,q: \Z\to X$ $(1, 20\delta)$ quasigeodesics  with $c(0)=x$ and $q(0)=y$, respectively, with end-points in $\H(\G)$. Fix a geodesic $[x,y]$. By Lemma ~\ref{Lem: N-gons quasitrees}, we have that $c(\Z)\cup q(\Z)\cup [x,y]$ is a quasitree which is   $(1,\lambda)$ quasi-isometric to a simplicial tree by \cite[Theorem 1.3]{Kerr} (where $\lambda=\lambda(4)$ is independent of the position of the points).

    By Lemma ~\ref{Lem:Geods in Tree} up to reversing the orientations of either $c$ or $q$ (or both) we have that $c(-\8,0]\cup[x,y]\cup q[0, \8)$ is a $(1, \lambda')$ quasigeodesic, where $\lambda'\geq 20\delta$ depends only on $\delta$ and $\lambda$. 

    This means that $[x,y]$ is within $M(1, \lambda')$ of any $(1, 20\delta)$ quasigeodesic connecting $c(-\8)$ and $q(\8)$, let $c': \Z\to X$ be a choice of such a quasigeodesic. By Theorem ~\ref{thm: lox lim set is dense}, there exists a loxodromic $\g\in \mathrm{Lox}(\G)$ so that if $q':\Z\to X$ is a $(1, 20\delta)$ quasigeodesic between the end points of $\g$ then $q'$ fellow travels $c'$. Thus the geodesic $[x,y]$ is within uniformly bounded distance (depending on $\lambda, \delta$ only) of the axis of the loxodromic $\gamma$, which is in $Y$. Hence $Y$ is quasiconvex. 

    By Remark ~\ref{rem:takingnbd}, there is an $r(\rho)\geq 0$ which depends only on the action so that $\mathcal{L}_\rho(X)$ is path connected and quasiconvex. Thus $\mathcal{L}_\rho(X)$ is a hyperbolic space with respect to the induced metric from $X$. 
    It remains to show the action is essential. Let $Z\subset \mathcal{L}_\rho(X)$ be nonempty,  quasiconvex and $\G$-invariant. By Remark ~\ref{rem:takingnbd}, we may take a uniform neighborhood of $Z$ (inside $\mathcal{L}_\rho(X)$) and assume $Z$ is also path connected, hence, hyperbolic with constant $\delta'$. It follows by  Lemma ~\ref{Lem: Limit set minimal} that $\partial Z\supset \L(\G)$.

    Let $x\in \mathcal{L}_\rho(X)$. Then there exists $\g\in \mathrm{Lox}(\G)$ and $c:\Z\to X$ a $(1,20\delta)$ quasigeodesic with $d(c(0),x) \leq r(\rho)$. However, as $c(\pm \infty) \in \partial Z$, $c$ connects two points in the boundary of $Z$ also. Therefore, there exists a $(1,20\delta')$ quasigeodesic $q: \Z\to Z$ and by the Morse property, we have that $d(c(0), q(n))\leq M(1,20\delta')$, for some $q(n) \in Z$. Thus $d(x,q(n)) \leq r(\rho) + M(1 ,20\delta')$ and this shows that $Z$ is coarsely dense.  
\end{proof}

While we normally do not label the map denoting the homomorphism from our acting group to the isometry group, we make an exception in the following as there are several maps to consider. 

\begin{lemma}\label{lem:tremblefreerk1}
    If  $\rho:\G\to \Isom X$ is of general type then $\mathcal{L}_\rho(X)$ is $\rho(\Gamma)$-invariant and quasiconvex and $\partial \mathcal{L}_\rho(X) = \L(\rho(\G))$. Moreover, setting  $\T=\ker(\G\to \Homeo(\L(\rho(\G))$ there is a $\delta'$-hyperbolic geodesic space $X'$ equipped with an injective homomorphism $\rho':\G/\T\hookrightarrow \Isom X'$ whose associated action is essential and a quasi-isometry $\f: \mathcal{L}_\rho(X) \to X'$ that is $\G$-equivariant with respect to the restriction $\rho(\G)|_{\mathcal{L}_\rho(X)}$ and the image $\rho'(\G/\T)$.
\end{lemma}

\begin{proof} By Lemma ~\ref{Lem: irred core is essential invariant} $\mathcal{L}_\rho(X)$ is a hyperbolic space which is $\rho(\G)$-invariant, on which the action is essential. By Lemma ~\ref{Lem: Limit set minimal} we have that $\partial \mathcal{L}_\rho(X) \supset \L(\rho(\G))$. To show the reverse inclusion, suppose that $\xi\in \partial \mathcal{L}_\rho(X)$. Then there exists a sequence of $(1, 20\delta)$ quasigeodesics $c_n: \Z\to X$ and $m_n\in \Z$ such that $c_n(m_n) \to \xi$ and such that $c_n(\Z)\subset \mathbb{A}(\rho(\g_n))$ for loxodromics $\rho(\g_n)$. Therefore, by Lemma ~\ref{Lem: Convergence}, up to possibly reversing the orientations,  we deduce that $c_n(+\8)\to \xi$. By \cite[Theorem 2.9]{Hamann}, we have that $\~{\H(\rho(\G))} = \Lambda(\rho(\G))$.  In particular, $\xi \in\Lambda(\rho(\G))$.

Next, apply Corollary ~\ref{Cor:Non-LC tremble tame} to trivialize the action of $\T$ and obtain $\X'$. The essentiality of the action is preserved since the promised quasi-isometry is $\G$-equivariant.  Therefore $\T$ is in the kernel. On the otherhand, the kernel surely acts trivially on $\partial X$ and so the homomorphism $\G/\T\to \Isom X'$ is injective. 
\end{proof}

\begin{cor}\label{Cor: FI-ess}  Suppose $\rho:\G\to X$ is general type and essential and $H < \G$ is of finite index. Then the induced action of $H$ on $X$ is also essential. 
\end{cor}

\begin{proof} Since the action is of general type and essential, we have that, up to possibly increasing $r(\rho)$,  $\mathcal L_\rho(X)=X$. Since $H$ is of finite index in $\G$, the set $\mathcal{L}_\rho(X)$ is the same for both $H$ and $\G$. Indeed, every loxodromic element from $G$ has a (uniform) power that is loxodromic in $H$. As a consequence, the essential core for $H$ is the same as the essential core for $\G$. Since the action of $\G$ was essential on $X$, so is the action of $H$.
\end{proof}

\begin{remark}[Tremble-free essential core]\label{rem: irred core in grom bordn} The space $X'$ as in Lemma   
~\ref{lem:tremblefreerk1} will be referred to as the \emph{tremble-free  essential core} of the action. 

We make the following observations. 
\begin{enumerate}
    \item If $\rho(\G)$ is elliptic, then the loxodromic limit set is empty and hence, setting $\T=\G$ and we may take $Y$ to be an orbit, which is quasiconvex by Lemma ~\ref{Lem: elliptic orbits are quasiconvex} (and clearly $\G$-invariant) and set $X'$ to be a point. 
    \item Lemma ~\ref{Lem: irred core is essential invariant} does not hold in the lineal case as loxodromics would act trivially on the associated boundary points and thus be in the kernel considered in the statement.
    \item  As in \cite[Remark 3.6]{BaderCapraceFurmanSisto}, it follows from \cite[Remark 4]{GruberSistoTessera} that if $\G\to \Isom X$ is not parabolic then there is another $\delta$-hyperbolic geodesic space $X'$ on which $\G$ also acts by isometries and a coarsely $\G$-equivariant quasi-isometric embedding $X'\hookrightarrow X$, and moreover, $X'$ may be made separable. 
  
\end{enumerate}
\end{remark}

\begin{defn}
  An action $\G\to \Isom X$ is said to be 
  \begin{itemize}
      \item \emph{tremble-free} if the only elements that act as trembles are those in the kernel;
      \item \emph{rift-free} if no subgroup of $\G$ acts as a rift, i.e. elliptic subgroups are trembles or rotations;
      \item \emph{essential-rift-free} if no subgroup of $\G$ acts as a rift on the essential core.
  \end{itemize}
\end{defn}

\section{Actions in Higher Rank}\label{sec:isomhigherrank}

For the notation used in the following results, please refer to Table ~\ref{table:notation} and Section ~\ref{Subsec:Products} for an explanation.

\subsection{Taming in higher rank}

We begin by assembling our taming results from Section ~\ref{Subsec:Taming Space via Action} in the higher rank setting. 

\begin{cor}[Higher rank tremble-free essential core]\label{Cor: pass to irreducible tremble free core}
    Let $\rho:\G \to \Aut \X$ be (AU-)acylindrical with general type factors. Then, there exists $\X':=\CProd{i=1}{D}X_i'$ a product of $D$-many  $\delta'$-hyperbolic spaces such that:
    \begin{enumerate}
        \item there is a homomorphism $\rho':\G\to \Aut \X'$ that factors through $\rho$  such that the action is (AU-)acylindrical;
        \item the action on each factor is essential, of general type, and tremble-free;
        \item if $\L_i$  denotes the limit set in the $i^{\text{th}}$ coordinate then the following kernel is finite. $$\ker(\rho') = \ker(\G \to \Homeo(\CProd{i=1}{D}\L_i));$$  
        \item if a factor of $\X$ is locally compact, then so is the associated factor of $\X'$.
    \end{enumerate}
\end{cor}

\begin{proof} 
    Apply Lemma ~\ref{lem:tremblefreerk1} to find the tremble-free essential cores. Of course, this is happening on the finite index subgroup  $\G_0$ that is the kernel $\G\to \Sym_\X(D)$. If this map is not trivial, then an element of $\G$ takes some factor $X_i$ to $X_j$. But this means that the projected action of $\G_0$ to $X_i$ and $X_j$ are conjugate, and therefore, so are the associated essential cores $X_i'$ and $X_j'$. This means that the action $\G_0\to \Aut_{\!0}\X'$ lifts to an action $\G\to \Aut\X'$, with isomorphic permutation quotient $\G\to \Sym_{\X'}(D)$.

    By Lemma ~\ref{Lem: irred core is essential invariant}, $\mathcal{L}_\rho(X_i)$ is an invariant subspace of $X_i$. Since the action of $\G$ on $\X$ is (AU-)acylindrical, the action of $\G$ on $\Prod{i=1}{D} \mathcal{L}_\rho(X_i)$ is (AU-)acylindrical. By Lemma ~\ref{lem:tremblefreerk1}, we have a $\G_0$-equivariant quasi-isometry from $\mathcal{L}_\rho(X_i)$ to $X'_i$ for all $1 \leq i \leq D$. This implies that we have a $\G_0$ equivariant quasi-isometry from $\Prod{i=1}{D} \mathcal{L}_\rho(X_i)$ to $\Prod{i=1}{D} X'_i = \X'$, which extends to an action of $\G$. Thus the action of $\G$ on $\X'$ is also (AU-)acylindrical. This justifies items (1) and (2). 

    Item (3) also follows from Lemma ~\ref{lem:tremblefreerk1}. Lastly, note that $\ker(\rho')$ is finite by  Lemma ~\ref{Lem:AU-acyl has finite kernel} and so item (4) follows from Proposition ~\ref{Prop: loc compact tremb quot} applied to each locally compact factor.  
\end{proof}

\begin{lemma}\label{lem:finquotsacyl}
Let $\G \to \Aut_{\!0}\X$ be (AU-)acylindrical and with general type factors. Let $N \norm \G$ be a finite normal subgroup. Then $\G/N$ admits an (AU-)acylindrical action on $\X'$. 
\end{lemma}

\begin{proof} Applying Lemma ~\ref{Lem: irred core is essential invariant}, we first obtain an action of $\G$ on $\Prod{i=1}{D} \mathcal{L}_\rho(X_i)$, which is (AU-)acylindrical and essential and of general type of each factor. Since $N$ is finite, it follows from Lemma ~\ref{Lem: normal in total gen type is tremb/gen type}, that the induced action of $N$ on (each factor of) $\Prod{i=1}{D} \mathcal{L}_\rho(X_i)$ is a tremble. But then $N$ is in the kernel of the action. As a consequence of Lemma ~\ref{lem:tremblefreerk1}, we get an action of $\G/N$ on $\X'$, which is also (AU-)acylindrical. \end{proof}

\subsection{Isometries in higher rank}
We now turn our attention to classifying isometries in higher rank. Note that some of these have been considered in the work of Button \cite[Section 4]{buttonhigherrank}. 

\begin{defn}
An element $g=(g_1, \dots, g_D)\in \Aut_{\! 0}\X$ is said to be
\begin{enumerate}
\item \emph{elliptic} if all components $g_i$ are elliptic;
\item \emph{parabolic} if all components $g_i$ are  parabolic;
\item \emph{pseudo-parabolic} if all components $g_i$ are either elliptic or parabolic, with at least one of each type;
\item \emph{regular loxodromic} if all components $g_i$ are loxodromic;
\item \emph{weakly-loxodromic} if at least one component $g_i$ is loxodromic, but $g$ is not regular. 

\end{enumerate}
\end{defn}

We now explore acylindrical actions in the higher rank setting (see  Corollary ~\ref{cor:elimpapra} and Lemma ~\ref{Lem:noqpinacyl}). Our goal is to show that factors where the action is elliptic or parabolic do not ``contribute" to acylindricity, and can thus be omitted without loss of acylindricity (see Lemmas ~\ref{Lem:elimellfactors} and ~\ref{prop:removepara}). For the case of AU-acylindricity, we note that elliptic factors can be omitted. 

We note that the notion of an elliptic action is well defined for any metric space, and so we prove the following lemma in that more general setting. 

\begin{lemma} \label{Lem:elimellfactors} Suppose that $\G \to \Isom Y \times \Isom Z$ is (AU-)acylindrical, where $Y,Z$ are metric spaces. If the projection $\G \to \Isom Y$ is elliptic then the projection $\G \to \Isom Z$ is (AU-)acylindrical. 
\end{lemma}

\begin{proof} We shall just prove the case of an acylindrical action since the case of nonuniformly acylindrical is analogous.  Let $\e >0$ be given. Fix $y \in Y$. Then there is a constant $B >0$ such that $\op{diam}(\G y) \leq B$. Let $R, \mfN$ be the acylindricity constants for $\e' = \e +B$. Let $a,b \in Z$ be any points such that $d(a,b) > R$. Consider the set $$S = \{g \in \G \mid d(a, ga) \leq \e \text{ and } d(b, gb) \leq \e \}.$$ We claim that $|S| \leq \mfN$, which will prove acylindricity.  

Firstly, observe that \begin{align*}
    d( (y ,a), (y, b)) &= d(y ,y) + d(a, b) \\
    &= d(a, b)\\
    & \geq R
\end{align*}

Further, if $g \in S$, then \begin{align*}
    d((y, a), g(y, a)) &= d(y, gy) + d(a, ga) \\
    & \leq B+ \e\\
    & = \e'
\end{align*}

Similarly, $d((y,b), g(y,b)) \leq \e'$. By acylindricity of the action on $Y \times Z$, we have that $|S| \leq \mfN$.    \end{proof}

In a vein similar to the previous lemma, the next lemma allows us to remove parabolic factors from an acylindrical action on a product. Again, we prove the result in a more general setting, as it might be helpful for future explorations.

\begin{prop}\label{prop:removepara} Let $\G \to \Isom Z \times \Isom X$ be acylindrical, where $Z$ is a metric space and $X$ is hyperbolic. If the projection $\G \to\Isom X$ is parabolic, then the projection $\G \to \Isom Z$ is acylindrical. 
 \end{prop} 

\begin{proof} Since the action of $\G$ on $X$ is parabolic, $\G$ has a fixed point in $\partial X$, and hence the associated Busemann quasimorphism $\b: \G\to \R$ is identically $0$. Let $E$ be the constant from Lemma ~\ref{Lem:shiftfn}. 

We proceed by contrapositive. Assume, that the action on $Z$ is not acylindrical. Then, there exists an $\e >0$ such that for any $R, \mfN >0$, there exist points $z,w\in Z$ with $d(z,w) \geq R$ and $$| \op{set}_\e(z,w)| =| \{ g \in \G \mid d(z,w) \leq \e \text{ and } d(w, gw) \leq \e \}| \geq \mfN.$$

Set $\e' = \e + E$. We will complete the proof by showing that the action of $\G$ on $Z \times Y$ is not acylindrical for this $\e' >0$.  Let $\G=\Cup{n\geq 1}{}S_{n}$ be a nested sequence of finite symmetric sets exhausting $\G$. Choose $m$ large enough so that $|\op{set}_\e(z,w)\cap S_m| \geq N+1$. By Lemma ~\ref{lemmaelimpara}, there is a $t_m>0$ such that $M(t) \geq m$ for all $t \geq t_m$. It follows from the proof of Lemma ~\ref{lemmaelimpara}, that $y = q(t_m) \in X$ is such that $$d(gy, y) \leq E$$ for all $g \in S_m \cap \op{set}_\e(z,w)$. Then $(z, y), (w,y) \in Z \times X$ and $d((z, y), (w,y)) \geq R.$ Consider $g \in S_m \cap \op{set}_\e(z,w)$. Then for $u\in \{z,w\}$, we have that $$d((u,y), g(u, y)) = \max{} \{d(u, gu), d(y, gy)\} \leq \e + E \leq \e'.$$ As a consequence $| \op{set}_{\e'} ((z, y), (w, y))| \geq \mfN+ 1 > \mfN$. This shows that any choice of $R, \mfN$ fail to satisfy the acylindricity condition for $\e'$ on $Z \times X$.
\end{proof}

\begin{cor}\label{cor:elimpapra} Let $\G \to \Aut_{\!0}\X$ such that each factor is parabolic. Then the action of $\G$ on $\X$ is not acylindrical. 
\end{cor}

\begin{proof} Assume, by contradiction, that the action on $\X$ is acylindrical. Repeatedly applying Proposition ~\ref{prop:removepara} to the factors of $\X$ ultimately leaves one factor with a parabolic, acylindrical action, which is impossible by \cite[Theorem 1.1]{Acylhyp}.
\end{proof}

We now eliminate the possibility of quasiparabolic actions in the case of acylindrical actions in higher rank. 

\begin{lemma}\label{Lem:noqpinacyl} Suppose that $\G \to \Aut\X$ such that the action on each factor is quasiparabolic. Then the action of $\G$ on $\X$ is not acylindrical. 
\end{lemma}

\begin{proof} Since acylindricity of the action passes to subgroups, without loss of generality, assume $\G\to \Aut_{\!0}\X$. By contradiction, assume that this action is acylindrical and quasiparabolic in each factor. For $i\in \{1, \dots, D\}$, let $\xi_i\in \partial X_i$ be the $\G$-fixed point. Fix $(1, 20\delta)$ infinite quasigeodesics $q_i: \N \to X_i$ converging to $\xi_i$, and  let $\b_i: \G\to \R$ be the Busemann quasimorphism associated to the action in each factor and of defect at most $B \geq 0$.  Set $\b: \G\to \R^D$, and $q: \N \to \X$ be the associated diagonal maps.  Considering  the $\ell^\infty$ metric on $\X$ and $\R^D$, $q$ is a $(1, 20\delta)$ quasigeodesic in $\X$.

Define $F_n = \left\{ g \in \G \: \|\b(g)\|_\8 \leq n \right\}$, for $n\in \N$. The reader may think of these as F\o lner sets and we will show they have polynomial growth. 

\noindent
\textbf{Claim 1:} For each $n \in \N$, $| F_n| \leq \mfN$, where $\mfN$ is the smallest constant from the definition of acylindricity associated to $\e_n = n+ 20\delta + E$. 

Suppose by contradiction, that there exists a $k_0$ such that $|F_{k_0}| > \mfN$ associated to $\e_{k_0}$. Let $Q\subset F_{k_0}$ be of cardinality  $\mfN+1$. Recall our notation that we have $h\mapsto (h_1, \dots, h_D)$. For each $ h\in Q$, it follows from Lemma ~\ref{Lem:shiftfn}, that if $t \geq t_0(h_i)$, then $d(q_i( t - \b_i(h)), h_iq_i(t)) \leq E$. Using the triangle inequality and that $q_i$ is a $(1, 20\delta)$ quasigeodesic, we get 
\begin{align*}
    d(q_i(t), h_iq_i(t)) &\leq d(q_i(t), q_i(t - \b_i(h)) + d(q_i(t - \b(h), h_iq_i(t))\\
    & \leq |\b_i(h)| + 20\delta + E\\
    & \leq k_0 + 20\delta +E.
    \end{align*}

Take $T =\op{max} \{ t_0(h_i) \mid h \in Q, 1 \leq i \leq D \}$. Then for all $t \geq T$, it follows that $$d(q(t), h.q(t)) \leq k_0 + 20\delta +E= \e_{k_0}$$

Choose $a, b \geq T$ such that $d(q(a), q(b)) \geq R(\e)$, where $R$ is the constant from the definition of acylindricity for $\e_{k_0}$. Such $a,b$ exists as $q(t)$ is an infinite quasigeodesic. Then it follows that for all $h \in Q$, $d(q(a), h.q(a)) \text{ and } d(q(b), h.q(b)) \leq \e_{k_0}$. This contradicts the acylindricty since $|Q| = \mfN +1 > \mfN$. Thus Claim 1 holds. 

It is easy to verify that $\bigcup_{n} F_n = \G$.  Since we use subscripts to denote the factors of $\X$, we will use superscripts here to enumerate elements -- these should not be confused for powers. Let $s^1, \dots, s^k \in \G$ and let $S=\<s^1, \dots, s^k \>^+$ be the semigroup generated by them. We now show that  $S$ has polynomial growth. Fix $m \in \N$ to be the smallest index such that $s^1,\dots, s^k \in F_m$.

\noindent
\textbf{Claim 2:} For any $l$, the product $a^1\cdots a^l \in F_{l(m + B)}$ where each $a^i \in \{s^1, \dots, s^k\}$ is a generator of $S$.

Since $a^1, \dots, a^l \in F_k$ and $\b_i$ is a quasimorphism, it is easy to check that for each $i= 1, \dots, D$, $$|\b_i(a^1\cdots a^l)| \leq lB + |\b_i(a^{1})| + \dots |\b_i(a^{l})| \leq lB + lm,$$ which proves the claim.   

However, as $\G$ admits a quasiparabolic action on a hyperbolic space, it follows from the standard Ping-Pong lemma, that $\G$ contains a finitely generated free subsemi-group, which is a contradiction.
\end{proof}

We now have all the tools to prove the following theorem.

\begin{theorem}\label{thm:elimmixedfactors} Let $\G \to \Aut_{\!0}\X$  such that the projections $\G\to \Isom X_i$ are all either elliptic, parabolic or quasiparabolic (with at least one factor being parabolic or quasiparabolic). Then $\G \to \Aut \X$ is not acylindrical.

In particular, if $\G \to \Aut_{\!0}\X$ is acylindrical, then every element of $\G$ is either an  elliptic, weakly loxodromic, or regular loxodromic isometry.
\end{theorem}

\begin{proof} The theorem follows by combining Lemmas  ~\ref{Lem:elimellfactors} and ~\ref{Lem:noqpinacyl} and Proposition ~\ref{prop:removepara}. If $\G \leq \Aut_{\!0}\X$ is acylindrical, then so is every subgroup, in particular $\langle g \rangle$ for each $g \in \G$. The result now follows since it must be the case that $g$ is either elliptic in each factor or has at least one loxodromic factor.
\end{proof}

\subsection{Regular Elements Exist}
We now turn our attention to regular elements i.e those that preserve factors and act as a loxodromic in each factor. Our goal is to prove the following proposition.  (See also \cite{ CapraceSageev, CapraceZadnik, FLM}.)

\begin{prop}\label{Prop:Regular Exist}
If  $\G \to \Aut\X$ has general type factors then $\G$ contains regular elements.
\end{prop}

The proof of Proposition ~\ref{Prop:Regular Exist} relies on the following results of Maher and Tiozzo and properties of probability measures, which will allow us to find regular elements in actions with general type factors with asymptotically high probability.

\begin{theorem}\cite[Theorem 1.4]{MaherTiozzo}\label{thm: MaherTiozzo}
 Let $\G\to \Isom X$ be a general type action, and $\mu$ be a generating symmetric probability measure on $\G$. Then the translation-length $\tau(w_n)$ of the $n^{th}$-step of the $\mu$-random walk grows at least linearly, i.e. there exists $T>0$ such that $\mathbb P(w\in \G^\N:\tau(w_n)\leq Tn) \to 0$ as $n\to \8$. 
\end{theorem}

\begin{lemma}\label{Lem:prob}
Let $\e\in (0,1)$, and $\P$ be a probability measure on $\Omega$. Assume $A_1, \dots, A_D\subset \Omega$ are measurable sets such that $\P\left(A_i\right) > 1-\frac{\e}{D}$ for each $i=1, \dots, D$. Then $\P\left(\Cap{i=1}{D}A_i\right) >1-\e$. 
\end{lemma}

\begin{proof}
We shall use $A^c:= \Omega \setminus A$ to denote set compliments. For $i=1, \dots, D$ our assumption is equivalent to $\P(A_i^c)\leq \frac{\e}{D}$ which means that 

$$ \P\left(\Cup{i=1}{D}A_i^c\right)\leq\Sum{i=1}{D}\,\P(A_i^c)\leq \e$$
Therefore, $\P\left(\Cap{i=1}{D}A_i\right)\geq 1-\e$.
\end{proof}

\begin{proof}[Proof of Proposition ~\ref{Prop:Regular Exist}]
Up to passing to a finite index subgroup, without loss of generality we  assume that $\G\to \Aut_{\!0}\X$. Fix a generating probability measure $\mu$ on $\G$, and $0<\e<1$. Since each projection is of general type, we may apply Theorem ~\ref{thm: MaherTiozzo} to each action $\G\to \Isom X_i$. This gives real numbers $T_i>0$, so that for any $n$ sufficiently large (where $\tau_i: \G\to \R$ denotes the translation length of elements acting on $X_i$) we have the following measurable sets:

\begin{enumerate}
    \item $A_i = \{w\in \G^\N:\tau_i(w)> T_i n\}$
    \item $\P(A_i) > 1 - \dfrac{\e}{D}$
    
\end{enumerate}

It then follows from Lemma ~\ref{Lem:prob} that $\P\left(\Cap{i=1}{D}A_i\right) > 1 -\e,$. In particular, the set $\Cap{i=1}{D}A_i$ consists of regular elements and is not empty.
\end{proof}

\section{A Plethora of Examples}\label{sec:examples}

In this section we will discuss the enormous class of groups which admit an AU-acylindrical action with general type factors on non-positively curved spaces. Observe that this class of groups is closed under taking products and  subgroups (for which the action has general type factors). Moreover, applying the technique of induced representations from finite index subgroups shows that the class of groups is also closed under finite extensions.  Corollary ~\ref{Cor: Class closed finite kernel} also shows that the class is closed under taking finite-kernel quotients and therefore under virtual isomorphism and Lemma ~\ref{lem:finquotsacyl} and Corollary ~\ref{cor:prodstrdecomp} show that the class is closed under direct factors. We note that the question as to whether acylindrical hyperbolicity is retained under finite index extension remains open. We begin by compiling a list of examples that can be found in the existing literature.

\begin{enumerate}

\item Acylindrically hyperbolic groups (see also \cite[Theorem 1.2]{Acylhyp}), with $D=1$, which is itself an extensive list; for example:

\begin{itemize}
\item Most mapping class groups \cite{Bowditch2008, MasurMinsky}.
    \item (Nonelementary) Hyperbolic groups
    \item (Nonelementary)  relatively hyperbolic groups (see e.g. \cite{DGO})
     \item Nonamenable amalgamated products (see \cite{MinasyanOsin} and \cite[Theorem 2.25]{DGO})
    \item Many small cancellation groups \cite{GruberSisto}
    \item  Groups acting properly on proper CAT(0) spaces with rank-1 elements \cite{Sisto2018}. In particular, this includes irreducible RAAGs
   
\item Many Artin groups \cite{HaettelXXLArtin, MartinPrzytycki, Vaskou, HagenMartinSisto}

\end{itemize}

\item The class of Hierarchically hyperbolic groups (abbreviated as HHG), is closed under direct products but not necessarily direct factors, has finitely many  eyries and the corresponding diagonal action is acylindrical, see \cite[Remark 4.9]{PetytSpriano}. To ensure the factor actions are of general type, it suffices to assume no eyrie is a quasi-line. This is a broad class itself and contains \cite{DurhamZalloum}:
\begin{itemize}
    \item Mapping class groups and right angled Artin Groups (RAAGs) and fundamental groups of special CAT(0) cube complexes, and cocompact lattices in products of $\delta$-hyperbolic spaces \cite{BHS1, HagenSusse}
    \item 3-manifold groups with no $Nil$ or $Sol$ components \cite{BHS2}
\item Surface-by-multi-curve stabilizer groups \cite{russell2021}
 \item The genus two handlebody group \cite{Chesser_2022}
\item Surface-by-lattice Veech groups \cite{DowdallDurhamLeiningerII}
\end{itemize}

\item Since our class is closed under passing to nonamenable subgroups that act properly, we also obtain groups with exotic finiteness properties such as the kernels of Bieri-Stallings and Bestvina-Brady, see \cite{Bieri, Stallings}, \cite{BestvinaBrady} and other subdirect products such as those in \cite{BridsonHowieMillerShort, Bridson, BridsonMiller} (See also \cite{NicolasPy}).

\item Infinite groups with property (QT) admit an action on a finite product of quasitrees whose orbit maps yield a quasi-isometric embedding, which are in particular proper \cite{BBFQt}. In this case, due to properness, the action is AU-acylindrical. 

\item Groups acting properly on products of trees, such as Coxeter groups and RAAGs \cite{Januszkiewicz, Button}, or finitely generated subgroups of $\PGL_2K$, where $K$ is a global field of positive characteristic \cite{FisherLarsenSpatzierStover}.

\item Finitely generated infinite groups that are quasi-isometric to a finite product of proper cocompact non-elementary $\delta$-hyperbolic spaces act geometrically, and hence acylindrically on a (possibly different) finite product of rank-1 symmetric spaces and locally finite $\delta$-hyperbolic graphs \cite[Theorem K]{Margolis}. 

\item\label{Ex:S-arith} Finitely generated subgroups $\G < \SL_2 \~\Q$, where $\~\Q$ is the algebraic closure of $\Q$. Indeed, in this case there exists $F$ a finite Galois extension of $\Q$ such that $\G\leq \SL_2 F$. Also because $\G$ is finitely generated, there exist finitely many places $S$ such that the $s$-norm of the entries in a generating set for $\G$ is bigger than 1. Therefore we have in fact that $\G < \SL_2 O_F[S^{-1}]$, where $O_F$ is the ring of algebraic integers in $F$. 

Recall that there exists $r\in \N$, $c\in \N$ such that $O_F[S^{-1}] \hookrightarrow \Prod{i=1}{r} \R \times \Prod{j=1} {c} \mathbb{C} \times \Prod{s \in S}{} F_s$ is a discrete cocompact embedding, where $F_s$ is the $s$-adic completion of $F$ (see for example \cite[Section 5.4]{Rapinchuk}). As a consequence the following is a discrete embedding $$\G \hookrightarrow  \Prod{i=1}{r} \SL_2\R \times \Prod{j =1}{c} \SL_2\mathbb{C} \times \Prod{s \in S}{}\SL_2 F_s.$$ 

Finally, taking $\X$ to be the product  of $r$-many copies of $\mathbb{H}^2$, $c$-many copies of $ \mathbb{H}^3$ and the Bruhat-Tits trees $T_s$ associated to each $\SL_2 F_s$ yields a proper (and hence AU-acylindrical) action on $\X$. 
 
\item Limit groups (in the sense of Sela) or fully residually free (in the sense of Kharlampovich and Miasnikov) are relatively hyperbolic with respect to their higher-rank maximal abelian subgroups \cite{Alibegovic,Dahmani}. Subdirect products of such have been studied \cite{BridsonHowieMillerShort}, \cite{KochloukovaLopezdeGamizZearra}.

\item If a finitely generated infinite group acts on a metric space $Y$ such that the orbit map is a quasi-isometric embedding of the group, then it follows easily that the group admits an AU-acylindrical action on $Y$. Such phenomena occurs for a large class of groups and was recently studied under the notion of $A/QI$ triples by Abbott-Manning in \cite{AbbottManning}. In forthcoming work, we will be explore this notion in the higher rank setting \cite{BFApSM}. 
\end{enumerate}

\begin{remark}
    Mapping class groups importantly appear in both (1) and (4) (with $D>1$). We take the opportunity to highlight the importance of taking a ``representation theoretic" approach when viewing this theory. Often within the context of geometric group theory there is a drive to view one action as superior to others. However, here we see that belonging to both of these classes provide different information and are therefore both of interest.    
\end{remark}

\subsection{Genericity}
In this section, we deduce that admitting an AU-acylindrical action on a CAT(0) space is generic among the finitely generated subgroups of $SU_n$. The key to this is the following result by Douba.

\begin{theorem}\cite[Theorem 1.1.i]{Douba}
Let $\G\leq \SL_n\mathbb C$ be finitely generated such that no element of $\G$ is unipotent. Then $\G$ acts properly on a CAT(0) space. 
\end{theorem}

If $n=2$ then the associated CAT(0) space will be the $\ell^2$-product of copies of $\mathbb{H}^2, \mathbb{H}^3$ and (not necessarily locally finite) Bruhat-Tits trees, as in Example ~\ref{Ex:S-arith} above.

The reader may wonder how to find examples of such groups that are free of non-trivial unipotents: Consider the special unitary group $SU_n\leq \SL_n\mathbb C$, which is the group of matrices whose inverse is given by the transpose-complex conjugate, i.e.  $AA^*=I$. By way of this equation, we see that all eigen-values of a unitary matrix must belong to the unit circle in $\mathbb C$, and in particular, the only unipotent element (i.e. all eigenvalues are identically 1) in $SU_n$ is the identity. Moreover, if $n>1$ then $SU_n$ is not virtually solvable, and so by the Tits alternative, it contains free groups. Therefore, the direct product of $SU_2\times SU_2$ contains uncountably many isomorphism classes of groups \cite{BaumslagRoseblade}. Using Lemma ~\ref{Lem:acyl+loc comp implies unif proper}, we obtain the following corollary\footnote{We note that, since our class closed under virtual isomorphism, $SU_n$ can also be replaced with any group within its isogeny class.}:

\begin{cor}
    Let $n_1, \dots, n_k\in \N$, and $\G\leq \CProd{i=1}{k}SU_{n_i}$ be finitely generated. Then $\G$ acts properly (and hence AU-acylindrically) on a CAT(0) space. If in fact $\G$ can be conjugated in to $SU_{n_i} \cap \SL_{n_i}(\~\Q)$ for each $i=1, \dots, k$ then the CAT(0) space is uniformly locally compact and hence the $\G$ action is acylindrical.  If $\{n_1, \dots, n_k\}=\{2\}$, then the CAT(0) space is a finite product of copies of $\mathbb H^2$, $\mathbb H^3$, and locally finite trees.  
\end{cor}

\subsection{\textit{S}-Arithmetic Lattices}\label{Sect: S-arith}

In this section we will primarily see examples of groups that do \emph{not} have acylindrical or even AU-acylindrical actions on finite products of $\delta$-hyperbolic spaces (but do admit such actions on CAT(0) spaces). 

Haettel proved that a lattice in a product of higher rank, almost simple connected algebraic groups with finite centers over a local field only admits elliptic or parabolic actions on $\delta$-hyperbolic spaces \cite{Haettel}. 

On the other hand, they admit proper actions on the corresponding product of symmetric spaces and/or Bruhat-Tits buildings (as in Example ~\ref{Ex:S-arith}). This action is uniformly proper (and hence acylindrical) if the lattice is uniform and the action is nonuniformly proper (and hence nonuniformly acylindrical) in the nonuniform case (i.e. when the lattice contains unipotent elements). Therefore, while these do not fall in to our framework of AU-acylindrical actions on products of hyperbolic spaces, they do fall in the framework of AU-acylindrical actions on non-positively curved spaces.

In a similar direction, Bader-Caprace-Furman-Sisto \cite{BaderCapraceFurmanSisto} proved that lattices in certain products of locally compact groups only admit actions on hyperbolic spaces in as much as their ambient group does.

\begin{defn}\label{Def: standard rank-1}
    A locally compact second countable group $G$ is a \emph{standard rank one group} if it has no non-trivial compact normal subgroup and either 
    \begin{enumerate}
        \item $G$ is the group of isometries or orientation-preserving isometries of a rank-1 symmetric space of non-compact type, or
        \item $G$ has a continuous, proper, faithful action by automorphisms on a locally finite non-elementary tree, without inversions,  with exactly two orbits of vertices, and whose action on the boundary is 2-transitive.
        \end{enumerate}
\end{defn}

With this definition in place, we state their striking result \cite[Theorem 1.1]{BaderCapraceFurmanSisto} in our setting:

\begin{theorem}\label{Thm: BCFS}
    Let  $\G$ be an  irreducible lattice in a product of locally compact second countable groups $G= \CProd{i=1}{k}G_i$, where each $G_i$ is either standard rank one or a semi-simple algebraic group over a local field. An  essential general type action of $\G$ on a $\delta$-hyperbolic space must be equivalent to one that factors through a projection of $\G$ to $G_i$, if it is standard rank-1. (Otherwise, no such actions exist.)
\end{theorem}

From this remarkable result, we deduce the following:

\begin{cor}\label{Cor:rank1}
    A lattice $\G$ as in Theorem ~\ref{Thm: BCFS} admits an AU-acylindrical action on a finite product of $\delta$-hyperbolic spaces if and only if its ambient group has only standard rank-1 factors. Furthermore, the corresponding action is acylindrical if and only if $\G$ is cocompact. 
\end{cor}

\begin{cor}\label{Cor:SL2ZS}
Let $\G=\SL_2\Z[S^{-1}]$, where $S$ is a finite nonempty set of primes, or more generally a nonuniform irreducible lattice in a product of standard rank-1 groups, with at least one factor of type (1) as in Definition ~\ref{Def: standard rank-1}.
Then $\G$ does not act acylindrically on any finite product of hyperbolic spaces, unless the actions are all elliptic.
\end{cor}

\begin{proof}
 By Theorem ~\ref{Thm: BCFS}, without loss of generality, we may assume that the action of $\G$ on a $\delta$-hyperbolic space is one of these standard actions. The fact that these lattices are nonuniform means that they have unipotent exponentially distorted elements \cite[Theorem B]{LubotzkyMozesRaghunathan}. Therefore, they must be elliptic or parabolic in each factor. By Theorem ~\ref{thm:elimmixedfactors} the action is not acylindrical. \end{proof}

\subsection{Groups with Property (NL)}\label{Subsec: NL} In \cite{PropNL}, the authors initiate the systematic study of groups that do not admit \emph{any} actions on hyperbolic spaces with a loxodromic element -- such groups are said to have Property $\nl$.  In other words, the only  actions such a group can admit on a hyperbolic space are elliptic or parabolic. If every finite index subgroup also has Property $\nl$, the group is said to be \emph{hereditary} $\nl$. Examples of groups with Property $\nl$, besides the obvious finite groups, include Burnside groups, Tarskii monster groups, Grigorchuk groups, Thompsons groups $T,V$ and many ``Thompson-like" groups (see \cite[Theorem 1.6]{PropNL} for a more extensive list). It is worth noting however that Thompson's group and similar-type diagram groups act properly on infinite-dimensional CAT(0) cube complexes \cite{Farley}. Furthermore, while finitely generated torsion groups do not act on 2-dimensional CAT(0) complexes \cite{NorinOsajdaPrzytycki},  free Burnside groups do act non-trivially on infinite dimensional CAT(0) cube complexes \cite{Osajda}. Finitely generated bounded-torsion groups are conjectured not to admit fixed-point free actions on locally compact CAT(0) spaces. In a personal communication, Coulon and Guirardel have told us of some upcoming work: there is a finitely generated infinite amenable (unbounded) torsion group G which admits a proper action on an infinite dimensional CAT(0) cube complex.

Returning to our context, it follows from \cite[Corollary 6.5]{PropNL}, that groups with the hereditary $\nl$ property also do not admit interesting actions on products of hyperbolic graphs. In particular, it follows from the proof of \cite[Proposition 6.4]{PropNL} and Theorem ~\ref{thm:elimmixedfactors}, that if $\G$ is a hereditary $\nl$ group, and $\X$ is a product of hyperbolic spaces, then $\G \to \Aut \X$ is acylindrical only if the action is elliptic in each factor.

%%%%%%%%%%%%%%%%%%%%%%%%%%%%%%%%%%%%%

\section{Elementary subgroups and the Tits Alternative}\label{sec:elemsubgrp}
In this section, we explore the structure of the stabilizer of a pair of points in $\partial_{reg}\X= \Prod{i=1}{D}\partial X_i$ . If the points have distinct factors and the action is AU-acylindrical, we prove the stabilizer is virtually free abelian of rank no larger than $D$. This generalizes the analogous result in rank-1 \cite[Proposition 6]{BestvinaFujiwara} and for HHGs \cite[Theorem 9.15]{DurhamHagenSisto}, although our proof uses new techniques. We then consider the case of a single regular point in an acylindrical action and obtain the same result. More specifically:

\begin{prop}\label{prop:elemsubvirtab} If $\G\to \Aut \X$ is AU-acylindrical, and $\xi^\pm \in \partial_{reg}\X$ with distinct factors then $E_\G(\xi^\pm):= \stab_\G\{\xi^-, \xi^+\}$ is virtually isomorphic to $\Z^k$ where $0\leq k \leq D$. \end{prop} 

\begin{prop}\label{prop: reg fixed implies reg pair} Suppose that $\xi \in \partial_{reg}  \X$, where $\G \to \Aut_{\!0}\X$ is acylindrical. Then either $\stab (\xi)$ is finite or there exists an $\xi' \in \partial_{reg} \X$ such that $\xi \neq \xi'$ and $\stab_\G(\xi) = \stab_\G\{\xi, \xi'\}$. 
\end{prop}

Note that the statement of Proposition ~\ref{prop: reg fixed implies reg pair} does not claim that $\xi'$ has distinct factors from $\xi$, as the reader will also see in the proof. However, we note that if $D'$ is the number of factors where $\stab_\G(\xi)$ is not elliptic nor parabolic, then we can ensure that $\xi,\xi'$ have $D'$ many distinct factors and hence $\stab(\xi)$ is virtually $\Z^k$ for some $1 \leq k \leq D'$.

Our strategy to prove the above propositions is as follows:

\begin{enumerate}
    \item We prove that $E_\G(\xi^\pm)$ is amenable by showing that all finitely finitely generated subgroups have uniformly polynomial growth. 
    \item Up to passing to a finite index subgroup to guarantee that the components of $\xi^\pm$ are in fact fixed, we utilize the associated product of the Busemann homomorphisms, to descend the action to $\R^D$ and prove that it is AU-acylindrical; this allows us to deduce Proposition ~\ref{prop:elemsubvirtab}.
    \item We then deduce Proposition ~\ref{prop: reg fixed implies reg pair}  by utilizing the uniformity  of acylindricity to exclude the possibility of a quasi-parabolic factor and hence  find another  point in $\partial_{reg} \X$ which is also fixed and then conclude by applying Proposition ~\ref{prop:elemsubvirtab}.
\end{enumerate}

Assuming the above results hold, we can immediately prove our version of the Tits Alternative as stated in Theorem ~\ref{intro:titsalt}. 

\begin{proof}[Proof of Theorem ~\ref{intro:titsalt}] By Lemma ~\ref{Lem:AU-acyl has finite kernel} the kernel $\G\to \Aut\X$ is finite. Without loss of generality, we assume $\G\leq \Aut_{\!0}\X$. If any projection $\G\to \Isom X_i$ is general type then $\G$ contains a free group by the standard Ping-Pong argument. Therefore, assume no factor is of general type.

    By Lemmas ~\ref{Lem:elimellfactors}, up to possibly obtaining another finite kernel, since the action is not elliptic, we may assume that no factor is elliptic.
    Furthermore, by Lemma ~\ref{prop:removepara}, we  may also remove any parabolic factor as by Corollary ~\ref{cor:elimpapra}, not all factors can be parabolic. 

    Therefore, each factor must be lineal or quasi-parabolic, and in particular fixes a point $\xi\in \partial_{reg}\X$. We are then in the case of Proposition ~\ref{prop: reg fixed implies reg pair}. Applying the proposition, we conclude that the action in each remaining factor is necessarily oriented lineal. This completes the proof.
\end{proof}

\subsection{The rank-1 case.}
We begin by establishing some results about isometries of quasi-lines. Recall  that $(1, \lambda)$ quasigeodesics between the same end-points synchronously fellow travel by Corollary ~\ref{cor: (1,mu)-fellow travel}. The following statement generalizes this to the group setting. Note that a quasi-isometry of a quasitree can be realized by a $(1,\lambda)$ quasi-isometry by \cite{Kerr}. 

\begin{lemma}\label{lemma: Unif Bound}
 Let $L$ be a quasi-line and $q: \Z \to L$ be a $(1,\lambda)$ quasi-isometry that is $M$-coarsely surjective. Let $\Isom_{\!0}(L)$ denote the isometries which fix  the Gromov boundary  $\{- \8, +\8\}$ pointwise. Fix  $g\in \Isom_{\!0}(L)$. Setting $C_0(g)= 2M+ 3\lambda + 3d(g q(0), q(0))$ and $C(g) = 2M + C_0(g)$ we have that

 \begin{enumerate}
\item  If $n\in \Z$ then $d(g q(n), q(n))\leq C_0(g)$. 
\item If $x\in L$ then $d(g x, x)\leq C(g)$.
\end{enumerate}
If furthermore $\G\to \Isom_{\!0}(L)$ is elliptic then the constant $C_{ell}(\G) = 4M + 3\lambda +3B_0$ is such that $d(g x, x)\leq C_{ell}(\G)$ for every $g\in \G$ and $x\in L$, where $B_0:=\Sup{g\in \G}d(g q(0), q(0))<\8$. In particular an elliptic action on a quasi-line that fixes the Gromov boundary pointwise is a tremble.
\end{lemma}

\begin{proof}
We note that since $q$ is $M$-coarsely surjective, we may take it to be the Morse constant. 

(1) Observe that for $g\in \Isom_{\!0}(L)$ we can directly apply Corollary ~\ref{cor: (1,mu)-fellow travel} with $c=g.q$ and hence we set $C_0(g)= 2M+3\lambda+ 3d(g q(0), q(0))\geq 0$ so that for every $n\in \N$ 
$$d(g. q(n), q(n))\leq C_0(g).$$

\noindent
(2) Let $x\in L$. There is an $n$ such that $d(x, q(n)) \leq M$ and since $g$ is an isometry, $d(gx, g.q(n)) \leq M$. We deduce:

\begin{eqnarray*}
d(gx,x) &\leq& d(x, q(n))+ d(q(n), g.q(n)) + d(g.q(n), gx)
\\
&\leq&2M +  C_0(g)= C(g).
\end{eqnarray*}

\noindent
Now consider the case when $\G\to \Isom_{\!0}(L)$ elliptic. Then the constant $B_0$ defined in the statement is finite. Applying  part (1) uniformly to $q$ and all $g.q$, $g\in \G$ we can set $C'_{ell}(\G)= 2M + 3\lambda +3B_0$ so that $d(g q(n), q(n)) \leq C'_{ell}(\G)$.

Let $x\in X$. Then, $d(x, q(n))\leq M$ for some $n\in \Z$. Letting $g\in \G$ and using the triangle inequality as above, we get 
\begin{align*}
    d(g x, x)     & \leq 2M+ C'_{ell}(\G) = C_{ell}(\G)
\end{align*}
\end{proof}

\begin{remark}\label{Rem: ell, fixing point on bound}
    We note that the above proof also shows that if $\G\to \Isom X$ is elliptic and also fixes a point $\xi \in \partial X$, then there is an $L>0$ so that the set $\O^L(\G)$ is unbounded. Indeed, take a $(1, 10\delta)$ quasi-ray $q$ converging to $\xi$. Since $\G$ fixes $\xi$, the Hausdorff distance between $q$ and $gq$ for any $g \in \G$ is uniformly bounded. The arguments from the lemma can then be adapted to the quasi-ray $q$ to get the desired conclusion.  
\end{remark}

\subsection{The higher rank case.} We are now ready to begin tackling elementary subgroups in the higher rank case. We begin with some definitions. 

\begin{defn}
Fix $\G \to  Aut_{\!0}\X$ and $\xi^-, \xi^+\in \partial_{reg}\X$ with distinct factors. The subgroup  $\Cap{i=1}{D}\fix_\G\{\xi_i^-, \xi_i^+\}$ is the associated \emph{elementary subgroup}, denoted by $E_\G(\xi^\pm)$. 
\end{defn}

The following is immediate from Lemma ~\ref{lemma: Unif Bound} applied to the factors. Note that we are using the $\ell^\8$-product metric. 

 \begin{cor}\label{cor: Unif Bound quasiflat}
 Let $q=(q_1, \dots, q_D): \Z^D \to \CProd{i=1}{D}L_i$ be a product of $(1,\lambda)$ quasigeodesics  both with  the $\ell^\8$-product metric. We take $q$ to be  $M$-coarsely surjective, where $M=M(\delta,1,\lambda)$ is the Morse constant. Consider an element  $g\in \Prod{i=1}{D}\Isom(L_i)$ whose factors act trivially on $\partial L_i$ for each $i=1, \dots, D$.  Set $C_0(g)= 2M+ 3\lambda+ 3d(g q(0), q(0))$ and $C(g) = 2M + C_0(g)$. Then we have that
 
 \begin{enumerate}
\item If $a\in \Z^D$ then $d(g.q(a), q(a))\leq C_0(g)$.
\item If $x\in \CProd{i=1}{D}L_i$ then $d(g x, x)\leq C(g)$.
\end{enumerate}

 Lastly, suppose that $\G\to \CProd{i=1}{D}\Isom (L_i)$ is elliptic and whose projections also act trivially on $\partial L_i$, for $i\in \{1, \dots, D\}$.  Set $B_0:=\Sup{g\in \G}d(g q(0), q(0))<\8$ and $C_{ell}(\G) = 4M + 3\lambda +3B_0$. Then we have that $d(g x, x)\leq C_{ell}(\G)$ for every $g\in \G$ and $x\in \CProd{i=1}{D}L_i$, i.e. $\G$ acts as a tremble.
\end{cor}

We are now ready  to prove the first step towards the amenability of $E_\G(\xi^\pm)$ when $\G$ is acting acylindrically on $\X$. We note that to keep with our convention that subscripts denote components in products, we must use superscripts for sequences. We hope that the reader will not confuse these for powers, as the only powers we take are inverses. 

\begin{theorem}\label{theorem: elementary acyl is amenable}
Let $\G \to \Aut_{\!0}\X$ be  AU-acylindrical. If $\xi^\pm\in \partial_{reg}\X$ have distinct factors then  $E_\G(\xi^\pm)$ is amenable.
\end{theorem}

\begin{proof}
We shall use the $\ell^\8$ metric on the product in this proof.
Let $\Fl=\Prod{i=1}{D}L_i\subset \X$ be the corresponding product of all $(1,20\delta)$ quasigeodesics between $\xi_i^-$ and $\xi_i^+$, and $q: \Z^D \to \Fl \subset \X$ be the map as defined in Corollary ~\ref{cor: Unif Bound quasiflat}. By Theorem ~\ref{Slim Q-Bigons infinite}, each component is $M(\delta, 1, 20\delta)$ coarsely surjective. Assume that the $i$-th component of $q$ is a $(1, \lambda_i)$ quasi-isometry for $i= 1, \dots, D$. Set $$\lambda=\max{} \{ 20\delta, M(\delta,1,20\delta), \lambda_i \mid i= 1, \dots, D\}.$$  Since $\G$ is AU-acylindrical on $\X$, so is the action of $E_\G(\xi^\pm)$. As $E_\G(\xi^\pm)$ preserves $\Fl$, it follows that  $E_\G(\xi^\pm)$  is also AU-acylindrical on $\Fl$. 

By  Corollary ~\ref{cor: Unif Bound quasiflat} (1),  for each $g\in \G$ and $a\in \Z^D$ and with $C_0(g)= 2M+ 3\lambda+ 3d(g q(0), q(0))$, $$d(g.q(a), q(a))\leq C_0(g) $$

As a consequence, given any $C \geq 0$, if $d(g.q(0), q(0)) \leq C$, then $$d(g.q(a), q(a))\leq C_0(g) \leq 2M + 3\lambda + 3C.$$

In particular the $C$-coarse stabilizer of $q(0)$ is contained in the $2M+3\lambda+3C$ coarse stabilizer of $q(0)$ and $q(a)$ for any $a\in \Z^D$.

Set $C=2\lambda$ and apply the AU-acylindricity to $\e = 2M + 3 \lambda + 3C = 2M + 9\lambda$. This give us an $R=R(\e)>0$ and we may fix a point $z \in \Z^D$ such that $d(q(0), q(z)) >R$. The AU-acylindricity condition on the pair $q(0), q(z)$ gives us that $N_{\e}=N_{\e}(q(0),q(z))$ (the smallest AU-acylindricity cardinality constant associated to $\e$, $q(0)$, $q(z)$  satisfies 

$$|\{g \in E_\G(\xi^\pm): d(g.q(0), q(0)) \leq 2\lambda\}|\leq N_{\e}<\8.$$

Let $F_n = \Cup{a\in [-n,n]^D}{}\{ g \in E_\G(\xi^\pm): d(g.q(0), q(a))\leq \lambda \}$.  It follows from the coarse surjectivity of $q$ that $\Cup{n\in \N}{} F_n = E_\G(\xi^\pm)$. To give an upper bound on the cardinality $|F_n|$ observe that  for each $a\in [-n,n]^D$ either $\{ g \in E_\G(\xi^\pm): d(g.q(0), q(a))\leq \lambda \}= \varnothing$ or there is a $g_{a}\in \{ g \in E_\G(\g): d(g.q(0), q(a))\leq \lambda \}$. If there is such a $g_{a}$, then for any $h\in \{ g \in E_\G(\xi^\pm): d(g.q(0), q(a))\leq \lambda\}$, we get that 
\begin{eqnarray*}
   d(g_{a}^{-1}.h .q(0), q(0))&\leq& d(g_{a}^{-1}.h .q(0), g_a^{-1}q(a))+ d(g_a^{-1}.q(a), q(0)) \leq  2 \lambda .
\end{eqnarray*}

Therefore, $| g_a^{-1}\cdot \{ g \in E_\G(\xi^\pm): d(g.q(0), q(a))\leq \lambda\}|\leq N_{\e}$ and we have shown by isometry of the action that

\begin{eqnarray*}
|F_n |\leq N_{\e}\cdot (2n+1)^D.
\end{eqnarray*}

Note that the finite constant $N_{\e}$ above is independent of $n$. Since amenable groups are closed under taking unions, we now prove that every finitely generated subgroup of $E_\G(\xi^\pm)$ has polynomial growth of degree bounded by $D$ and is hence amenable. To this end, let $S\subset E_\G(\xi^\pm)$ be a finite set and consider the associated finitely generated subgroup $H$. Then there is a $k^0\in \N$ so that  $S\subset F_{k^0}$, i.e. if $s^1, s^2 \in S$ then there are $a^1, a^2 \in [-k^0,k^0]^D$ so that $d(s^i q(0), q(a^i))\leq \lambda$ for $i=1, 2$. 

Since the action is by isometries, this implies that 
\begin{eqnarray*}
d(s^1s^2.q(0), q(a^1+ a^2))&\leq& d(s^1s^2.q(0), s^1.q(a^2)) + 
d(s^1.q(a^2), s^1.q(a^1+a^2)) \\
&&+ d(s^1.q(a^1+a^2), q(a^1+a^2))\\
&\leq& \lambda +\|a^1\|_\8+ \lambda + \max{s\in S}\, d(s.q(a^1 + a^2), q(a^1 + a^2))\\
&\leq&2\lambda +k^0 + \max{s\in S}\, C_s,
\end{eqnarray*}
where $C_s = C_0(s)$ are the constants provided by Corollary ~\ref{cor: Unif Bound quasiflat}.

Now by the coarse surjectivity, there exists  $b \in \Z^D$ such that $d( s^1s^2.q(0), q(b)) \leq \lambda$. But then 

\begin{eqnarray*} 
\|b\|_\8 & = &\|a^1 + a^2 +b -a^1 -a^2\|_\8 \\
 & \leq & \|a^1 + a^2 -b\|_\8 + \|a^1 +a^2\|_\8 \\
& \leq & d( q(a^1+ a^2) , q(b)) + \lambda + 2k^0 \\
& \leq & d(q(a^1+ a^2), s^1s^2.q(0)) + d(s^1s^2.q(0), q(b)) + \lambda + 2k^0 \\
& \leq & 4\lambda + 3k^0 + \max{s\in S}\, C_s
\end{eqnarray*}

Thus $s^1s^2 \in F_{4\lambda + 3k^0 + \max{s\in S}\, C_s}$. 

We claim that $s^1s^2\cdots s^n \in F_{ (3n -2)\lambda + 3k^0 + (n-1) \max{s\in S}\, C_s}$, which we shall prove by induction on $n$, where $s^i \in S$ for all $1 \leq i \leq n$. Obviously, we have established above that this holds for $n = 2$ (and holds for $n=1$ as well). Assume we have proven the claim for $s^1s^2\cdots s^n$. i.e. there exists a $b \in \Z^D$ such that $$d(s^1s^2\cdots s^n q(0), q(b)) \leq \lambda$$ and $\|b\|_\infty \leq (3n -2)\lambda + 3k^0 + (n-1)\max{s\in S}\, C_s$. 

Let $s^{n+1}\in S$ and consider $s^1s^2\cdots s^n s^{n+1}. q(0)$. Then, with $b$ as above,  choose $w \in \Z^D$ such that $$d(s^1s^2\cdots s^n s^{n+1}. q(0), q(w)) \leq \lambda.$$ 

Then similarly to above, \begin{eqnarray*} 
\|w\|_\8 &= &\|w +b-b\|_\8\\
& \leq &\|w-b\|_\8 + \|b\|_\8 \\
& \leq& \|w-b\|_\8 + (3n -2)\lambda + 3k^0 + (n-1)\max{s\in S}\, C_s\\
& \leq &d(q(w), q(b)) + \lambda + (3n -2)\lambda + 3k^0 + (n-1)\max{s\in S}\, C_s\\
& =& d(q(w), q(b)) + (3n -1)\lambda + 3k^0 + (n-1)\max{s\in S}\, C_s \\
& \leq& d(q(w), s^1\cdots s^n s^{n+1}. q(0)) + d( s^1\cdots s^n s^{n+1}. q(0), s^1s^2\cdots s^n. q(0)) \\
 && + d( s^1s^2\cdots s^n. q(0), q(b)) + (3n -1)\lambda + 3k^0 + (n-1)\max{s\in S}\, C_s  \\
& \leq &\lambda + d(s^{n+1}. q(0), q(0)) + \lambda + (3n -1)\lambda + 3k^0 + (n-1)\max{s\in S}\, C_s \\
& \leq &(3n +1)\lambda + 3k^0 +  d(s_{n+1}. q(0), q(0))  + (n-1)\max{s\in S}\, C_s \\
& \leq &(3n +1)\lambda + 3k^0 + n\max{s\in S}\, C_s.
\end{eqnarray*}

This completes the induction argument, and gives us the linear relation that the set of words from $S$ of length bounded by $n$ is a subset of $F_{n(\lambda + \max{s \in S}C_s) + \lambda'}$.  Therefore, $H$ has polynomial growth and is hence amenable. Finally $E_\G(\g)$ is the union of it's finitely generated subgroups, all of which have polynomial growth of exponent $D$, and hence $E_\G(\g)$ is also amenable.
\end{proof}

For the remainder of this subsection,  suppose that $\G$ is an amenable group with an  action $\G\to \Aut_{\!0}\Fl$, where $\Fl = \CProd{i=1}{D}L_i$. Furthermore, suppose that $\G$ acts trivially on $\partial_{reg} \Fl$. We then obtain $\b_i: \G \to \R$, the Busemann homomorphisms for each factor and thus an action by translations on given $\R^D$ by $$g.(r_1, \dots, r_D) = (r_1 + \b_1(g), \dots, r_D + \b_D(g))$$ for any $(r_1, \dots, r_D) \in \R^D.$

We will show that AU-acylindricity of the original action on $\F$ descends to this action on $\R^D$ as well.

\begin{prop}\label{acyldescends} Let $\G$ be amenable and $\G \to \Aut_{\!0}\Fl$ be an AU-acylindrical action. Then the associated action on $\R^D$ is also AU-acylindrical.  \end{prop} 

\begin{proof}  We shall continue to use the $\ell^\8$ metric in both $\F$ and $\R^D$. Fix  $q: \Z^D\to \Fl$ to be the product of $(1, 20\delta)$ quasi-isometries on each factor.  Let $\e > 0$. Since $\G$ preserves the factors,  for  each $i\in \{1, \dots, D\}$  we may project $\G \to \Isom(L_i)$ and apply Proposition ~\ref{prop:smalltranslation} and Lemma ~\ref{Lem:shiftfn}, yielding the functions $A_i$  and associated constants $B_i, E_i$. Let $B = \max{i} B_i$ and $E = \max{i} E_i$.

Let $x= q(0)$ be a base point in $\Fl$. Let $R$ be the AU-acylindricity constant associated to the action of $\G \to \Aut \Fl$ for $\displaystyle \e' = \e + 20\delta + B + E$.

Let $r, s\in\R^D$ at distance greater than $2\e$. We claim that the AU-acylindricty on $R^D$ holds with constants $R' = 2\e$. Consider $g \in \G$ such that  $$d( r, gr) \leq \e \text{  and  } d(s, gs) \leq \e.$$

Since $\G$ is acting by translation, the above inequalities imply that $|\b_i(g_i)| \leq \e$ for all $i \in\{1, \dots, D\}$. Applying Lemma ~\ref{Lem:shiftfn} to each factor, it follows that $$d(x, gx) \leq \e + 20\delta + E.$$ 

Fix an element $h \in \G$ such that $d(h  x,  x) \geq R$. Such an element must exist as the orbit of $\G$ on $\F$ is unbounded. There are now two possibilities to consider in each factor $L_i$. If $\b_i(h_i) \leq A_i(\e + 20\delta + E_i)$, then it follows from Proposition ~\ref{prop:smalltranslation} that $$d(h_ix_i, g_ih_ix_i) \leq |\b_i(g_i)| + B_i \leq \e + B_i$$.

If not, then $\b_i(h_i) \leq A_i(n)$ for a sufficiently large $n \geq \e + 20\delta + E_i$. However, then $d(x_i, g_ix_i) \leq \e + 20\delta + E_i \leq n$, and Proposition ~\ref{prop:smalltranslation} still implies $d( h_ix, g_ih_ix) \leq |\b_i(g_i)| + B_i \leq \e + B_i$.

Thus we have $$d(x, gx)  \leq \e + 20\delta + E \leq \e'$$ and $$d(gh x , hx)  \leq \e  + B \leq \e'$$

By the AU-acylindricity, the element $g$ has at most finitely many choices (which is a uniform $N' = N(\e')$ in the acylindrical case) and thus we are done. \end{proof}

\begin{proof}[Proof of Proposition ~\ref{prop:elemsubvirtab}] By Theorem ~\ref{theorem: elementary acyl is amenable}, we know that $E_\G(\xi^{\pm})$ is amenable. Let $E'_\G(\xi^\pm)\leq E_\G(\xi^\pm)$ be the finite index subgroup whose projections to the factors of $\X$ fix $\xi^-_i$ and $\xi^+_i$ for $i\in \{1, \dots, D\}$. By Proposition ~\ref{acyldescends}, we see that the action via the Busemann quasimorphisms descends to an AU-acylindrical action on $\R^D$ by translations. By Lemma ~\ref{Lem:acyl+loc comp implies unif proper} the action is proper and hence discrete and hence by \cite[Lemma 4, p102]{BorevichShafarevich}, $E'_\G(\xi^\pm)$ is virtually isomorphic to  $\Z^k$, where $0\leq k\leq D$. Thus so is $E_\G(\xi^\pm)$. 
\end{proof}

We now turn to the proof of Proposition ~\ref{prop: reg fixed implies reg pair}. Thinking of acylindricity as being a generalization of a co-compact lattice, this result is in line with saying that there are no unipotents in a co-compact lattice (see Section ~\ref{Sect:Theory is SS} for more details). 

\begin{proof}[Proof of Proposition ~\ref{prop: reg fixed implies reg pair}] Let $H = \stab_\G(\xi)$, where $\xi\in \partial_{reg} \X$. First assume that $H$ contains a regular element $\g$. Set $\xi'$ to be the limit point of $\g$ which is distinct from $\xi$. Observe that the projection of $H$ fixes $\xi_i\in\partial X_i$ and hence the action is either (oriented) lineal or quasiparabolic for $i \in \{1, \dots, D\}$.

Proceeding as in the proof of Lemma ~\ref{Lem:noqpinacyl}, we can show that every finitely generated subsemigroup of $H$ has at most polynomial growth. This implies that the action of $H$ on each factor $X_i$ must be oriented lineal. Moreover $H \to \Aut_{\!0}\X$ is acylindrical. It follows that $H =\fix_\G\{\xi, \xi'\}$ and by Corollary ~\ref{prop:elemsubvirtab}, is virtually isomorphic to  $\Z^k$ where $0\leq k \leq D$. 

If $H$ does not contain a regular elements, then first note that the action of $H$ could be elliptic in each factor, in which case the result follows obviously. Indeed, it follows from Remark ~\ref{Rem: ell, fixing point on bound} and the acylindricity applied to a half-ray tending to each $\xi_i \in \partial X_i$ that $H$ is finite.

So we may assume that there is a factor with an unbounded action. By Theorem ~\ref{thm:elimmixedfactors}, there must be a factor with an oriented lineal action of $H$, since $H$ has a fixed point on each $\partial X_i$. Without loss of generality, we may assume that $I_e, I_p \subset \{1, \dots, D\}$ are the sets of indices where the action is elliptic and parabolic respectively. By Lemmas ~\ref{Lem:elimellfactors} and ~\ref{prop:removepara}, we may then consider the action of $H$ on $\mathbb{Y} = \Prod{ i \notin I_e \sqcup I_p}{} X_i$, which is still acylindrical. As argued above, we can conclude that $H \to \Aut\mathbb{Y}$ must be oriented lineal on all factors. Take $\xi'\in \partial_{reg} \mathbb{Y}$ where 
\begin{itemize}
    \item[(a)] $\xi'_i = \xi_i \in \partial X_i$ for $i \in I_e \sqcup I_p$, and 
    \item[(b)] $\xi'_i \in \partial X_i$ is the other limit point of the oriented lineal action for $i \notin I_e \sqcup I_p$. 
\end{itemize}

It then follows that $\xi' \neq \xi$ and that $H = \fix_\G\{\xi, \xi'\}$. Lastly, note that the factors $\xi_i$ and $\xi'_i$ are distinct for all $i \notin I_e \sqcup I_p$. Let $\eta = \Prod{i \notin I_e \sqcup I_p}{} \xi_i$ and $\eta' = \Prod{i \notin I_e \sqcup I_p}{} \xi'_i$. Then $\eta, \eta' \in \partial_{reg} \mathbb{Y}$ with distinct factors. Let $D' = D - |I_e \sqcup I_p|$. Then it follows that $H \leq \fix_\G\{\eta, \eta'\}$. By Proposition ~\ref{prop:elemsubvirtab}, $\fix_\G\{\eta, \eta'\}$ is virtually $\Z^k$ for some $1 \leq k \leq D'$; and therefore so is $H$. 
\end{proof}

\section{Connections to Lattice Envelopes}\label{sec:latticeenvs}
In recent work of Bader, Furman and Sauer, the authors establish a short list of properties which, when exhibited by a discrete countable group $\G$, dictate the  possible locally compact second countable groups in which $\G$  can be a lattice \cite{BaderFurmanSauer}.  In this section we shall see the connections of these properties to groups that admit AU-acylindrical actions on finite products of $\delta$-hyperbolic geodesic spaces. Note that since these properties are also enjoyed by linear groups with finite amenable radical, in particular ($S$-)arithmetic  lattices in semi-simple linear groups, this establishes a partial connection between our framework and such lattices.

We begin with some definitions and preliminaries that are used in the work contained in this and further sections.

\begin{defn}
    Two discrete countable groups $\G$ and $\Lambda$ are  \emph{virtually isomorphic} if there are finite index subgroups $\G'\leq\G$ and $\Lambda'\leq \Lambda$ and finite normal subgroups $K\norm \G'$, $L\norm \L'$ such that 
  $$\G'/K \cong \L'/L.$$  
\end{defn}

\begin{remark}\label{Rem: comm passed to finite intersections}
    We note that if $H \leq \G$ is commensurated, and $F\subset \G$ is a finite subset then $\Cap{g\in F}{}\, gHg^{-1}$ is of finite index in $gHg^{-1}$ for each $g\in F$ and is furthermore also commensurated in $\G$. Note also that normal subgroups are always commensurated.
\end{remark}

The following are the properties considered in \cite{BaderFurmanSauer}. 

\begin{defn}\label{defn:BFSprops}
A countable discrete group $\G$ is said to have 
 
 \begin{enumerate}
\item Property (CAF) (for Commensurated Amenable implies Finite) if every commensurated  amenable subgroup is finite;
\item Property (wNbC) (for weak Normal by Commensurated) if for every normal subgroup $K\norm \G$ and every commensurated subgroup $H < \G$ such that $K \cap H = \{1\}$ there is a finite index subgroup of $H$ that commutes with $K$;
\item Property (NbC) if every quotient of $\G$ by a finite normal subgroup has Property (wNbC); 
\item Property (BT) (for Bounded Torsion) if there is an upper bound on the order of its finite subgroups.
\end{enumerate}

Further $\G$ is called irreducible if it can only be expressed as product with itself and the trivial group and  strongly irreducible if all of its finite index subgroups and quotients by finite normal subgroups are irreducible. 
\end{defn}

With the above definitions, the following is one of the main theorems of \cite{BaderFurmanSauer}.

\begin{theorem}\cite[Lattice Envelopes]{BaderFurmanSauer}\label{BFS}
Let $\G$ be a countable strongly irreducible group with properties (CAF) and (NbC). Then, up  to virtual isomorphism, if $\G\hookrightarrow G$ is a nondiscrete lattice embedding into $G$  a compactly generated locally compact second countable group, then the embedding is one of the following types:

\begin{enumerate}
\item an  irreducible  lattice in a connected center-free, semi-simple real Lie group  without compact factors;
\item an $S$-arithmetic lattice  or an $S$-arithmetic lattice up  to  tree extension (see Section ~\ref{Sect:tree ext}).

\item a lattice in a  non-discrete totally disconnected locally compact group with trivial amenable radical.
\end{enumerate}
If in addition, $\G$ has (BT), then the  lattice  in (3) is uniform. 
\end{theorem}

\begin{remark} Since the  category of groups admitting AU-acylindrical actions on $\X$ is closed under taking direct products, it is not possible to expect them to be strongly irreducible. But we shall see in Theorem ~\ref{Thm: CanProdDecomp} that, up to virtual isomorphism, they do enjoy a canonical product decomposition into strongly irreducible groups. Another interesting connection to direct products is Corollary ~\ref{cor:prodstrdecomp}, which, under certain hypotheses, shows that if $\G$ is a direct product of infinite groups, then the factors of $\G$ are also in the class of groups acting AU-acylindrically on products of hyperbolic spaces. 

In a similar vein, we can also not expect (BT). Consider the examples given as follows: Let $W$ be a group with unbounded torsion and let $\G= W*\Z$. This group is acylindrically hyperbolic (relatively hyperbolic in fact with peripheral subgroup $W$) and does not have bounded torsion. However, acylindrically hyperbolic groups do enjoy properties (CAF) and (NbC), see \cite[Theorem 1.4 and Remark 2.9]{BaderFurmanSauer}. We note that for the group $W$ we may take the direct sum (or free product) of $\Z/n\Z$ for all $n\in \N$, or, if we wish to stay in the class of finitely generated groups, we may take $W$ to be $\Oplus{n\in \Z}{}\Z/2\Z \rtimes\Z$ or even torsion groups like Tarski monsters.
\end{remark}

\subsection{Property (CAF)} Our first goal in this subsection is to show the following. 

\begin{theorem}\label{Thm: Acyl -> CAF}
Let $\G \to \Aut\X$ be AU-acylindrical with all factors of general type. Then $\G$ has property (CAF). 
\end{theorem}

\begin{lemma}\label{Lem: Bound on number of inf factors}
    Let $\G=\CProd{i=1}{F} K_i\to \Aut_{\! 0} \X$ be AU-acylindrical such that $\G$ acts on each factor by a general type action. Assume that for each $i=1, \dots, F$, $K_i$ admits at least one general type factor. Then $1\leq F\leq D$, where $D$ is the number of factors of $\X$. 
\end{lemma}

\begin{proof}
    Up to replacing $\X$ with $\X'$ as in Corollary ~\ref{Cor: pass to irreducible tremble free core}, we may assume that all the factor actions are essential, general type, and with tremble-free. By Lemma ~\ref{Lem: normal in total gen type is tremb/gen type}, since each $K_i$ is normal in $\G$, the action in each factor is either a tremble (and hence trivial) or of general type.

    Further note that by \cite[Lemma 4.20]{ABO} (see also Lemma ~\ref{Lem:norm com}), for distinct $i,j$, if $K_i \times K_j$ acts on a hyperbolic space $X$ via a general type action, then exactly one of $K_i, K_j$ acts as general type while the other is elliptic and hence a tremble. It follows that if $K_i$ acts on a factor as general type, then all other $K_j$'s must act as trembles. Since each $K_i$ acts on at least one factor by a general type action, it follows that $1\leq F \leq D$. 
\end{proof}

We are now ready to prove the result connecting AU-acylindricity and (CAF).

\begin{proof}[Proof of Theorem ~\ref{Thm: Acyl -> CAF}] 
We will prove that if $\G\to \Aut\X$ is AU-acylindrical general type factors, then any commensurated subgroup of $\G$ with all factor actions elliptic must be finite. In particular, by Lemma ~\ref{Lem: normal in total gen type is tremb/gen type}, this will imply that any commensurated amenable subgroup of $\G$ must be finite, i.e. $\G$ has property (CAF); and that $\G$ has a finite amenable radical. 

We begin by applying Corollary ~\ref{Cor: pass to irreducible tremble free core} so that we may assume all the factors are of general type and essential and tremble-free, and so that the corresponding action has finite kernel. Therefore, up to virtual isomorphism, we may assume that $\G\leq \Aut_{\!0}\X$. 

Let $H \leq \G$ be commensurated and such that the  action of $H$ on each factor is elliptic. We now show that $H$ is finite. Fix $L$ so that $\O^L(H)\subset \X$ is not empty, and fix $x\in \O^L(H)$.  By Proposition ~\ref{Prop:Regular Exist} there exists a regular element $\g\in \G$. Set $\e=L$, and let $R$ be the associated AU-acylindricity constant. Then, up to replacing $\g$ with a power, we may assume that $d(x,\g x)\geq R$. Consider $H'=\g H\g^{-1}\cap H$, which is contained in $\cs L(x,\g x)$. Thus $H'$ is finite by the AU-acylindricity. Since $H'$ is of finite index in $H$, we deduce that $H$ is also finite. 
\end{proof}

We now prove the following crucial lemma which will be used later in this paper.

\begin{lemma}\label{Lem:norm com}
Let $\G\to \Isom X$ be essential and of general type. Consider a normal subgroup  $K\norm \G$ and a commensurated subgroup $H <\G$  with $K\cap H=\{1\}$ and $\G=KH$. One of the following cases hold:
\begin{enumerate}
    \item The action of $H$ is of general type and the action of $K$ is a tremble; or 
    \item The action of $K$ is of general type and the action of $H$ is elliptic.
\end{enumerate}
 \end{lemma}

\begin{proof}
It follows from Lemma ~\ref{Lem: normal in total gen type is tremb/gen type} that the action of  $K$ (and respectively $H$) can only be elliptic or general type. Furthermore, one of them must be general type since the action of $\G= KH$ is of general type. 

We first show that if $K$ is of general type then $H$ is elliptic. Let $g, s\in K$ be independent loxodromics. Since $H$ is commensurated, we have that  $H':=gHg^{-1}\cap sHs^{-1}\cap H$ is of finite index in $H$, as in Remark ~\ref{Rem: comm passed to finite intersections}.

Let $m\in gHg^{-1}\cap   H$. Then there is an $m' \in H$ such  that $m= gm'g^{-1}$. Therefore:
\begin{eqnarray*}
[g^{-1},m]&=&g^{-1}(mgm^{-1})\\
&=&g^{-1}gm'g^{-1}gm\\
&=&m'm.
\end{eqnarray*}
The first equality  shows that $[g^{-1},m]\in K$ whereas the last equality shows that $[g^{-1},m]\in H$. As $K \cap H = \{1\}$ by assumption, we conclude that $\<g\>$  commutes  with $gHg^{-1}\cap H$. Similarly, we obtain that $\<s\>$ commutes with $sHs^{-1}\cap H$, and so $\<g,s\>$ commutes  with  $H'$. Thus $H'$ (and consequently $H$) cannot contain a loxodromic and therefore acts elliptically. Therefore, part (2) is proven. 

A similar proof will show that if $H$ is of general type then $K$ must act elliptically. We provide the details for (1) due to some changes from the above argument. 

Let $k \in K$. As $H$ is commensurated, we have $H \cap kHk^{-1}$ is of finite index in $H$. Note that the action of $H \cap kHk^{-1}$ is of general type since it is of finite index in $H$. Take independent loxodromics $g,s$ in $H \cap kHk^{-1}$. Then $[g^{-1}, k] = g^{-1}kgk^{-1} \in kHk^{-1}$. Further, as $K$ is normal, $k = gk'g^{-1}$ for some $k' \in K$. Then $[g^{-1},k] = g^{-1}kgk^{-1} = g^{-1}gk'g^{-1}gk^{-1} = k'k \in K$.

Thus $[g^{-1},k] \in kHk^{-1} \cap K$, so $k^{-1}[g^{-1},k]k \in K \cap H = \{1\}$. It follows that $\<g\>$ commutes with $K$. By a similar argument for $s$, we get that $\<g,s\>$ commutes with $K$ and so $K$ must act elliptically. To conclude that $K$ is in fact a tremble, we apply Lemma ~\ref{Lem: normal in total gen type is tremb/gen type}. 
\end{proof}
  
We end this subsection with the following result about our class of groups and direct products, which may also be thought of as a generalization of \cite[Lemma 4.20]{ABO}.

\begin{cor}\label{cor:prodstrdecomp} Let $\G = \Prod{j=1}{F}K_j$ be a direct  product of infinite groups and $\G \to \Aut_{\! 0} \X$ be (AU-)acylindrical with general type factors. Then $\{1, \dots, D\}= \overset{F}{\Sqcup{j=1}}I_j$, where $I_j\neq \varnothing$ is the set of indices on which $K_j$ acts as general type. Setting $\X_j' = \Prod{i\in I_j}{}X_j'$ to be the associated products of the tremble-free essential cores, we have that $K_j\to \Aut_{\! 0} \X_j'$ is (AU-)acylindrical with general type factors and the map $\G\to \Aut_{\! 0} \X_j'$ factors through the projection $\G\to K_j$ for each $j=1, \dots, F$. 
\end{cor}

\begin{proof} 

We begin by applying Corollary ~\ref{Cor: pass to irreducible tremble free core} if necessary, so we may assume that all factors are additionally essential and tremble-free.
Fix $j$. It follows from Lemma ~\ref{Lem:norm com} that if $X_i$ is a factor with a general type action of $K_j$, then the action of $K_n$ for all $n \neq j$ is a tremble and hence trivial on $X_i$.

We claim that  $I_j\neq \varnothing$. Indeed by Lemma ~\ref{Lem:norm com} $K_j$ acts either as a tremble or as general type on each factor. If $I_j$ was empty, then $K_j$ would be a tremble in each factor, in particular in the kernel of the action. This is a contradiction since $K_j$ is infinite and the kernel of the $\G$-action is finite by Lemma ~\ref{Lem:AU-acyl has finite kernel}. 

Furthermore, since  $K_j$ acts (AU-)acylindrically on $\X$, it follows from Lemma ~\ref{Lem:elimellfactors} that $K_n \to \Aut_{\! 0}\X_j$ is (AU-)acylindrical.
Thus we have produced the desired partition  $\{1, \dots, D\} = \overset{F}{\Sqcup{j=1}}I_j$.
\end{proof}

\begin{cor}\label{Cor: Class closed finite kernel}
    Let $\G\to \Aut\X$ be (AU-)acylindrical and with general type factors, and $\A\norm \G$ the amenable radical and $\A_0= \G_0\cap \A$. There exists a nonempty subset $J\subset \{1, \dots, D\}$ such that $\G/\A$ acts (AU-)acylindrically on $(\Prod{j\in J}{}X_j)^I$, where $I= [\G_0/\A_0: \G/\A]$. In particular, $\G/\A$ acts (AU-)acylindrically on a finite product of $\delta$-hyperbolic spaces. 
\end{cor}

\begin{proof}
    Without loss of generality, by Corollary ~\ref{Cor: pass to irreducible tremble free core} we may assume that the action is essential and tremble-free. By Theorem ~\ref{Thm: Acyl -> CAF}, $\A$ is finite. Therefore, $\A_0:= \A\cap \G_0$ is the trembling radical and hence is the kernel of the action. Up to taking the quotient of $\G$ by $\A_0$, we may assume that $\A_0$ is trivial, which is the kernel of the map from $\A\to \Sym_\X(D)$.
    
    Consider the group generated by $\G_0$ and $\A$. Since $\G_0$ is the kernel of the map from $\G$ to $\Sym_\X(D)$, we have that $\A\cap \G_0=\A_0=\{1\}$. Furthermore, these are both normal subgroups of $\G$ and so normalize each other, and therefore $\G_0\A\cong \G_0\times\A$. 

Let $I_1\sqcup \cdots \sqcup I_F= \{1, \dots, D\}$ be a decomposition into $\A$ orbits. Up to permutation, we may assume that $j\in I_j$ for $j= 1, \dots, F$. 

We claim that $\G_0 \to \Prod{j=1}{F}X_j$ is (AU-)acylindrical. First observe that since $\G_0$ commutes with $\A$, we must have that for each $j=1, \dots, F$ the coordinates of $\G_0$ in $\Isom X_i$ must be equivalent under $\A$, i.e. are conjugate by an  element of $\A$ for all $i\in I_j$. Therefore, if we consider the diagonal injections $X_j \hookrightarrow \Prod{i \in I_j}{} X_i$ to compose an injection  $\iota: \Prod{j=1}{F}X_i\hookrightarrow \X$ we see that the image $\iota(\Prod{j=1}{F}X_i)$ is a $\G_0$-invariant subspace, which is quasi-isometrically embedded and hence the action of $\G_0$ is (AU-)acylindrical. 

Finally, since $\G_0\cap \A = \{1\}$ we get that $\G_0$ injects as a finite index subgroup into $\G/\A$. We may therefore induce the representation from $\G_0$ to $\G/\A$, which concludes the proof. 
\end{proof}

\subsection{Property (NbC)}\label{subsec:nbc} We now turn our focus to  property (NbC). As the following example demonstrates, we cannot deduce property (wNbC) in our full generality. However, it is unclear as to whether we can expect the property to hold up to taking the quotient by the amenable radical.  We will examine certain conditions under which property (NbC) can be obtained. 

\begin{example}\label{rem:nbcnotpossible} It is also not true that groups admitting acylindrical actions have (wNbC). Consider the following counter example: Let $p$ be an odd prime and define $\G$ to be the group generated by $\{g_i: i \in \Z\}$ with relations: $$g_0^p =[g_i, g_{-j}] =[g_k, g_0]=1, [g_k, g_{-k}]=g_0, \text{ for } i,j\in \N, i\neq j, k\in \Z$$

%$$\<g_i, i \in \Z: g_0^p =[g_i, g_{-j}] =[g_k, g_0]=1, [g_k, g_{-k}]=g_0, \text{ for } i,j\in \N, i\neq j, k\in \Z\>.$$

 Here, the groups $\G_+= \<\g_i,i\in \N\>$ and $\G_-= \<\g_{-i},i\in \N\>$ are each isomorphic to $F_\8$, the free group on infinitely many generators and each commutes with the finite group $C:=\<g_0\>$. Let $K=C\G_\pm$ and $H = \G_\mp$. Then $K$, $H$,  and $\G$ satisfy the hypotheses of property (wNbC). However, it is straightforward to verify that $H$ does not contain a finite index subgroup that commutes with $K$. Indeed, the required subgroup would be generated by the  $p$th powers of the generators of $H$ and is therefore of infinite index. \emph{However, this is not a counter-example up to virtual isomorphism: taking the quotient by the central subgroup $C$ produces $\G_-\times \G_+$ which does have property (wNbC).}
 
 This group $\G$ is a ``universal"-type cover of the infinite extra-special group obtained by taking the quotient by the normal subgroup $[\G_-,\G_-][\G_+,\G_+]$, and is not finitely generated. Note though that the splitting necessary to deduce Property (wNbC) follows from \cite[Lemma 2.11]{BaderFurmanSauer} when the normal subgroup is in fact finitely generated.  
\end{example}

We begin by establishing the connection between Properties (NbC) and (wNbC) for $\G\to \Aut\X$ acting AU-acylindrically with general type factors. Recall that  AU-acylindricity  is invariant under virtual isomorphism (see Corollary ~\ref{Cor: Class closed finite kernel}) so that by Theorem ~\ref{Thm: Acyl -> CAF} we may assume that $\G$ has trivial amenable radical, and in particular has no finite normal subgroups. This means that Property (NbC) will follow immediately if Property (wNbC) is established. 

Now, if $\{1\}\neq K\norm \G$, $H\leq \G$ is infinite commensurated, and $K\cap H=\{1\}$ then $KH$ must also have general type factors  by Lemma ~\ref{Lem: normal in total gen type is tremb/gen type}.  Therefore, by Corollary ~\ref{Cor: pass to irreducible tremble free core}, up to a virtual isomorphism of $KH$, we may assume that $KH\leq \Aut_{\! 0}\X$ is AU-acylindrical, all the factor actions are essential, general type, and tremble-free. 

Furthermore, since the actions on all of the factors are general type, by Lemma ~\ref{Lem:norm com} (and as in the proof of Corollary ~\ref{cor:prodstrdecomp}), we obtain a decomposition $\{1, \dots, D\}= I_K\sqcup I_H$, where $I_K$ is the nonempty set of indices where the $K$ action is of general type (and $H$ is elliptic), and $I_H$ is the nonempty set (nonempty as in the proof of Theorem ~\ref{Thm: Acyl -> CAF}) of indices where the $H$ action is of general type (and $K$ is trivial). We therefore have that $K \leq \ker(\Aut_{\!0}\X\to \Prod{i\in I_H}{}\Isom X_i)$.

\noindent
\emph{\textbf{Goal:}} Prove that if a factor of $\X$ is of general type for $K$, then it is a tremble for $H$ up to finite index (not just elliptic).

If the claim is proven, then up to finite index $H \leq \ker(\Aut_{\!0}\X\to \Prod{i\in I_K}{}\Isom X_i)$, and so $K$ commutes with a finite index subgroup of $H$, which proves Property (wNbC).

Since an immediate proof of the above goal seems out of reach, we therefore introduce the following hypothesis to prove Property (wNbC). 

\begin{defn}\label{Def: (wNbC)-essential-rift-free}
    An action $\G\to \Isom X$ is said to be \emph{(wNbC)-essential-rift-free} if it is general type and whenever $H\leq \G$ is commensurated and acting elliptically and   $K\norm \G$ acting general type such that $K\cap H= 1$ then the action of $HK$ is essential-rift-free.
\end{defn}

The following example shows that a successful proof that groups in our class have Property (wNbC) must take into consideration all of the hyperbolic factors that yield the acylindrical action.

\begin{example}\label{Ex:Denis}
    Let $F_1=\<g_n: n\in \N\>$ and $F_2=\<h_n: n\in \N\>$ be free groups on the countably many given generators.  Consider the action of $\f: F_2\to \Aut(F_1)$ determined by:
    $$
\f(h_m)(g_n)=
\begin{cases}
g_n^{-1},\text{ if } n=m;\\
g_n,\text{ if } n\neq m.
\end{cases}
$$
   With this in place the semidirect product  $F_1\rtimes_\f F_2$  acts on the  Cayley tree $T_1$ for $F_1$ in the obvious way and factors through $F_1\rtimes \G$, where $\Oplus{n\in \N}{}\Z/2\cong\G= F_2/\ker(\f)\leq \Aut(F_1)$.

   The reader may check the following: \begin{enumerate}
       \item the action of both groups is essential, general type, tremble-free, but not (wNbC)-essential-rift-free;
       \item $F_1$ is normal  in both groups and acting by general type on $T_1$;
       \item $\G$ is commensurated but $F_2$ is not (in their respective ambient groups); 
       \item $F_1\rtimes_\f \F_2$ admits an acylindrical action with general type factors on $T_1\times T_2$ but $F_1\rtimes \G$ does not act in this fashion on any $\X$. Indeed, one can show that the restriction to $\G$ of a general type action of $F_1\rtimes \G$ on any hyperbolic space  will necessarily be rift.
   \end{enumerate}

 In this case, one has to taken into consideration both factors of $T_1\times T_2$, where $T_2$ is the Cayley tree for $F_2$, with $F_1$ acting trivially on $T_2$. 
\end{example}

The following is the necessary adaptation of Lemma ~\ref{Lem: normal in total gen type is tremb/gen type} to this specialized setting.
\begin{lemma}
  Let $\G\to \Isom X$ be  (wNbC)-essential-rift-free. If $H\leq \G$ is commensurated and acts elliptically, $K\norm\G$ acting general type with $H\cap K = \{1\}$ then the action of $H$ on the tremble-free essential core is trivial. 
\end{lemma}

\begin{proof}
    Without loss of generality, assume that $\G= HK$ and the action of $\G$ is essential and tremble free. Let $g\in K$ be loxodromic, and let $H'=gHg^{-1}\cap H$.

Let $L>0$ such that $\O^L(H'):=\{x: d(mx, x)\leq L, h\in H'\}\neq \varnothing$. 
Fix $q: \Z\to X$ be a $(1,20\delta)$ quasigeodesic between the end points of $g$ in $\partial X$. We claim that $q(\Z)$ is at bounded distance from $\O^L(H')$. Otherwise, the closest point projection from $\O^L(H')$ to $q(\Z)$ would be either bounded below or bounded above.  But this is not possible since, as in the proof of Lemma ~\ref{Lem:norm com} $g$ commutes with $H'$, as does the closest point projection. Therefore, Lemma ~\ref{Lem: Contracting QGeod} applies, and thus $\O^L(H')$ is unbounded. Since the action is (wNbC)-essential-rift-free, we deduce that $H'$ and hence $H$ is a tremble. 
\end{proof}

\begin{cor}\label{cor:nbc}
    Let $\G\to \Aut \X$ be AU-ayclindrical with  (wNbC)-essential-rift-free factors. Then $\G$ has Property (NbC).    
\end{cor}

\subsection{Arithmetic Lattices and Tree Extensions}\label{Sect:tree ext}
 
The concept of a tree extension requires the development of a certain amount of background to define. We shall avoid this by not giving the general definition, but rather by explaining an example. We refer the reader to \cite[Section 4.1]{BaderFurmanSauer} for a more in depth discussion.

Recall that the hyperbolic plane $\mathbb H^2$ is the symmetric space for $\PSL_2\R$ and that, for $p$ a prime, the $p+1$ regular tree is the Bruhat-Tits tree $T_{p +1}$ for $\PSL_2\Q_p$ (see for example \cite{Morris} and \cite{Shalen}).  

However, while the isometry group of $\mathbb H^2$ is  $\PSL_2\R$, the same is not true of $\Isom T_{p+1}$. In fact, there are intermediate closed subgroups of infinite index $\PSL_2\Q_p \lneq H\lneq\Isom T_{p+1}$.   

Let $S\subset \N$ be a finite set of primes. And fix such an intermediate closed subgroup $H_p$ for each $p\in S$. Then the diagonal embedding $\PSL_2\Z[S^{-1}]\hookrightarrow\PSL_2\R\times\CProd{p\in S}{}\PSL_2\Q_p$ is an irreducible lattice embedding of the $S$-arithmetic lattice, as is $\PSL_2\Z[S^{-1}]\hookrightarrow\PSL_2\R\times\CProd{p\in S}{}H_p$. The latter product is called an \emph{ $S$-arithmetic tree extension} of the former product.

\section{The Theory is Semi-Simple}\label{Sect:Theory is SS}

At the beginning of this paper, we promised to establish a dictionary (see Table \ref{SS dictionary}) between the world of ($S$-) arithmetic semi-simple lattices and groups acting AU-acylindrically on $\X$ with general type factors. In this section, we shall fill in the details concerning this dictionary.

The fact that an AU-acylindrical action subsumes the type of action a lattice enjoys on its ambient group is established in Lemma \ref{Lem:acyl+loc comp implies unif proper} in complete generality.

Turning to the world of linear algebraic groups over local fields, there are some basic classes of groups therein:  parabolic, Borel, Cartan, central, and (semi-)simple. Together, these provide a general structure theory. Specifically, if $G$ is a linear algebraic group (over an algebraically closed field), then up to virtual isomorphism, $G=R(G)\rtimes \Prod{i=1}{D}G_i$, where $R(G)$ is the solvable radical and $G_i$ is simple for $i = 1,\dots, D$. (We refer the reader to \cite{Humphreys} for details.)

Fix $i$. If  $G_i$  is a rank-1 group then a parabolic subgroup is the stabilizer of some $\xi_i$ in the visual boundary $\partial_\sphericalangle X_i$ of $X_i$, which is the rank-1 symmetric space or Bruhat-Tits tree associated to $G_i$. 

Within this framework, a minimal parabolic subgroup is given by a choice of $\xi_i\in \partial_\sphericalangle X_i$ for each $i=1, \dots, D$, resulting in a maximal solvable subgroup. Furthermore, Cartan subgroups correspond to stabilizers of maximal flats in the product of the corresponding model geometries, and may be realized as intersecting ``opposite" minimal parabolics, i.e. $\stab(\xi, \xi')$ for $\xi_i \neq \xi_i'$, for $i=1, \dots, D$.  Further, parabolic subgroups correspond to a choice of $I\subset D$, and $\xi_i\in \partial_\sphericalangle X_i$ for $i \in I$. Finally, the center acts trivially on these spaces and their boundaries and therefore corresponds to our construction of the tremble-free essential core. This completes the explanation of the dictionary in Table \ref{SS dictionary}. 

We have seen in Theorem ~\ref{Thm: Acyl -> CAF} that a group admitting an  AU-acylindrical action on a product of $\delta$-hyperbolic spaces with general type factors has finite center  and more generally enjoys property (CAF). In Section ~\ref{sec:examples} we discussed how these include ($S$-arithmetic) lattices in semi-simple groups with rank-1 factors and hence we think of our class of groups as a generalization of these and the others discussed there. We shall devote the rest of this section to establishing that these groups admit a strongly canonical product decomposition. Moreover, this descends to a product decomposition of the (outer-)automorphism groups (See also \cite[Section 3.5]{BaderFurmanSauer}.) \ref{SS dictionary}

Our first goal is to prove the following result. 

\begin{theorem}\label{Thm: CanProdDecomp}
Let $\G\to \Aut \X$ be AU-acylindrical and with all factor actions of general type with (finite) amenable radical $\A$. If $\G$ is finitely generated  then there is a characteristic subgroup of finite index $\G'\norm \G/\A$ such that $\G'$ admits a strongly canonical product decomposition. Moreover, the collection of subgroups $\G'$ which are of minimal index in $\G/\A$ is finite.
\end{theorem}

The above result yields a strong consequence for automorphism groups. Indeed, any automorphism of $\G'$ induces another product decomposition, which up to permutation must be the same, and hence descends to an automorphism of the factors. The possibility of a permutation is responsible for the finite index of the inclusions below. We note that in the particular case of products of right-angled Artin groups, similar results are known thanks to the work of Charney-Crisp-Vogtmann \cite[Section 3]{CharneyCrispVogtmann}. 

\begin{cor}\label{Cor: AU to Aut/Out}
Let $\G\to \Aut \X$ be a group acting essentially, AU-acylindrically and with general type factors. There exists a unique $F$, where $1\leq F\leq D$ and $\G'$ which is virtually isomorphic to $\G$ with strongly canonical product decomposition $\G'= \G_1\cdots \G_{F}$, such that the following are finite index inclusions $$ \Prod{i=1}{F}\Aut(\G_i)\norm \Aut(\G') \;\;\; \text{ and } \;\;\;  \Prod{i=1}{F}\Out(\G_i) \norm \Out(\G').$$
\end{cor}

Recall that for a discrete countable group, the \emph{automorphism group}  $\Aut(\G)$ is the group of isomorphisms $\G\to \G$, where the group operation is function composition. Recall also that if $\g\in \G$ then the conjugation map $\g: g\mapsto \g g\g^{-1}$ is an isomorphism. This yields a homomorphism $\G\to \Aut (\G)$, whose kernel is $Z(\G)$ the \emph{center} of $\G$. The image of this homomorphism is denoted  by $\Inn(\G)$, and is called the \emph{inner automorphism group}. It is straightforward to verify that $\Inn(\G) \norm \Aut (\G)$. The quotient group is the \emph{outer automorphism group} and denoted by $\Out(\G)$. A subgroup $H\leq \G$ is said to be \emph{characteristic} if for every $\a\in \Aut(\G)$ we have that $\a(H) = H$.  We begin with the following definition. 

\begin{defn}\label{Defn: StrongCanonProductDecomp}
    An internal product decomposition of $\G= K_1\cdots K_F$ is said to be  \emph{canonical} if each factor is infinite and whenever  $\G = K'_1\cdots K'_{F'}$ is another product decomposition then $F=F'$ and up to permutation $K_i = K'_i$. It is said to be \emph{strongly canonical} if it is canonical and if for any $\G'\leq \G$ of finite index and $\f: \G'\to K'_1\cdots K'_{F'}$ with finite kernel, then $F'=F$ and up to permutation, $K_i'= \f(K_i\cap \G')$. 
\end{defn}

We note that a group that is irreducible (i.e. not isomorphic to a non-trivial direct product) and whose finite index subgroups are also irreducible admits itself as a strongly canonical product decomposition with one factor.

\begin{proof}[Proof of Theorem ~\ref{Thm: CanProdDecomp}] We begin by applying  Theorem ~\ref{Thm: Acyl -> CAF} to deduce that the amenable radical $\A$ is finite. Furthermore, we may apply Corollary ~\ref{Cor: Class closed finite kernel}, and assume that $\A$ is trivial. Without loss of generality, an application of Corollary ~\ref{Cor: pass to irreducible tremble free core}, allows us to deduce that the factor actions are tremble-free and essential.

Recall that if $\G$ is finitely generated then it has finitely many homomorphisms to any particular finite group. Hence, the intersection of these kernels is characteristic as performing an automorphism will not yield a new homomorphism to the chosen finite group (see for example \cite[Lemma 6.7]{FernosRelT}). Taking the finite group in question to be $\Sym_\X(D)$ we obtain the characteristic finite index subgroup $\G_0\norm \G$ that preserves the factors of $\X$. (The reader should note that this could be smaller than the kernel $\G\to \Sym_\X(D)$, which is what we usually refer to as $\G_0$.) 

 Since $\G_0$ and hence all its finite index subgroups have no non-trivial finite normal subgroups, we need only look at finite index subgroups of $\G_0$ for a strongly canonical product decomposition. Of course, these finite index subgroups also inherit all of the hypotheses. 
 
 We will denote the direct factors by $\G_j$. Now, since each $\G_j$ is infinite, and the action $\G\to \Aut \X$ has no kernel, for each $j=1, \dots, F$ there is a non-empty collection of factors $I_j$ on which $\G_j$ acts by a general type action (and the other $\G_j$'s are trembles and hence trivial, as in Corollary ~\ref{cor:prodstrdecomp}). So we may apply Corollary ~\ref{Lem: Bound on number of inf factors}  and guarantee  that the number of direct factors of a finite index subgroup of $\G_0$ is bounded above by $D$. Let $F$ be  the maximal number of factors in   a decomposition of a finite index subgroup of $\G_0$. This gives us a finite index subgroup $\G' = \G_1 \cdots \G_F$. We claim that $\G'\leq \G_0$ can be chosen such that

\begin{enumerate}[nolistsep]
\item $\G'=\G_1\cdots \G_F$;
\item $\G'$ is characteristic;
\item $\G'$ has minimal finite index in $\G_0$ among finite index subgroups satisfying properties (1) and (2).
\end{enumerate}

To justify that $\G'$ can be made characteristic, first note that if a finite index subgroup satisfied (1)  alone, then the intersection of all its $\G$-conjugates would again be finite index, normal, and satisfying (1). Then intersecting all possible kernels to the associated finite quotient as above, using the finite generation of $\G$ will produce a characteristic subgroup. The same argument also shows that there are finitely many characteristic subgroups of $\G$ of the same minimal index as $\G'$, and there are therefore finitely many subgroups satisfying these properties. 

Observe that since $F$ is maximal, and $\G'$ has trivial amenable radical, $\G_i$ is infinite and strongly irreducible, for each $j=1, \dots, F$. We now show that $\G'= \G_1 \cdots \G_F$ is a strongly canonical product decomposition. To this end, let $\G''\leq \G'$ be of finite index and let $\G''=K_1\cdots K_{F'}$ be an internal product decomposition. Since $F$ is maximal, and $\G''$ has trivial amenable radical we must have that $F'=F$, and $K_j$ are infinite and strongly irreducible for $j= 1,\dots, F$.

Let $\G'_j= \G_j\cap \G''$ and note that this yields another internal product decomposition of infinitely strongly irreducible groups $\G''= \G'_1\cdots\G'_F$.

We introduce the following notation: For nonempty $\J\subsetneq \{1, \dots, F\}$ let 
$\G'_\J\norm \G''$ (respectively $ K_\J\norm \G''$) be the product of the factors $\G'_j$, (respectively $K_j$) for $j\in \J$ and $\^\G'_j\norm \G''$ (respectively $ \^K_\J\norm \G''$) be the product of the remaining factors so that $\G''= \G'_\J\^\G'_\J=K_\J\^K_\J$ are  internal product decompositions. Also let $\pi_\J: \G''\to \G'_\J$, $\^\pi_\J: \G''\to \^\G'_\J$ be the natural projections. Note that for each $j= 1, \dots, F$ we have that 
$$\Cap{\J \mid j\notin\J}{}\ker(\pi_\J)= \G'_j.$$

 Let $\J= \{1\}$. 
Then $\pi_\J(K_1 \cdots \K_F)=\G'_\J$. Since $\G'_\J$ is strongly irreducible, we must have that all but exactly one of the factors $\{K_1, \dots, K_F\}$  is not in the kernel of $\pi_\J$ and all others are.  Up to permutation, we assume that $K_1\not\leq \ker(\pi_\J)$ and $$\^K_\J\leq  \ker(\pi_\J)= \^\G'_\J.$$

Furthermore, since there are the same number of factors in both product decompositions for $\G''$ by the pigeon-hole principal and the above argument, up to permutation, the same conclusion holds where $\J$ runs over all singletons in $\{1, \dots, F\}$. This means that for each $j=1, \dots, F$ we have that 

$$K_j\leq \Cap{j\notin \J}{}\ker(\pi_\J) = \G'_j.$$

Reversing the roles of the two product decompositions  we obtain that 
$\G'_j\leq K_j$, which concludes the proof.
\end{proof}

\begin{remark}\label{Rem: Out to MCG} We note that if $\G$ has trivial amenable radical (e.g. if it is torsion-free)  and does not permute factors (i.e. $\G\to \Aut_{\! 0}\X$) then the proof shows that $\G'=\G$ in Theorem ~\ref{Thm: CanProdDecomp} and Corollary ~\ref{Cor: AU to Aut/Out}. 
\end{remark}

\section{On a Conjecture by Sela}\label{sec:ripsandsela}

In this last section, we examine the implications of our work in the setting of Sela's recent  work. In a lecture in 2023, Sela spoke about ongoing  work developing Makanin-Razborov diagrams in higher rank (see \cite{SelaHR1Pub, SelaHR2}). These were used by Sela in the rank-1 setting to understand homomorphisms between groups, and in particular automorphisms of groups, leading to the following: 

\begin{theorem}\label{Thm:OutHyp1End}\cite{Sela1997,Levitt2005} Let $\G$ be a one-ended hyperbolic group. Then, $\Out(\G)$ contains a finite index subgroup $\Out'(\G)$ that fits into the following central short exact sequence, with $|F|<\8$:

$$1\to \Z^n\to \Out'(\G)\to \CProd{s\in F}{} \mathrm{PMCG(\Sigma_s)}\to 1. $$
    
\end{theorem}

 In Sela's lecture, he attributed the following question to E. Rips as inspiration:

\begin{question}[Rips]
    What can we say about $\Aut(\pi_1(K))$ or $\Out(\pi_1(K))$ when $K$ is a  finite locally CAT(0) cube complex? 
\end{question}

Fioravanti has made significant progress in this direction by establishing necessary and sufficient conditions for when $\Out(\pi_1(K))$ is infinite \cite{FioravantiOut}. (See also \cite{CasalsKazachkov}; Casals-Hagen-Kazachkov have work in progress in a related direction.)

Since the fundamental groups of virtually special non-positively curved cube complexes are known to be HHGs, this led Sela to make a conjecture (which we have clarified via personal communication into the form below). Loosely, it states that the outer-automorphism group of a ``nice" HHG is very similar to the structure as in Theorem ~\ref{Thm:OutHyp1End}.

\begin{conj}[Sela]\label{Conj: Sela}
Let $\G$ be a colorable HHG. Suppose that either:

\begin{enumerate}
    \item The action of $\G$ on the quasitree of metric spaces (that can be associated with it using the construction of Bestvina-Bromberg-Fujiwara) is weakly acylindrical.
$$\text{or}$$
    \item The action of $\G$ on the projection complex (of Bestvina-Bromberg-Fujiwara) is weakly acylindrical and the construction of each domain (set) stabilizer of the corresponding domain is weakly acylindrical.
\end{enumerate}

Suppose further that the associated higher rank Makanin-Razborov diagram is single ended (i.e., contains no splittings along finite groups along it).

Then there exists a finite index subgroup $\Out'(\G) \leq  Out(\G)$ and a map $\eta:Out'(\G) \to B$, where $B$ is the direct product of mapping class groups of 2-orbifolds and outer automorphism
groups of finitely genereated virtually abelian groups, that appear in the higher rank JSJ decomposition of $G$. Furthermore, $\ker(\eta)$ is finitely generated virtually abelian.
\end{conj}

\begin{remark}\label{Rem: Colourable HHG}  
    By Theorem 3.1 \cite{HagenPetyt}, a colorable HHG quasi-isometrically embeds into a finite product of hyperbolic spaces, which one could call Property (Q$\X$). 
\end{remark}

 Sela has made progress towards resolving this conjecture by obtaining a short exact sequence with the quotient in the desired class, but whose kernel is not virtually abelian. 
 In a personal communication he elaborated on how he has succeeded in constructing a map $\nu: \Out'(\G) \to B$, where $\Out'(\G)\leq \Out(\G)$ is of finite index and $B$ is a direct product of groups in the desired class. But the 2-orbifolds obtained here are not the ones that appear in the higher rank JSJ decomposition. Instead, they can be obtained from the ones in the JSJ decomposition by attaching disks to their boundaries, and hence the mapping class groups are smaller. Further, the kernel of this map cannot be finitely generated virtually abelian in general \cite{SelaHR2}. However, in the special case of right-angled Artin groups, Charney and Vogtmann have foundational results that affirm the Conjecture ~\ref{Conj: Sela}; see \cite[Corollary 3.3 and Theorem 4.1]{CharneyVogtmann}.

\subsection{A Partial Resolution}
We give a partial resolution to Sela's conjecture in the form of Corollary ~\ref{Cor: Partial Resol Sela Conj}.  The context of Sela's work  in \cite{SelaHR1Pub, SelaHR2} allows for lineal factors, but we will restrict our attention to the case where the actions on the factors of $\X$ are of general type. Along the way, we also clarify the meaning of his weak and strong  notions of acylindricity (see Definitions ~\ref{Def: weak acyl} and ~\ref{def:Sela Strong Acyl}).  

\begin{theorem}\label{Thm: Strong acyl HHG product}
    Let $\G\to \Aut \X$ be strongly acylindrical with general type factors. Up to  virtual isomorphism, $\G=\CProd{i=1}{D}{}\G_i$ is the direct product of hyperbolic groups, where $D$ is the number of factors of $\X$.  
\end{theorem}

We will require the following definition for the next results. 
\begin{defn}
    Given groups $\G_1,\dots, \G_D$, a subdirect product of $\G_1$ and $\G_2$ is any subgroup $\G \leq \Prod{i=1}{D}\G_i$ which maps surjectively onto each factor. 
    
    A subdirect product is called \emph{full} if $\G \cap \G_i$ is non-trivial for $i =1,\dots, D$. 
\end{defn}

\begin{theorem}\label{Thm: Weak Acyl is Product}
    Let $\G\to \Aut \X$ be acylindrical and cobounded with general type factors  such that the projection of $\G_0\to \Isom X_i$ is  weakly acylindrical for all $i=1, \dots, D$. Then $\G$ is virtually isomorphic to a subdirect product of $\CProd{i=1}{D}{}\G_i$ where each $\G_i$ is acylindrically hyperbolic and $D$ is the number of factors of $\X$. 
\end{theorem}

\begin{cor}\label{Cor: Weak Acyl HHG Product}
    Let $\G$ be an HHG with no eyrie that is a quasi-line and whose eyrie stabilizers act weakly acylindrically. Then $\G$ is virtually isomorphic to a  subdirect product of $D$-many acylindrically hyperbolic groups, where $D$ is the number of eyries.  
\end{cor}

\begin{remark}\label{Rem:Eyries}
    Key to passing from a general HHG to a group in our setting is \cite[Remark 4.9]{PetytSpriano} which states that an HHG acts \emph{acylindrically and coboundedly} on the product of the maximal unbounded domains. A maximal unbounded domain is called an \emph{eyrie}. 

    In a personal communication, Petyt and Spriano elaborated on their remark: to prove that an HHG acts acylindrically on the product of eyries, one should replace instances of the maximal domain in the HHG strucutre with the product of the eyries in the proof of \cite[Theorem K]{BHS1}. For coboundedness, without loss of generality, we may assume that $\G$ is preserving the factors.  Suppose we have eyries $U_1,\dots,U_D$, and projections $p_i$ to $U_i$. Let $x,y\in \Prod{i=1}{D} U_i$. The partial realisation axiom says that there exist $g,h \in \G$ such that $p_i(g)$ is close to  $x_i$ and $p_i(h)$ is close to $y_i$ for all $i= 1, \dots, D$. 
    Then $hg^{-1}$ will send $x$ close to $y$. 
\end{remark} 

Applying Corollary ~\ref{Cor: AU to Aut/Out} we immediately deduce the following and note that the colorability asumption has not been used:

\begin{cor}\label{Cor: Partial Resol Sela Conj}
    An HHG none of whose eyries are quasi-lines satisfies the conclusion of Sela's Conjecture if and only if the irreducible factors in its canonical product decompositions satisfy Sela's Conjecture. 
\end{cor}

\subsection{On Sela's Weak and Strong Acylindricities}

This section is dedicated to reframing Sela's notions of weak and strong acylindricity within our context. As discussed in the introduction, we highlight to the reader that this notion is distinct from other notions of weak acylindricity such as those found in \cite{Hamenstaedt} and \cite{Genevois}. 
\begin{defn}\cite[Definition 2.1]{SelaHR1Pub}\label{Def: weak acyl}
    The action $\G\to \Isom Y$ is \emph{weakly acylindrical} if there exists $r>0$ such that for every $\e>0$ there exist $R, N>0$, so that for every $x,y \in Y$ with $d(x,y)>R$ there exists at most $N$
 elements $g_1, \dots, g_k \in \G$ ($k\leq N$), so that if $g\in \G$ with $d(gx,x)<\e$ and $d(gy,y)<\e$, then $g= g_ju$ for some $1\leq j \leq k,$ and some $u \in \G$ such that  $d(uz, z)<r$ for every $z\in Y$. 
 \end{defn}

Our first contribution here is to simplify this via the use of trembles in the context of $\delta$-hyperbolic spaces. 

Recall that given an essential general type action $\rho: \G\to \Isom X$ the trembling radical is $\T_\rho(\G) = \ker (\G\to \Homeo (\partial X))$ and moreover, the essential general type action is called tremble-free if $\T_\rho(\G) = \ker (\rho)$.

 \begin{defn}   
  We shall say an  action $\rho:\G\to \Isom X$ is 
  \begin{itemize}
      \item \emph{acylindrical modulo the kernel} if for every $\e>0$ there exist $R, N>0$, so that for every $x,y \in X$ with $d(x,y)>R$ we have that the cardinality of the cosets
   $$|\{g\ker (\rho)\in \G/\ker (\rho): d(gx, x)\leq \e; d(gy, y)\leq \e\}|\leq N.$$
      \item \emph{acylindrical modulo trembles} if for every $\e>0$ there exist $R, N>0$, so that for every $x,y \in X$ with $d(x,y)>R$ we have that the cardinality of the cosets
   $$|\{g\T_\rho(\G)\in \G/\T_\rho(\G): d(gx, x)\leq \e; d(gy, y)\leq \e\}|\leq N.$$
  \end{itemize}
  
 \end{defn}

The following is immediate:

 \begin{cor}\label{Cor: Sela weakly acyl iff acyl mod trembles}
     An essential general type action $\rho:\G\to \Isom X$ is weakly acylindrical if and only if it is acylindrical modulo trembles on the essential core $X'$. In particular,  $\rho':\G/\T(\rho(\G))\to \Isom X'$ is acylindrical.
 \end{cor}

We note that the following can be found under the name of \emph{strongly acylindrical} in the pre-print \cite{SelaHR1} and as weakly acylindrical in the published version \cite{SelaHR1Pub}, though there is a difference in the case of factors that are quasi-lines, which we don't consider here. (This is why we have included both citations.) As weak acylindricity is already defined for a metric space as above, and we will not consider the case of quasi-lines, we have decided to keep the term ``strongly acylindrical". Moreover, Theorem ~\ref{Thm: Strong acyl HHG product} shows that it is indeed a strong form of acylindricity. 

Recall our notation that for an action $\G\to \Aut \X$, we denote by $\G_0$ the finite index subgroup whose image is in $\Aut_{\!0} \X$.

\begin{defn}\label{def:Sela Strong Acyl}
    An action  $\G\to \Aut \X$ is said to be \emph{strongly acylindrical} if it is proper cocompact and for each $i\in \{1, \dots, D\}$, the action of $\G_0$ on $X_i$ is weakly acylindrical.
\end{defn}

\begin{proof}[Proof of Theorem ~\ref{Thm: Strong acyl HHG product}]

Let $\G\to \Aut \X$ be strongly acylindrical. Without loss of generality, we may assume that $\G$ does not permute factors. Moreover, up to taking the quotient by a finite kernel since the action is proper-cocompact (and hence acylindrical), we may assume by Corollary ~\ref{Cor: pass to irreducible tremble free core} that $\G\leq \Aut_{\!0} \X$ is strongly acylindrical with essential, general type and tremble-free factors. 

This means that the action is  proper cocompact, with factor actions  cobounded and acylindrical modulo the corresponding trembles. Note that cobounded general type actions are necessarily essential. Also note that since $\G$ is proper co-compact in $\Prod{i=1}{D} \Isom X_i$ it follows that $X_i$ is locally compact. 

Let $\G_i\leq \Isom X_i$ be the image of the projection.  
Since $X_i$ is locally compact, by Lemma ~\ref{Lem:acyl+loc comp implies unif proper}, $\G_i$ is  in fact acting uniformly properly, coboundedly and is therefore, a hyperbolic group.  We have  established that $\G\leq \CProd{i=1}{D}\G_i$ where each factor is a hyperbolic group. This must be a finite index inclusion since both actions of $\G$ and $\CProd{i=1}{D}\G_i$ are proper cocompact on $\X$.
\end{proof}

\begin{proof}[Proof of Theorem ~\ref{Thm: Weak Acyl is Product}]

Without loss of generality, we may assume that $\G$ does not permute factors. Moreover, by Corollary ~\ref{Cor: pass to irreducible tremble free core}, we may assume that $\G\leq \Aut_{\!0} \X$ is acylindrical with essential, general type and tremble-free factors. 

Let $\G_i$ be the projection of $\G\to \Isom X_i$. By Corollary ~\ref{Cor: Sela weakly acyl iff acyl mod trembles}, the action of $\G_i$ is acylindrical and general type, and therefore, $\G_i$ is acylindrically hyperbolic. Therefore $\G\leq \Prod{i=1}{D}\G_i$. Since $\G$ is surjective onto each $\G_i$, this means $\G$ is a subdirect product of $\Prod{i=1}{D}\G_i$.
\end{proof}

\bibliographystyle{alpha}
\bibliography{AHR}

\end{document}